\documentclass[11pt,letterpaper]{amsart}

\usepackage[utf8]{inputenc}
\usepackage{microtype}
\usepackage{amsfonts}
\usepackage{amsmath}
\usepackage{graphicx}
\usepackage[dvipsnames]{xcolor}
\usepackage{hyperref}
\usepackage[noabbrev,nameinlink,capitalise]{cleveref}
        \crefname{subsection}{Subsection}{Subsections}
\usepackage{tikz-cd}
\usepackage[normalem]{ulem} 

\usepackage{amssymb}
\usepackage{enumitem}
\usepackage{mathrsfs}

\usepackage{tikz}
\usetikzlibrary{topaths}
\usetikzlibrary{calc}

\usepackage{xpatch}
\patchcmd{\subsection}{\normalfont}{\boldmath}{}{}

\usepackage{empheq}

\usepackage[all]{xy}   

\newtheorem{theorem}{Theorem}[section]
\newtheorem{proposition}[theorem]{Proposition}
\newtheorem{lemma}[theorem]{Lemma}

\newtheorem{corollary}[theorem]{Corollary}

\newtheorem*{bhconj}{Boone--Higman conjecture}
\newtheorem*{main:bh_hyp}{Theorem~A}

\theoremstyle{definition}
\newtheorem{definition}[theorem]{Definition}
\newtheorem{question}[theorem]{Question}
\newtheorem{remark}[theorem]{Remark}
\newtheorem{example}[theorem]{Example}

\newcommand{\newword}[1]{\textbf{#1}}
\newcommand{\Z}{\mathbb{Z}}
\newcommand{\N}{\mathbb{N}}
\newcommand{\R}{\mathcal{R}}
\newcommand{\Nuc}{\mathcal{N}}
\newcommand{\Rel}{\mathrm{Rel}}
\newcommand{\C}{\mathfrak{C}}
\newcommand{\A}{\mathcal{A}}
\newcommand{\Homeo}{\mathrm{Homeo}}
\newcommand{\Classes}{\mathrm{Classes}}
\newcommand{\Fix}{\mathrm{Fix}}
\newcommand{\Stab}{\mathrm{Stab}}
\newcommand{\Aut}{\mathrm{Aut}}
\newcommand{\F}{\mathrm{F}}
\newcommand{\GL}{\mathrm{GL}}
\newcommand{\class}{\mathrm{class}}
\newcommand{\Cones}{\mathrm{Cones}}

\newcommand{\oF}{\mkern 3.75mu\overline{\mkern-3.75mu F}}
\newcommand{\od}{\mkern 3mu\overline{\mkern-3mu d\mkern 1.5mu}\mkern-1.5mu}
\newcommand{\of}{\mkern 3.5mu\overline{\mkern-3.5mu f\mkern0.75mu}\mkern-0.75mu}
\newcommand{\og}{\mkern 2mu\overline{\mkern-2mu g\mkern0.5mu}\mkern-0.5mu}
\newcommand{\oh}{\mkern 2mu\overline{\mkern-2mu h\mkern-0.25mu}\mkern 0.25mu}

\usepackage{scalerel}
\newlength\bshft
\def\fakebold#1{\ThisStyle{\ooalign{$\SavedStyle#1$\cr%
  \kern-\bshft$\SavedStyle#1$\cr%
  \kern\bshft$\SavedStyle#1$}}}

\newcommand{\listlabel}[1]{%
  \item[\quad(#1)]\hypertarget{#1}{}\global\expandafter\def\csname #1\endcsname{\hyperlink{#1}{(#1)}}}

\begin{document}

\title{Hyperbolic groups satisfy the Boone--Higman conjecture}

\date{\today}


\author[J.~Belk]{James Belk}
\address{University of Glasgow, Glasgow, Scotland}
\email{jim.belk@glasgow.ac.uk}

\author[C.~Bleak]{Collin Bleak}
\address{University of St Andrews, St Andrews, Scotland}
\email{collin.bleak@st-andrews.ac.uk}

\author[F.~Matucci]{Francesco Matucci}
\address{Universit\`{a} di Milano--Bicocca, Milan, Italy}
\email{francesco.matucci@unimib.it}

\author[M.~C.~B.~Zaremsky]{Matthew C.~B.~Zaremsky}
\address{Department of Mathematics and Statistics, University at Albany (SUNY), Albany, NY}
\email{mzaremsky@albany.edu}

\begin{abstract}
The 1973 Boone--Higman conjecture predicts that every finitely generated group with solvable word problem embeds in a finitely presented simple group. In this paper, we show that hyperbolic groups satisfy this conjecture, that is, each hyperbolic group embeds in some finitely presented simple group. This shows that the conjecture holds in the ``generic'' case for finitely presented groups. Our key tool is a new family of groups, which we call \emph{rational similarity groups (RSGs)}, that is interesting in its own right. We prove that every hyperbolic group embeds in a full, contracting RSG, and every full, contracting RSG embeds in a finitely presented simple group, thus establishing the result. Another consequence of our work is that all contracting self-similar groups satisfy the Boone--Higman conjecture.
\end{abstract}

\maketitle
\tableofcontents
\thispagestyle{empty}

\section{Introduction}\label{sec:intro}

The Boone--Higman conjecture, posed by William Boone and Graham Higman in 1973 \cite{boone73,boonehigman}, predicts the following group theoretic equivalent condition to a finitely generated group having solvable word problem:

\begin{bhconj}
A finitely generated group has solvable word problem if and only if it embeds as a subgroup of some finitely presented simple group.
\end{bhconj}

Recall that a group has \newword{solvable word problem} if there exists an algorithm to determine whether a given word in the generating set represents the identity element of the group. Historically, the word problem for groups was one of the foundational issues in the development of combinatorial and geometric group theory, being first studied by Max Dehn in the early 20th century, along with the conjugacy and isomorphism problems; see \cite{dehn12}. A primary motivation for the Boone--Higman conjecture is that it would establish a very straightforward, purely group theoretic, equivalent condition for finitely generated groups to have solvable word problem, namely, being a subgroup of a finitely presented simple group. We will often refer to an injective homomorphism from a group to a finitely presented simple group as a \newword{Boone--Higman embedding}.

The ``if'' direction of the Boone--Higman conjecture is straightforward. First, any finitely generated subgroup of a finitely generated group with solvable word problem itself has solvable word problem. Moreoever, finitely presented simple groups are easily seen to have solvable word problem, as was first observed by Alexander Kuznetsov in 1958 \cite{kuznetsov}. The ``only if'' direction, however, has been open for fifty years. Boone and Higman did prove a weaker version of this direction, namely, every finitely presented group with solvable word problem embeds in a simple subgroup of a finitely presented group \cite{boonehigman}. This was improved by Thompson to include that the simple subgroup can be taken to be finitely generated \cite{thompson80}. Also, it suffices to prove the conjecture in the finitely presented case, since every finitely generated group with solvable word problem embeds in a finitely presented group with solvable word problem \cite{clapham}. The Boone--Higman conjecture is a relative of the famous Higman embedding theorem, which states that a finitely generated group is computably presented if and only if it embeds in a finitely presented group \cite{higman61}. Both the Boone--Higman conjecture and the Higman embedding theorem are thus of the form ``a finitely generated group has a certain algorithmic property if and only if it embeds in a certain kind of group.'' The survey \cite{bbmz_survey} contains more history and background around the Boone--Higman conjecture.

Perhaps the most robust partial result toward the Boone--Higman conjecture, historically speaking, is that the groups $\GL_n(\Z)$ (which have solvable word problem) satisfy the conjecture -- this follows from work of Elizabeth Scott \cite{scott84}. Her proof involves (what would in modern terminology be called) finding a faithful self-similar action of $\Z^n\rtimes \GL_n(\Z)$ on an appropriate tree and then taking what is now known as the R\"over--Nekrashevych group. As a consequence, all groups embeddable into some $\GL_n(\Z)$ also satisfy the conjecture, such as right-angled Artin groups, Coxeter groups, and virtually nilpotent groups. See \cite{bbmz_survey} for a more comprehensive list of prominent groups that are known to satisfy the Boone--Higman conjecture. To be clear, whenever we say that a finitely generated group, ``satisfies the Boone--Higman conjecture,'' we are implicitly asserting that the group is known to have solvable word problem, and that moreover it has been proven to embed in a finitely presented simple group.

\subsection{Statement of results and background}\label{ssec:statement}

The main result of this paper is the following, which establishes that hyperbolic groups satisfy the Boone--Higman conjecture.

{\renewcommand*{\thetheorem}{\Alph{theorem}}
\begin{theorem}[Boone--Higman for hyperbolic groups]\label{thrm:bh_hyp}
Every hyperbolic group embeds as a subgroup of a finitely presented simple group.
\end{theorem}}

Recall that a geodesic metric space is \newword{hyperbolic} if there exists $\delta>0$ such that for any triangle of geodesic paths, each side lies in the $\delta$-neighborhood of the other two. A finitely generated group $G$ is called \newword{hyperbolic} if some (equivalently any) Cayley graph of $G$ with respect to a finite generating set is hyperbolic under the path metric. Hyperbolic groups are among the most prominent objects of study in geometric group theory. They were first introduced by Gromov in \cite{Gromov1987}, although in some sense they were already present in Dehn's early 20th century work, since it turns out a group is hyperbolic if and only if it admits a so-called Dehn presentation \cite[Theorem~III.$\Gamma$.2.6]{BrHa}. Hyperbolic groups are always finitely presented, and indeed have type~$\F_{\!\infty}$; see Proposition~III.$\Gamma$.2.2 and Corollary~III.$\Gamma$.2.26 of \cite{BrHa}. Recall that a group has type~$\F_n$ if it has a classifying space with finite $n$-skeleton, and type~$\F_{\!\infty}$ means type~$\F_n$ for all $n$. Since coming to prominence thanks to Gromov's work, hyperbolic groups have inspired an incredible amount of research. Prominent examples of hyperbolic groups include free groups, free products of finite groups, groups acting properly discontinuously cocompactly by isometries on a hyperbolic metric space (e.g.\ fundamental groups of compact hyperbolic manifolds), many Coxeter groups,  and groups satisfying certain small cancellation conditions. The class of hyperbolic groups is also closed under taking free products. Prominent examples of groups that are not hyperbolic include $\Z^2$, Baumslag--Solitar groups $BS(m,n)$, and groups containing any of these as subgroups. See, e.g.\ \cite[Chapter~III]{BrHa} for an overview of hyperbolic groups.

Hyperbolic groups are characterized by having word problem solvable in linear time (see, e.g.\ \cite[Section~III.$\Gamma$]{BrHa}), and so are a natural test case for the Boone--Higman conjecture. Moreover, they are ubiquitous among finitely presented groups; indeed, a ``random'' finitely presented group is hyperbolic with probability approaching $1$ \cite{Champetier,Ol'shanskiiYu,Gromov1993}. Our result therefore establishes that the Boone--Higman conjecture holds in the ``generic'' case for finitely presented groups. While several examples of hyperbolic groups were previously known to embed into finitely presented simple groups, such as all virtually special hyperbolic groups, to the best of our knowledge this result is new outside these examples. In particular it is new for any non-linear hyperbolic groups. Non-linear hyperbolic groups exist, for example the first construction was given by Michael Kapovich in \cite[Section~8]{kapovich05}, and see \cite{canary19} for some more recent examples.

\medskip

The finitely presented groups that we use to find Boone--Higman embeddings of hyperbolic groups come from the extended family of ``Thompson-like'' groups. This is a loosely defined class of groups that generalize the classical Thompson groups $F$, $T$, and $V$, introduced by Richard Thompson; see \cite{CFP} for background on Thompson's groups. Groups that are considered Thompson-like often have a combination of good finiteness properties like finite presentability together with good simplicity properties like having simple commutator subgroup, and so are good candidates for targets of Boone--Higman embeddings. In this paper, the finitely presented simple Thompson-like groups that we use are so called twisted Brin--Thompson groups, introduced by the first and fourth authors in \cite{BelkZaremskyTwisted} following the construction of the Brin--Thompson groups by Matthew Brin in \cite{brin04}. Thanks to a criterion due to the fourth author in \cite{zaremskyTaste} (see \cref{thrm:action_to_simple}), we do not need to work directly with twisted Brin--Thompson groups here, but rather embed hyperbolic groups into a new family of finitely presented groups that (may or may not be simple but) admit certain actions, ensuring their embeddability into finitely presented simple twisted Brin--Thompson groups.

This new family of groups is the main focus of this paper. We call them ``rational similarity groups'', or ``RSGs'' for short. They arise as natural subgroups of the rational group of a subshift of finite type, and many existing examples of interesting groups turn out to be RSGs. Besides consisting of elements that act rationally on the subshift, the key property of an RSG is that it can realize every ``canonical similarity'' between compatible cones in the subshift. See \cref{def:rsg} for the formal definition of an RSG. An important family of examples of RSGs is the R\"over--Nekrashevych groups of self-similar groups, as in \cref{ex:rn}, and to some extent RSGs can be viewed as an ``asynchronous'' analog of R\"over--Nekrashevych groups. However, we should emphasize that the R\"over--Nekrashevych groups are only a very special case of RSGs, as will be made clear at various points.

\subsection{Outline of the proof}\label{ssec:outline}

Let us give a few more details of how we prove \cref{thrm:bh_hyp}. One of the most difficult steps is the following finite presentability result, which is \cref{thrm:fin_pres}.

{\renewcommand*{\thetheorem}{\Alph{theorem}}
\begin{theorem}\label{thrm:FPTheorem}
Every full, contracting RSG is finitely presented.
\end{theorem}}

See \cref{def:full} and \cref{def:contracting} for the definitions of full and contracting. To prove \cref{thrm:FPTheorem}, we write down an explicit infinite presentation, and use a family of words that we call ``normalish forms'' (\cref{def:normalish}) to prove that it presents the group. Then we reduce the presentation to an explicit finite one. Crucial to all of this is that a certain group $V_{\Gamma,E}$ is finitely presented (\cref{cor:V_fp}); this is already interesting in its own right, and generalizes work of Matui from \cite{Matui15}.

Combining \cref{thrm:FPTheorem} with the aforementioned \cref{thrm:action_to_simple}, we conclude:

{\renewcommand*{\thetheorem}{\Alph{theorem}}
\begin{theorem}\label{thrm:rsg_to_simple}
Every contracting RSG embeds into a finitely presented simple group, and hence satisfies the Boone--Higman conjecture.
\end{theorem}}

One consequence of this is the following, which is \cref{cor:RN_BH}.

{\renewcommand*{\thetheorem}{\Alph{theorem}}
\begin{corollary}\label{cor:ss_embed}
Every contracting self-similar group embeds into a finitely presented simple group, and hence satisfies the Boone--Higman conjecture.
\end{corollary}}

Finally, in order to prove \cref{thrm:bh_hyp}, it remains to prove the following, which is \cref{thrm:hyp_to_contracting}.

{\renewcommand*{\thetheorem}{\Alph{theorem}}
\begin{theorem}\label{thrm:hyp_embed}
Every hyperbolic group embeds in a full, contracting RSG.
\end{theorem}}

We prove \cref{thrm:hyp_embed} by using a group denoted $[[\,G\mid \partial_h G\,]]$, for $G$ a given hyperbolic group. This group arises from the action of $G$ on its horofunction boundary $\partial_h G$, as defined by Gromov.  The horofunction boundary of a hyperbolic group is a totally disconnected, compact metrizable space that maps onto the more familiar Gromov boundary $\partial G$ by a finite-to-one map \cite{WeWi,perego2023rationality}. The group $[[\,G\mid\partial_h G\,]]$ consists of all homeomorphisms of $\partial_h G$ that locally agree with elements of the hyperbolic group.  That is, a homeomorphism $f\colon \partial_h G\to \partial_h G$ lies in $[[\,G\mid \partial_h G\,]]$ if and only if there exists a partition of $\partial_h G$ into finitely many clopen sets such that $f$ agrees with some element of $G$ on each set of the partition.

These groups $[[\,G\mid \partial_h G\,]]$ are ``Thompson-like'' in the sense that their action shares many properties with the usual action of Thompson's group $V$ or a R\"over--Nekrashevych group on the Cantor set. Indeed, in the framework established by Matui and Matsumoto \cite{Matui12,Matui13,MatuiMatsumoto14,Matsumoto15,Matui15,MatuiMatsumoto17}, the groups $[[\,G\mid \partial_h G\,]]$ can be viewed as topological full groups of certain \'etale groupoids.  For most hyperbolic groups $G$, this groupoid is purely infinite and minimal.  Then $[[\,G\mid \partial_h G\,]]$ is the topological full group associated to the \'etale groupoid of all germs of elements of $G$ acting on $\partial_h G$. Note that the construction of $[[\,G\mid X\,]]$ makes sense whenever $G$ is a group acting on a topological space $X$, though the resulting group is not always Thompson-like.  For example, if the action of $G$ on $X$ preserves some measure, then the action of $[[\,G\mid X\,]]$ on $X$ will too.

In \cite{BelkBleakMatucci}, the first three authors described an assignment of addresses to $\partial_h G$ with respect to which the action of $G$, for most hyperbolic groups $G$, is faithful and asynchronous rational, as defined by Grigorchuk, Nekrashevych, and Sushchanski\u{\i} \cite{GNS}.  Specifically, $\partial_h G$ can be viewed as a clopen subset of a subshift of finite type, and $G$ acts by finite-state transducers. It follows that the action of $[[\,G\mid \partial_h G\,]]$ on $\partial_h G$ is faithful and rational. The group $[[\,G\mid \partial_h G\,]]$ also contains the topological full group $V_{\Gamma,E}$  for the associated clopen subset $E$ of a subshift $\Sigma_\Gamma$, which ensures that it is an RSG. Finally, we prove that it is contracting; the proof of this involves some intricate geometric arguments, and is the part of the overall argument that utilizes hyperbolicity most essentially; see \cref{lem:contractinglemma} for the culmination of these technical arguments.

All of these steps work for hyperbolic groups $G$ outside some low-complexity cases; in particular they work whenever $G$ has a proper $\Z$ free factor, and so every hyperbolic group embeds into a hyperbolic group $G$ for which $[[\,G\mid \partial_h G\,]]$ is a full, contracting RSG, thus establishing \cref{thrm:hyp_embed}.

\subsection{Open questions}\label{ssec:open}

We end the introduction with some questions that naturally arise.

\setcounter{theorem}{0}
\begin{question}
Is every non-elementary hyperbolic group isomorphic to a contracting RSG?\end{question}

We show in \cref{sec:hyp} that every hyperbolic group $G$ that acts faithfully on its Gromov boundary is isomorphic to an RSG. Moreover, if $G$ has $\Z$ as a proper free factor then the RSG is contracting. The main impediment to obtaining the contracting property outside this case is that we do not know whether $\Sigma_\Gamma$ always has an irreducible core.

\begin{question}\label{quest:F_infty}
Do full, contracting RSGs have type $\F_{\!\infty}$? Even for R\"over--Nekrashevych groups of contracting self-similar groups, this remains open, and was conjectured to be true by Nekrashevych; see \cite{NekraFP}.
\end{question}

A first step toward answering \cref{quest:F_infty} would be finding a topological proof of finite presentability of full, contracting RSGs. An attempt that the authors made, following a variant of the ``standard'' approach to deducing type $\F_{\!\infty}$ for Thompson-like groups (see, e.g.\ \cite{BelkMatucci,ZaremskyF}), was unsuccessful, so it seems that some new ideas are needed.

One could also ask whether finite presentability of the full group of a hyperbolic group $G$ acting on its horofunction boundary could be deduced in some manner making direct use of the finite presentability of $G$, rather than the contracting property. More generally, one can ask:

\begin{question}\label{quest:all_fp_subgroups_of_rat}
Does every finitely presented subgroup of the rational group $\R_{\Gamma,E}$ embed in a finitely presented subgroup whose action on some orbit in $E$ is oligomorphic and has finitely generated stabilizers of finite subsets? (And hence embed in a finitely presented (simple) twisted Brin--Thompson group?)
\end{question}

Scott in \cite[Theorem~2]{scott84} showed that (what is now called the) R\"over--Nekrashevych group $V_d(G)$ of a finitely presented self-similar group $G$ is finitely presented, and Skipper, Witzel, and the fourth author extended this result to higher finiteness properties in \cite[Theorem~4.15]{SWZ}.  We can ask the following analogous question about RSGs.

\begin{question}
If $G$ is a finitely presented RSG, then must the full closure of $G$ also be finitely presented?
\end{question}

The following more ambitious question is a concrete approach to trying to prove the full Boone--Higman conjecture:

\begin{question}
Does every finitely presented group with solvable word problem embed in a finitely presented oligomorphic group with finitely generated stabilizers of finite subsets?
\end{question}

If the answer to this question is ``Yes'', then we would have that every finitely presented group with solvable word problem embeds in a finitely presented (simple) twisted Brin--Thompson group, thus proving the Boone--Higman conjecture. In a similar vein, one can wonder whether twisted Brin--Thompson groups might always serve as the target of a Boone--Higman embedding:

\begin{question}\label{quest:always_twisted_bt}
Does every finitely presented simple group embed in a finitely presented (simple) twisted Brin--Thompson group?
\end{question}

Finally, another obvious question is whether one could try to extend the techniques here to prove the Boone--Higman conjecture for other classes of groups that are related to hyperbolic groups, such as CAT(0) groups, automatic groups, relatively hyperbolic groups with nice peripheral subgroups, and so forth. A starting point would have to be an interesting action on a Cantor space (e.g.\ the horofunction boundary in our case), which is not an issue in and of itself, but in order for our techniques here to work it turns out that the group of germs at any point must be virtually cyclic (see \cref{prop:CyclicStabilizers}), and this is an impediment for several potential approaches. Let us also mention another obvious family of groups related to hyperbolic groups, namely acylindrically hyperbolic groups; it turns out there exist acylindrically hyperbolic groups with unsolvable word problem \cite[Theorem~7.7]{MinasyanOsin} (in fact, one can just take $G*\Z$, for any $G$ with unsolvable word problem), so in fact they do not all embed into finitely presented simple groups, and it is not clear whether assuming acylindrical hyperbolicity should help with finding Boone--Higman embeddings.

\medskip

This paper is organized as follows. We begin by defining rational similarity groups (RSGs) in \cref{sec:the_groups}. In \cref{sec:fp} we prove that full, contracting RSGs are finitely presented. In \cref{sec:hyp} we show that every hyperbolic group embeds into a full, contracting RSG. Finally, in \cref{sec:ending} we prove that full, contracting RSGs embed into finitely presented simple groups, and hence so do all hyperbolic groups.

\medskip

\noindent\textbf{Acknowledgments.} Thanks are due to Martin Bridson, Matt Brin, Francesco Fournier-Facio, Anthony Genevois, James Hyde, Ilya Kapovich, Davide Perego, Shayo Olukoya, Lorna Richardson, Rachel Skipper, Owen Tanner, Slobodan Tanushevski, Matteo Tarocchi, and Xiaolei Wu for a variety of helpful discussions and pointers to references. We would particularly like to thank James Hyde for getting all of us interested in the Boone--Higman conjecture in the first place, and for contributing some of the ideas in the proof of \cref{prop:fin_gen_stabs}. We would also like to thank several anonymous referees for many helpful suggestions and comments. The first and second authors would like to thank the Universit\`a degli Studi di Milano--Bicocca (FA projects 2020-ATE-0006 and 2021-ATE-0033 ``Strutture Algebriche'') for travel support, as well as Anna Maria Savi and Patrizio Matucci for graciously hosting us in Florence.  They also wish to gratefully acknowledge partial support from EPSRC grant EP/R032866/1 during the creation of this paper. The third author is a member of the Gruppo Nazionale per le Strutture Algebriche, Geometriche e le loro Applicazioni (GNSAGA) of the Istituto Nazionale di Alta Matematica (INdAM) and of the PRIN 2022 ``Group theory and its applications'' research group and gratefully acknowledges the support of the PRIN project 2022-NAZ-0286 and of the Funda\c{c}\~ao para a Ci\^encia e a Tecnologia  (CEMAT-Ci\^encias FCT projects UIDB/04621/2020 and UIDP/04621/2020) and of the Universit\`a degli Studi di Milano--Bicocca (FA projects 2020-ATE-0006 and 2021-ATE-0033 ``Strutture Algebriche''). The fourth author is supported by grant \#635763 from the Simons Foundation.

\section{Rational similarity groups (RSGs)}\label{sec:the_groups}

In this section we construct a new family of groups called rational similarity groups (RSGs), see \cref{def:rsg}. These are certain groups of homeomorphisms, with two key defining properties: they consist of ``rational'' homeomorphisms of a subshift of finite type $\Sigma_\Gamma$, and they can achieve any ``canonical similarity'' between compatible cones in~$\Sigma_\Gamma$. Rational similarity groups seem to be fundamental to the study of asynchronous rational groups in the same way that self-similar groups are fundamental to the study of synchronous rational groups, in that they emerge as a unifying property common to most examples of interest in the literature. This will become clear in the course of constructing the groups and detailing many examples.

In \cref{ssec:subshifts} we establish notation and terminology for subshifts of finite type, and in \cref{ssec:rational} we develop a theory of rational homeomorphisms of such subshifts.  In Subsections~\ref{ssec:thompson} and~\ref{ssec:full_irred} we recall the Thompson group $V_{\Gamma,E}$ associated to a subshift and summarize its known properties. We introduce the class of RSGs in \cref{ssec:RSGs} and discuss several existing examples in the literature.  Finally, we introduce the nucleus for an RSG together with the property of an RSG being contracting in \cref{ssec:nuclei}, and in \cref{ssec:construct_groups} we use nuclei to prove a classification theorem for full RSGs.

\subsection{Subshifts of finite type}\label{ssec:subshifts}

In this subsection we briefly recall the definition of a subshift of finite type and establish relevant notation.

Throughout this subsection, let $\Gamma$ be a finite directed graph. Note that we allow $\Gamma$ to have both loops and multiple edges.  If $e$ is an edge of $\Gamma$, we let $o(e)$ denote the origin of $e$, and $t(e)$ denote the terminus of $e$. A \newword{directed path} in $\Gamma$ is a (finite or infinite) sequence $\{e_i\}$ of edges such that $t(e_i)=o(e_{i+1})$ for each~$i$.  A finite directed path $e_1\cdots e_n$ has an origin and terminus defined by
 \[
 o(e_1\cdots e_n)=o(e_1)\qquad\text{and}\qquad t(e_1\cdots e_n)=t(e_n).
 \]
An infinite directed path $e_1e_2\cdots$ has an origin $o(e_1)$ but no terminus. We will often just say ``path'' with ``directed'' being implicitly understood.
 
We will often think of the set of edges of $\Gamma$ as a finite alphabet, with finite directed paths being words.  In particular, an initial segment of a path $\alpha$ will be called a \newword{prefix} of $\alpha$ and any final segment of a path $\alpha$ will be called a \newword{suffix} of $\alpha$. For convenience, we also regard each node $v$ in $\Gamma$ as a path of length $0$ with $o(v)=t(v)=v$, where $v$ is considered to be a prefix of every path that has origin~$v$.  Finally, there is a \newword{null path} $\varnothing$ of length~$0$, which is a prefix of every path.

\begin{definition}[Subshift of finite type]
The \newword{subshift of finite type} associated to $\Gamma$ is the set $\Sigma_\Gamma$ of all infinite directed paths $e_1e_2 \cdots$ in $\Gamma$.
\end{definition}

This coincides with the terminology in~\cite{Matui15}, though ``subshift of finite type'' often refers to a slightly more general kind of sequence space (cf.\ \cite[Definition~2.1.1]{LindMarcus}). In that context, the subshifts $\Sigma_\Gamma$ defined above are sometimes called ``edge shifts'' or ``topological Markov shifts'', and up to topological conjugacy there is no distinction \cite{LindMarcus}.

If $\mathcal{E}$  is the set of edges of $\Gamma$, then we can view $\Sigma_\Gamma$ as a closed subspace of the infinite product space~$\mathcal{E}^{\N}$, where $\mathcal{E}$ has the discrete topology. With respect to this topology, $\Sigma_\Gamma$ is compact, totally disconnected, and metrizable.  Indeed, there is a \newword{standard ultrametric} on $\Sigma_\Gamma$, for which the distance between two distinct infinite paths $\omega$ and $\omega'$ is $1/2^n$, where $n$ is the length of the greatest common prefix of $\omega$ and $\omega'$.

\begin{example}
If $\Gamma$ has a single node and $n\ge 2$ (loop) edges, then $\Sigma_\Gamma$ is the usual $n$-ary Cantor space $C_n=\{0,\dots,n-1\}^\N$.
\end{example}

If $\alpha$ is a finite path in $\Gamma$, the associated \newword{cone} is the set $\C_\alpha\subseteq \Sigma_\Gamma$ of all infinite directed paths that have $\alpha$ as a prefix.  Note that if $v$ is a node in $\Gamma$, then $\C_v$ is the set of all infinite directed paths with origin~$v$. In addition, the cone $\C_\varnothing$ is the entire subshift~$\Sigma_\Gamma$. Each cone is a clopen subset of~$\Sigma_\Gamma$, and these form a basis for the topology.  In particular, a subset of~$\Sigma_\Gamma$ is clopen if and only if it is a disjoint union of finitely many cones.

\begin{remark}
If $\Sigma_\Gamma$ is a subshift of finite type, then the set of non-null finite paths in $\Gamma$ has the structure of a forest of finitely many rooted trees, with one tree for each node of~$\Gamma$.  The subshift $\Sigma_\Gamma$ is precisely the Gromov boundary of this forest, i.e.\ the space of leaves at infinity.
\end{remark}

\begin{remark}
We will largely be concerned with subshifts $\Sigma_\Gamma$ with the following two properties:
\begin{enumerate}
    \item The subshift $\Sigma_\Gamma$ has no isolated points; and\smallskip
    \item There are no empty cones, i.e.\ $\C_\alpha\ne\emptyset$ for every finite directed path~$\alpha$.
\end{enumerate}
If the first property holds, then the topological space $\Sigma_\Gamma$ is a Cantor space.  Together, these properties are equivalent to saying that for every node $v$ of $\Gamma$ there is a node $w$ of $\Gamma$ such that there is a directed path from $v$ to $w$ and $w$ has at least two outgoing edges.
\end{remark}

If $\alpha$ and $\beta$ are directed paths with $t(\alpha)=o(\beta)$, let $\alpha\cdot \beta$ denote the concatenation of $\alpha$ and $\beta$, where $v\cdot\alpha=\alpha$ for any directed path $\alpha$ with origin $v$ and $\varnothing \cdot \alpha=\alpha$ for every directed path~$\alpha$. If $\alpha$ is any finite path, observe that
\[
\C_\alpha = \{\alpha\cdot \omega \mid \omega\in \C_{t(\alpha)}\}\text{.}
\]
This holds for $\alpha=\varnothing$ as well if we adopt the convention that $t(\varnothing)=\varnothing$. 
The \newword{canonical similarity}
\[
L_\alpha \colon \C_{t(\alpha)}\to \C_\alpha
\]
is the homeomorphism defined by $L_\alpha(\omega) = \alpha\cdot \omega$.  More generally, if $\alpha$ and $\beta$ are finite paths with $t(\alpha)=t(\beta)$, the homeomorphism $L_\beta\circ L_\alpha^{-1}$ is the \newword{canonical similarity} $\C_\alpha\to \C_\beta$, sending $\alpha\cdot\omega$ to $\beta\cdot\omega$ for all $\omega\in \C_\alpha$.

\subsection{Rational homeomorphisms}\label{ssec:rational}

In this subsection we define rational homeomorphisms on subshifts of finite type, and we briefly develop the corresponding theory.  Rational homeomorphisms on full shifts were defined by Grigorchuk, Nekrashevych, and Sushchanski\u{\i} in~\cite{GNS}.

Throughout this subsection, let $\Sigma_\Gamma$ be a subshift of finite type without isolated points or empty cones. 
A \newword{nondegenerate map} on $\Sigma_\Gamma$ is a map $f\colon E\to \Sigma_\Gamma$, where $E$ is any nonempty clopen subset of $\Sigma_\Gamma$, with the property that no cone in $E$ maps to a single point.  For example, since $\Sigma_\Gamma$ has no isolated points and hence no one-point cones, any injective map on $\Sigma_\Gamma$ is nondegenerate.

If $E$ is a nonempty clopen subset of $\Sigma_\Gamma$, let $\Cones(E)$ denote the set of all finite paths $\alpha$ for which $\C_\alpha\subseteq E$. For a nondegenerate map $f\colon E\to \Sigma_\Gamma$, let
\[
\of\colon \Cones(E) \to \Cones(\Sigma_\Gamma)
\]
be the function that sends each $\alpha\in \mathrm{Cones}(E)$ to $\beta$, where $\C_\beta$ is the smallest cone in $\Sigma_\Gamma$ that contains $f(\C_\alpha)$.  Equivalently, $\of(\alpha)$ is the greatest common prefix of all points in $f(\C_\alpha)$ (which could be null). Note that this is well-defined since $f(\C_\alpha)$ has at least two points.

\begin{definition}[Local Action]
Let $f\colon E\to \Sigma_\Gamma$ be a nondegenerate map on $\Sigma_\Gamma$. If $\alpha\in \Cones(E)$, the \newword{local action} of $f$ at $\alpha$ is the map $f|_\alpha \colon \C_{t(\alpha)} \to \C_{t(\of(\alpha))}$ defined by
\[
f(\alpha\cdot \omega) = \of(\alpha)\cdot f|_{\alpha}(\omega)
\]
for all $\omega\in \C_{t(\alpha)}$.
\end{definition}

That is, $f|_\alpha$ is the map that fits into a commutative diagram
\[
\xymatrix@R=0.5in{
\C_{t(\alpha)} \ar_{f|_\alpha}[d] \ar^{L_\alpha}[r] & 
\C_{\alpha} \ar^{f}[d] \\ 
\C_{t(\of(\alpha))}\ar_{L_{\of(\alpha)}}[r] & \C_{\of(\alpha)}
}
\]
where $L_\alpha$ and $L_{\of(\alpha)}$ are canonical similarities.  Note then that two local actions $f|_\alpha$ and $g|_\beta$ are equal if and only if there is a commutative diagram
\[
\xymatrix@R=0.5in{
\C_{\alpha} \ar_{f}[d] \ar^{\phi}[r] & 
\C_{\beta} \ar^{g}[d] \\ 
\C_{\of(\alpha)}\ar_{\psi}[r] & \C_{\og(\beta)}
}
\]
where $\phi$ and $\psi$ are canonical similarities.
Moreover, note that if
$\alpha$ and $\beta$ are two finite directed paths with $f|_{\alpha}=g|_{\beta}$, then necessarily $t(\alpha)=t(\beta)$. This is because the functions $f|_{\alpha}$ and $g|_{\beta}$ have the same domain, made of infinite directed paths necessarily starting from the same node of $\Gamma$.

\begin{definition}[Rational map]
A nondegenerate map $f\colon E\to \Sigma_\Gamma$ is \newword{rational} if it only has finitely many distinct local actions.
\end{definition}

\begin{remark}
The local actions $f|_\alpha$ correspond to what Grigorchuk, Nekrashevych, and Sushchanski\u{\i} refer to as ``restrictions'' in~\cite{GNS}.  In the context of self-similar groups, they are also often referred to as ``states'', since they correspond to the states of a transducer, and rational maps are said to be ``finite-state''.

Though we do not use transducers here, there is no particular obstacle to writing a rational map on a subshift of finite type as a transducer.  The input and output alphabet would be the set of edges of $\Gamma$, and the transducer would have the property that it can only input or output valid words in the subshift.
\end{remark}

Let us collect some lemmas about local actions, which among other things will establish that the property of being rational is closed under compositions and inverses.

\begin{lemma}\label{lem:finite_to_one}
Let $f\colon E\to \Sigma_\Gamma$ be a nondegenerate map.  Then for each $\beta\in\Cones(\Sigma_\Gamma)$, the set $\of^{-1}(\beta)$ is finite.  If $f$ is rational and injective with exactly $k$ distinct local actions, then more precisely we have $\bigl|\of^{-1}(\beta)\bigr|\leq k$.
\end{lemma}
\begin{proof}
Let $n$ be the length of $\beta$.  Under the standard ultrametric, the cone $\C_\beta$ has diameter $2^{-n}$, and any subset of $\C_\beta$ of diameter $2^{-(n+1)}$ or less is contained in a proper subcone of $\C_\beta$. Since $f$ is uniformly continuous by virtue of $E$ being compact, there exists $\delta>0$ so that
\[
d(p,q) < \delta \quad\Rightarrow\quad d\bigl(f(p),f(q)\bigr) < 2^{-(n+1)}
\]
for all $p,q\in E$.  It follows that any $\alpha\in \of^{-1}(\beta)$ must have $\mathrm{diam}(\C_\alpha) \geq \delta$, and there are only finitely many such cones.

Now suppose that $f$ is rational and injective with exactly $k$ distinct local actions, and suppose to the contrary that $\bigl|\of^{-1}(\beta)\bigr|>k$.  By the pigeonhole principle, there must exist distinct $\alpha,\alpha'\in \of^{-1}(\beta)$ so that $f|_\alpha=f|_{\alpha'}$. Since $\Sigma_{\Gamma}$ has no isolated points, if $v$ is a node of $\Gamma$ such that $\Gamma$ admits an infinite path starting from $v$, then it must be the case that $\Gamma$ admits at least two (and in fact, infinitely many) distinct paths starting at $v$. As $t(\alpha)=t(\alpha')$ is a node of $\Gamma$ admitting an infinite path $\omega$ based at $t(\alpha)$, then it must admit another distinct path $\omega'$ based at $t(\alpha)$, and it follows that either $\alpha\cdot\omega\neq \alpha'\cdot\omega$ or $\alpha\cdot\omega'\neq \alpha'\cdot\omega'$.  Assume without meaningful loss of generality that $\alpha\cdot\omega\neq \alpha'\cdot\omega$. But now we have
\[
f(\alpha\cdot\omega) = \beta\cdot f|_{\alpha}(\omega) = \beta\cdot f|_{\alpha'}(\omega) = f(\alpha'\cdot \omega)
\]
and so $f$ is not injective, a contradiction.
\end{proof}

\begin{lemma}\label{lem:restrict_twice}
Let $f\colon E\to \Sigma_\Gamma$ be a nondegenerate map.  Then for finite paths $\alpha,\beta$ in\/ $\Gamma$ for which $\alpha\cdot\beta$ is defined, we have $(f|_\alpha)|_\beta = f|_{\alpha\cdot\beta}$.
\end{lemma}

\begin{proof}
Note that, since $t(\beta)=t(\alpha\cdot\beta)$, the maps $(f|_\alpha)|_\beta$ and $f|_{\alpha\cdot\beta}$ have the same domain $\C_{t(\beta)}$.  Furthermore, if $\omega\in \C_{t(\beta)}$, then
\begin{multline*}
\of(\alpha\cdot\beta)\cdot f|_{\alpha\cdot\beta}(\omega)
=
f(\alpha\cdot\beta\cdot\omega)
\\
=
\of(\alpha)\cdot f|_\alpha(\beta\cdot\omega)
=
\of(\alpha)\cdot \overline{f|_\alpha}(\beta)\cdot (f|_\alpha)|_\beta(\omega)\text{,}
\end{multline*}
so it suffices to prove that $\of(\alpha\cdot\beta)=\of(\alpha)\cdot \overline{f|_\alpha}(\beta)$.

Recall that $\overline{f|_\alpha}(\beta)$ is the greatest common prefix of $f|_\alpha(\C_\beta)$, or equivalently the greatest common prefix of all words $f|_\alpha(\beta\cdot \omega)$ for $\omega \in \C_{t(\beta)}$.  Then $\of(\alpha)\cdot \overline{f|_\alpha}(\beta)$ is the greatest common prefix of all words $\of(\alpha)\cdot f|_\alpha(\beta\cdot\omega)$ for $\omega\in \C_{t(\beta)}$.  But this is  precisely the set of words $f(\alpha\cdot\beta\cdot\omega)$ for $\omega\in \C_{t(\beta)}$, and the greatest common prefix of this set is $\of(\alpha\cdot\beta)$.  We conclude that $\of(\alpha\cdot\beta)=\of(\alpha)\cdot\overline{f|_\alpha}(\beta)$, and hence $(f|_\alpha)|_\beta = f|_{\alpha\cdot\beta}$.
\end{proof}

\begin{lemma}\label{lem:restrict_composition}
Let $E,E'\subseteq \Sigma_\Gamma$ be clopen, let $f\colon E'\to \Sigma_{\Gamma}$ and $g\colon E\to E'$ be nondegenerate maps, and suppose $f\circ g$ is nondegenerate. Then for any $\alpha\in\Cones(E)$,
\[
(f\circ g)|_\alpha = (f|_{\og(\alpha)} \circ g|_\alpha)|_{t(\alpha)}.
\]
\end{lemma}

\begin{proof}
Let $h=f|_{\og(\alpha)}\circ g|_\alpha$. 
 We must prove that $(f\circ g)|_\alpha = h|_{t(\alpha)}$.  Since $t(t(\alpha))=t(\alpha)$,
 both of these maps have the same domain~$\C_{t(\alpha)}$.  Furthermore, if $\omega\in \C_{t(\alpha)}$, then
\begin{multline*}
\overline{f\circ g}(\alpha)\cdot (f\circ g)|_\alpha(\omega) = (f\circ g)(\alpha\cdot\omega)
= f\bigl(\og(\alpha)\cdot g|_\alpha(\omega)\bigr) \\ = (\of\circ\og)(\alpha) \cdot h(\omega) = (\of\circ\og)(\alpha) \cdot h(t(\alpha)\cdot\omega) \\ = (\of\circ\og)(\alpha) \cdot \oh(t(\alpha)) \cdot h|_{t(\alpha)}(\omega)
\end{multline*}
so it suffices to prove that $\overline{f\circ g}(\alpha) = (\of\circ\og)(\alpha)\cdot \oh(t(\alpha))$.  (This will ensure that the maps agree, and also that they have the same codomain.)

Recall that $\oh(t(\alpha))$ is the greatest common prefix of the words $h(\omega)$ for $\omega\in \C_{t(\alpha)}$.  Then $(\of\circ\og)(\alpha)\cdot \oh(t(\alpha))$ is the greatest common prefix of the words $(\of\circ\og)(\alpha)\cdot h(\omega)$ for $\omega\in \C_{t(\alpha)}$.  By the calculation above, this is precisely the set of words $(f\circ g)(\alpha\cdot\omega)$ for $\omega\in \C_{t(\alpha)}$, and the greatest common prefix of this set is $\overline{f\circ g}(\alpha)$.  We conclude that $\overline{f\circ g}(\alpha) = (\of\circ\og)(\alpha)\cdot \oh(t(\alpha))$, and therefore  $(f\circ g)|_\alpha = h|_{t(\alpha)}
$.
\end{proof}

\begin{lemma}\label{lem:inverses_rational}
Let $E,E'\subseteq \Sigma_\Gamma$ be clopen.  If $f\colon E\to E'$ is a rational homeomorphism, then $f^{-1}$ is rational.
\end{lemma}
\begin{proof}
For simplicity, we suppose instead that $f^{-1}$ is rational, and we prove that $f$ is rational.  Put an equivalence relation on $\Cones(E)$ by $\alpha\sim\beta$ if $f^{-1}|_{\of(\alpha)}=f^{-1}|_{\of(\beta)}$.  Since $f^{-1}$ is rational, there are only finitely many equivalence classes.  Let $\mathcal{C}$ be such an equivalence class.  It suffices to prove that $f$ has only finitely many distinct local actions at elements of~$\mathcal{C}$.

Fix a $\gamma\in \mathcal{C}$.  By the first part of \cref{lem:finite_to_one}, there are only finitely many $\beta\in\mathcal{C}$ so that $\of(\beta)=\of(\gamma)$.  Therefore, it suffices to prove that for every $\alpha\in\mathcal{C}$ there exists $\beta\in\mathcal{C}$ such that $f|_\alpha=f|_\beta$ and $\of(\beta)=\of(\gamma)$.

Let $\alpha\in \mathcal{C}$.  Since $f^{-1}|_{\of(\alpha)}=f^{-1}|_{\of(\gamma)}$ we have a commutative diagram
\[
\xymatrix@R=0.5in{
\C_{\of(\alpha)} \ar_{f^{-1}}[d] \ar^{\psi}[r] & 
\C_{\of(\gamma)} \ar^{f^{-1}}[d] \\ 
\C_{\alpha'}\ar_{\phi}[r] & \C_{\gamma'} }
\]
where $\C_{\alpha'}$ and $\C_{\gamma'}$ are the smallest cones that contain $f^{-1}(\C_{\of(\alpha)})$ and $f^{-1}(\C_{\of(\gamma)})$, respectively, and $\phi$ and $\psi$ are canonical similarities.  
Since $f(\C_\alpha)\subseteq \C_{\of(\alpha)}$, we know that $\C_\alpha\subseteq f^{-1}(\C_{\of(\alpha)})\subseteq \C_{\alpha'}$, so $\alpha'$ is a prefix of $\alpha$.  Thus, $\phi(\C_\alpha)$ must be some cone $\C_\beta\subseteq \C_{\gamma'}$.  Since $f(\C_\beta)=\psi(f(\C_\alpha))$ and $f(\C_\alpha)$ is not contained in any proper subcone of $\C_{\of(\alpha)}$, the set $f(\C_\beta)$ is not contained in any proper subcone of $\C_{\of(\gamma)}$, and therefore $\of(\beta)=\of(\gamma)$.  We also have the commutative diagram
\[
\xymatrix@R=0.5in{
\C_{\of(\alpha)} \ar^{\psi}[r] & 
\C_{\of(\gamma)} \\ 
\C_{\alpha}\ar_{\phi}[r]\ar^{f}[u] & \C_{\beta}\ar_{f}[u] }
\]
which proves that $f|_\alpha=f|_\beta$. This finishes the proof that $f$ is rational.
\end{proof}

We have proven the following.

\begin{proposition}
Let\/ $\Sigma_\Gamma$ be a subshift of finite type without isolated points or empty cones, and let $E\subseteq \Sigma_\Gamma$ be a nonempty clopen set.  Then the set\/ $\R_{\Gamma,E}$ of rational homeomorphisms $E\to E$ forms a group under composition.\qed
\end{proposition}

The group $\R_{\Gamma,E}$ is the \newword{rational group} associated to $E$.  In the case where $E$ is the whole subshift $\Sigma_{\Gamma}$, we write $\R_{\Gamma}$ for $\R_{\Gamma,E}$.  A \newword{rational representation} of a group $G$ is any homomorphism $G\to \R_{\Gamma,E}$.

\begin{remark}\label{rem:FullShift}
If $\Gamma$ consists of a single node and $n\geq 2$ (loop) edges, then the corresponding subshift $\Sigma_\Gamma$ is known as the \newword{full shift} with an alphabet of size~$n$, and is the usual $n$-ary Cantor space $\C_n$. In this case, the rational group $\R_\Gamma$ is the same as the rational group $\R_n$ defined by Grigorchuk, Nekrashevych, and Sushchanski\u\i~\cite{GNS}.  \end{remark}

\subsection{Thompson groups on subshifts}\label{ssec:thompson}

Matsumoto introduced the topological full group associated to a subshift of finite type $\Sigma_\Gamma$ in 2015~\cite{Matsumoto15}. Here we give a brief definition of Matsumoto's groups that does not use the language of \'{e}tale groupoids. Specifically, we define one group $V_{\Gamma,E}$ for each clopen set $E\subseteq \Sigma_\Gamma$, which we refer to as the ``Thompson group associated to~$E$''.  In the case where $E=\Sigma_\Gamma$, we write $V_\Gamma$ instead of~$V_{\Gamma,E}$. To define the groups $V_{\Gamma,E}$, recall that any two cones $\C_\alpha,\C_\beta\subseteq E$ with $t(\alpha)=t(\beta)$ have a corresponding canonical similarity $\alpha\cdot\omega \mapsto \beta\cdot\omega$.

\begin{definition}[Thompson group]\label{def:thomp}
Let $\Sigma_\Gamma$ be a subshift of finite type and $E\subseteq \Sigma_\Gamma$ a nonempty clopen set. The \newword{Thompson group associated to $\boldsymbol{E}$} is the group of all homeomorphisms $f\colon E\to E$ satisfying the following property: there exist two partitions $\C_{\alpha_1},\ldots,\C_{\alpha_n}$ and $\C_{\beta_1},\ldots,\C_{\beta_n}$ of $E$ into cones, such that $t(\alpha_i)=t(\beta_i)$ for each $i$, and $f$ maps each $\C_{\alpha_i}$ to $\C_{\beta_i}$ by the canonical similarity.
\end{definition}

It is not difficult to prove that the set $V_{\Gamma,E}$ of all such homeomorphisms really does form a group.

\begin{example}[Higman--Thompson groups]\label{ex:V} If $\Sigma_\Gamma$ is the full shift on an alphabet of size $n$ (see \cref{rem:FullShift}), then $V_\Gamma$ is the well-known Higman--Thompson group $V_{n,1}$.  More generally, if $E$ is any nonempty clopen subset of $\Sigma_{\Gamma}$, then $V_{\Gamma,E}$ is isomorphic to one of the Higman--Thompson groups $V_{n,r}$.
\end{example}

In the case where $\Sigma_\Gamma$ has no isolated points or empty cones, the following proposition shows that the Thompson group $V_{\Gamma,E}$ is a subgroup of the corresponding rational group $\R_{\Gamma,E}$.

\begin{proposition}
Suppose $\Sigma_\Gamma$ has no isolated points or empty cones, and let $E\subseteq \Sigma_\Gamma$ be a nonempty clopen set.  Then $V_{\Gamma,E}$ is precisely the group of all $f\in \R_{\Gamma,E}$ with the property that $f|_{\alpha}$ is the identity on $\C_{t(\alpha)}$ for all but finitely many $\alpha\in\Cones(E)$.
\end{proposition}

\begin{proof}
Note that $f|_\alpha$ is the identity on $\C_{t(\alpha)}$ if and only if $f$ maps $\C_\alpha$ to some cone $\C_\beta$ by a canonical similarity.  If $f\in V_{\Gamma,E}$ has domain partition $\C_{\alpha_1},\ldots,\C_{\alpha_n}$, then $f$ acts as a canonical similarity on any cone that is contained in one of the $\C_{\alpha_i}$, and all but finitely many cones have this property.  Conversely, if $f$ acts as a canonical similarity on all but finitely many cones, then by compactness we can find a finite cover of $E$ by cones on which $f$ acts as a canonical similarity, and any such cover has a subcover whose cones are disjoint.
\end{proof}

Matui proved that $V_{\Gamma,E}$ is finitely presented for a certain class of subshifts~$\Sigma_\Gamma$.

\begin{definition}[Irreducible]
A subshift $\Sigma_\Gamma$ is  \newword{irreducible} if the following conditions are satisfied:
\begin{enumerate}
    \item $\Gamma$ is \newword{strongly connected}, i.e.\ for any two nodes $v,w$ of $\Gamma$ there is a directed path with origin $v$ and terminus $w$.\smallskip
    \item $\Gamma$ is not a directed cycle.
\end{enumerate}
\end{definition}

Note that if $\Sigma_\Gamma$ is irreducible, then it has no isolated points, and is therefore a Cantor space. Also, $\Sigma_\Gamma$ has no empty cones, so the rational group $\R_{\Gamma,E}$ is defined for any nonempty clopen $E\subseteq \Sigma_\Gamma$. The following is proven in \cite[Theorem~6.21]{Matui15}.

\begin{theorem}[Matui] 
If\/ $\Sigma_\Gamma$ is irreducible and $E\subseteq \Sigma_\Gamma$ is a nonempty clopen set, then $V_{\Gamma,E}$ is finitely presented.  Indeed, $V_{\Gamma,E}$ has type\/~$\F_{\!\infty}$.\qed
\end{theorem}

We will need a slight generalization of Matui's theorem.

\begin{definition}[Irreducible core]
A subshift $\Sigma_\Gamma$ has an \newword{irreducible core} if there exists an induced subgraph $\Gamma_0$ of $\Gamma$ (the \newword{core}) such that:
\begin{enumerate}
    \item $\Sigma_{\Gamma_0}$ is irreducible.\smallskip
    \item For every node $v$ of $\Gamma$, there is a directed path in $\Gamma$ from $v$ to a node of~$\Gamma_0$.\smallskip
    \item There exists $N\geq 0$ so that every directed path in $\Gamma$ of length $N$ (with any origin) has terminus in~$\Gamma_0$.
\end{enumerate}
\end{definition}

Note that this is still enough to ensure that $\Sigma_\Gamma$ has no isolated points or empty cones.

Given clopen subsets $E,E'\subseteq \Sigma_\Gamma$,  a homeomorphism $h\colon E\to E'$ is \newword{Thompson-like} if there exists a partition of $E$ into cones $\C_{\alpha_1},\dots,\C_{\alpha_n}$ and a partition of $E'$ into cones $\C_{\beta_1},\dots,\C_{\beta_n}$ such that $h$ maps each $\C_{\alpha_i}$ to $\C_{\beta_i}$ by a canonical similarity.  For example, $V_{\Gamma,E}$ is the group of all Thompson-like homeomorphisms from $E$ to itself.

\begin{proposition}[Pushing into the core]\label{prop:push_into_core}
Suppose that\/ $\Sigma_\Gamma$ has an irreducible core\/ $\Gamma_0$. Let $E\subseteq \Sigma_\Gamma$ be a nonempty clopen set.  Then there exists a clopen set $E_0\subseteq \Sigma_{\Gamma_0}$ and a Thompson-like homeomorphism $h\colon E\to E_0$ such that $hV_{\Gamma,E}h^{-1} = V_{\Gamma_0,E_0}$.
\end{proposition}

\begin{proof}
Let $N\geq 0$ be such that every directed path in $\Gamma$ of length $N$ has terminus in~$\Gamma_0$.  Note that every finite directed path of length greater than $N$ also has terminus in~$\Gamma_0$, and therefore every directed edge in $\Gamma$ whose origin lies in $\Gamma_0$ also has its terminus in $\Gamma_0$.  In particular $\Sigma_{\Gamma_0}$ is precisely the union of the cones $\C_v$ in $\Sigma_\Gamma$ as $v$ ranges over all nodes in~$\Gamma_0$.

Let $\alpha_1,\ldots,\alpha_k$ be all the elements of $\Cones(E)$ of length~$N$. Then $\C_{\alpha_1},\ldots, \C_{\alpha_k}$ is a partition of $E$ into cones, and each $t(\alpha_i)$ is a node in~$\Gamma_0$.  Since $\Gamma_0$ is irreducible, it is not difficult to find disjoint cones $\C_{\beta_1},\ldots,\C_{\beta_k}$ in $\Sigma_{\Gamma_0}$ such that $t(\beta_i)=t(\alpha_i)$ for all $i$.  Let $E_0=\C_{\beta_1}\cup \cdots\cup \C_{\beta_k}$. Then the homeomorphism $h\colon E\to E_0$ that maps each $\C_{\alpha_i}$ to $\C_{\beta_i}$ by the canonical similarity has the desired properties.
\end{proof}

Now our generalization of Matui's theorem is immediate.

\begin{corollary}\label{cor:V_fp}
If\/ $\Sigma_\Gamma$ has an irreducible core and $E\subseteq \Sigma_\Gamma$ is a nonempty clopen set, then $V_{\Gamma,E}$ is finitely presented.  Indeed, it has type\/~$\F_{\!\infty}$.\qed
\end{corollary}

\begin{remark}
If $\Gamma$ and $\Gamma'$ both have irreducible cores, then $\R_{\Gamma,E}\cong \R_{\Gamma',E'}$ for any $E$ and $E'$; in particular all such groups are isomorphic to~$\R_2$, the rational group on the usual binary Cantor space. This follows from \cite[Proposition~2.22]{BelkBleakMatucci} together with the proof of \cite[Theorem~2.16]{BelkBleakMatucci}. We should also mention that this group is simple, and is not finitely generated \cite{Belk-Hyde-Matucci-1}.  Note, however, that these isomorphisms do not map $V_{\Gamma,E}$ to $V_{\Gamma',E'}$.
\end{remark}

\begin{example}\label{ex:houghton}
The groups $V_{\Gamma,E}$ are much less well-behaved when $\Gamma$ does not have an irreducible core. For example, as observed by Matteo Tarocchi \cite{Tarocchi}, if $\Gamma$ is the graph with nodes $u,v_1,\ldots,v_n$, a loop at each node, and a directed edge from each $v_i$ to $u$, then one can check that $V_{\Gamma}$ is isomorphic to the Houghton group~$H_n$, which has type $\F_{\!n-1}$ but not type $\F_{\!n}$~\cite{Brown}. (See also \cite[Proposition~3.7.4]{cornulier14} for a similar construction.) In particular, none of these examples have type~$\F_{\!\infty}$, in contrast to \cref{cor:V_fp}.  It would be interesting to determine the finiteness properties of $V_{\Gamma,E}$ for all finite directed graphs~$\Gamma$.
\end{example}

\subsection{Full groups, flexible groups, and classes }\label{ssec:full_irred}

In this subsection we define full groups and flexible groups, and we discuss the abelian group of classes associated to any full flexible group. The ``full and flexible'' terminology was introduced in \cite{Belk-Hyde-Matucci-1}, and the group of classes is essentially due to Matui~\cite[Section~6.2]{Matui15}, though our description follows~\cite[Section~3]{BleakElliottHyde}.

\begin{definition}[Full group]\label{def:full}
For $X$ a topological space and $G$ a group of homeomorphisms of~$X$, say that a homeomorphism $h$ of $X$ \newword{locally agrees with $G$} if for all $x\in X$ there exists an open neighborhood $U$ of $x$ and an element $g\in G$ such that $h(u)=g(u)$ for all $u\in U$. The \newword{full group induced by $G$}, denoted $[[\,G\mid X\,]]$, is the group of all homeomorphisms that locally agree with $G$. Call $G$ \newword{full} if $[[\,G\mid X\,]] = G$.
\end{definition}

For example, the Thompson group $V_{\Gamma,E}$ is full for any $\Sigma_\Gamma$ and any nonempty clopen set~$E$.  

This terminology and notation is inspired by the theory of \'{e}tale groupoids, where $[[\mathfrak{G}]]$ denotes the topological full group of an \'{e}tale groupoid~$\mathfrak{G}$ (see \cite{Matui12,Matui13,Matsumoto15,Matui15,MatuiMatsumoto17}).  In particular, if $X$ is a Cantor space then $[[\,G\mid X\,]]$ is the topological full group of the \'{e}tale groupoid of all germs of elements of~$G$.  Note that if $X$ is compact, then $h$ locally agrees with $G$ if and only if there exists a finite covering of $X$ by (basic, if desired) open sets $U_1,\dots,U_n$ and elements $g_1,\dots,g_n\in G$ such that $h(u)=g_i(u)$ for all $i$ and all $u\in U_i$.

\begin{definition}[Flexible group]
If $X$ is a Cantor space, a group $G$ of homeomorphisms of $X$ is \newword{flexible} if, for every pair $E_1,E_2$ of proper, nonempty clopen subsets of $X$, there exists $g\in G$ such that $g(E_1)\subseteq E_2$.
\end{definition}

\begin{remark}
Flexible groups are closely related to the vigorous groups defined by the second author, Elliott, and Hyde~\cite{BleakElliottHyde}, and are precisely the CO-transitive groups of homeomorphisms of Cantor spaces as defined by Kim, Koberda, and Lodha \cite{KKL,KimKoberdaBook}.  Moreover, the class of full flexible groups is the same as the class of topological full groups of essentially principal, minimal, purely infinite  \'etale groupoids with unit space a Cantor space (see~\cite{Matui15}), and Matui has proven that all such groups have simple commutator subgroup.
\end{remark}

If $G$ is a full flexible group of homeomorphism of a Cantor space~$X$ and $E$ is a proper, nonempty clopen subset of $X$, the \newword{class of $\boldsymbol{E}$}, denoted $\class(E)$, is the collection of all clopen sets in the same $G$-orbit as~$E$.  Let $\Classes(G)$ be the collection of such classes.  There is a natural binary operation $+$ on $\Classes(G)$ defined by
\[
\class(E_1) + \class(E_2) = \class(E_1\cup E_2)
\]
for every pair $E_1,E_2$ of disjoint, nonempty clopen sets whose union $E_1\cup E_2$ is not all of~$X$.  Under this operation, $\Classes(G)$ forms an abelian group (see \cite[Section~6.2]{Matui15} or~\cite[Section~3]{BleakElliottHyde}).  The key thing is that, since $G$ is flexible, up to choosing different representatives the sum $\class(E_1)+\class(E_2)$ makes sense for any proper, nonempty $E_1$ and $E_2$.  From this it is straightforward to see that it is a group operation, since for any $E_1$ and $E_2$ there exists $E_3$ with $\class(E_1)+\class(E_3)=\class(E_2)$ (just use flexibility to move $E_1$ into $E_2$ and then take $E_3$ to be the complement).  The group $\Classes(G)$ is precisely the $0$th homology group of the associated \'etale groupoid as defined by Crainic and Moerdijk~\cite{CrainicMoerdijk} and described concretely by Matui \cite{Matui12}.  

By convention, we can also assign a class to the whole space~$X$, namely the sum $\class(E_1)+\class(E_2)$ for any partition $\{E_1,E_2\}$ of $X$ into nonempty clopen sets.  It is easy to show that this does not depend on the chosen partition.

Matui describes the group $\Classes(V_{\Gamma,E})$ in the irreducible case, where it has an especially nice group presentation. See \cite[Section~6.1]{Matui15} or \cite[Theorem~4.14]{Matui12}, where it is proven that $\Classes(V_{\Gamma,E})$ is isomorphic to the $K_0$ group for the associated reduced groupoid \mbox{$C^*$-algebra}.  These $C^*$-algebras are in the family of Cuntz--Krieger algebras \cite{CuntzKrieger}, and their \mbox{K-theory} was computed by Cuntz~\cite{Cuntz}.

\begin{theorem}[Matui]
\label{thm:ClassesVGammaE}
If\/ $\Sigma_\Gamma$ is irreducible and $E\subseteq \Sigma_\Gamma$ is a nonempty clopen set, then $V_{\Gamma,E}$ is flexible.  If $v$ is a node of\/ $\Gamma$, then all cones\/ $\C_\alpha$ in $E$ with $t(\alpha)=v$ lie in the same class\/ $[v]$, and such classes generate\/ $\Classes(V_{\Gamma,E})$.  Indeed, $\Classes(V_{\Gamma,E})$ has a presentation with these generators and one relation
\[
[v] = [w_1] + \cdots + [w_n]
\]
for each node $v$ of\/ $\Gamma$, where $w_1,\ldots,w_n$ are the termini of all edges $e_1,\ldots,e_n$ in\/ $\Gamma$ that have origin~$v$.\qed
\end{theorem}

In the above theorem, the relation $[v]=[w_1]+\cdots +[w_n]$ comes from the decomposition of a cone $\C_{\alpha}$ with $t(\alpha)=v$ into the disjoint union $\C_{\alpha e_1}\cup \cdots \cup \C_{\alpha e_n}$, where each $t(\alpha e_i)=w_i$.  Note that $e_1,\ldots,e_n$ are distinct but $w_1,\ldots,w_n$ need not be.

Note that Matui's description of $\Classes(V_{\Gamma,E})$ does not actually depend on the clopen set~$E$.  Indeed, if $E\subseteq \Sigma_\Gamma$ is any nonempty clopen set, the inclusion $E\to \Sigma_{\Gamma}$ induces a canonical isomorphism $\Classes(V_{\Gamma,E})\to \Classes(V_\Gamma)$.

\begin{example}\label{ex:ClassesInHigmanThompson}
In the Higman--Thompson case where $\Gamma$ is a single node with $n\geq 2$ loops (see \cref{ex:V}) the group $\Classes(V_{\Gamma,E})$ is isomorphic to $\Z/(n-1)\Z$. In this case, each $V_{\Gamma,E}$ is isomorphic to one of the Higman--Thompson groups $V_{n,r}$ for $1\leq r\leq n$, where $r$ depends on the class of~$E$.
\end{example}

\begin{example}
If $\Gamma$ has two nodes $v$ and $w$, each with two loops, and with one edge going from $v$ to $w$ and one edge going from $w$ to $v$, then
\[
\Classes(V_{\Gamma,E}) \cong \langle [v],[w] \mid [v]=2[v]+[w],[w]=[v]+2[w]\rangle \cong \Z.
\]
In particular $\Classes(V_{\Gamma,E})$ can be infinite.
\end{example}

\begin{corollary}[Irreducible core case]\label{cor:IrreducibleCoreFlexible}
If\/ $\Sigma_\Gamma$ has an irreducible core\/~$\Gamma_0$ and $E\subseteq \Sigma_{\Gamma}$ is a nonempty clopen set, then $V_{\Gamma,E}$ is flexible.  Moreover, if\/ $\C_\alpha$ and\/ $\C_\beta$ are cones in $E$ and $t(\alpha)=t(\beta)$, then\/ $\class(\C_{\alpha})=\class(\C_\beta)$.
\end{corollary}
\begin{proof}
By \cref{prop:push_into_core} there is a Thompson-like homeomorphism $h\colon E\to E_0$, where $E_0$ is a clopen subset of $\Sigma_{\Gamma_0}$.  Then $h$ conjugates $V_{\Gamma,E}$ to $V_{\Gamma_0,E_0}$, and the latter is flexible by \cref{thm:ClassesVGammaE}, so $V_{\Gamma,E}$ is flexible as well.  Furthermore, observe that $h$ induces an isomorphism $\Classes(V_{\Gamma,E})\to \Classes(V_{\Gamma_0,E_0})$.

Now suppose $\C_\alpha$ and $\C_\beta$ are cones in $E$ with $t(\alpha)=t(\beta)$, and let $f\colon \C_\alpha\to \C_\beta$ be the canonical similarity.  Then there exists a decomposition of $\C_{\alpha_1},\ldots,\C_{\alpha_n}$ of $\C_{\alpha}$ into cones so that $h$ acts as a canonical similarity on both $\C_{\alpha_i}$ and $\C_{\of(\alpha_i)}$ for each~$i$.  Since $t(\of(\alpha_i)) = t(\alpha_i)$, we have $t(\overline{hf}(\alpha_i))=t(\oh(\alpha_i))$ for each~$i$, so $\class(h(\C_{\alpha_i}))=\class(h(\C_{\of(\alpha_i)}))$. Since $h$ induces an isomorphism on classes, it follows that $\class(\C_{\alpha_i})=\class(\C_{\of(\alpha_i)})$ for each~$i$, so
\[
\class(\C_{\alpha})=\sum_{i=1}^n \class(\C_{\alpha_i}) = \sum_{i=1}^n \class(\C_{\of(\alpha_i)}) = \class(\C_\beta).\qedhere
\]
\end{proof}

We will often make use of the following corollary.

\begin{corollary}[Mapping cones with $V_{\Gamma,E}$]\label{cor:MappingConesWithV}
Suppose $\Sigma_\Gamma$ has an irreducible core, let $E\subseteq \Sigma_\Gamma$ be a nonempty clopen set, and let\/ $\C_{\alpha_1}\,\ldots,\C_{\alpha_n}$ and\/ $\C_{\beta_1}\,\ldots,\C_{\beta_n}$ be collections of pairwise disjoint cones in~$E$.  If $t(\alpha_i)=t(\beta_i)$ for each $i$ and neither of the unions\/ $\bigcup_{i=1}^n \C_{\alpha_i}$ and\/ $\bigcup_{i=1}^n \C_{\beta_i}$ is equal to $E$, then there exists $f\in V_{\Gamma,E}$ that maps each\/ $\C_{\alpha_i}$ to\/ $\C_{\beta_i}$ by a canonical similarity.
\end{corollary}

\begin{proof}
Let $E_\alpha$ and $E_\beta$ denote the two unions, neither of which is equal to~$E$.  Since $t(\alpha_i)=t(\beta_i)$ for each~$i$, it follows from \cref{cor:IrreducibleCoreFlexible} that $\class(\C_{\alpha_i})=\class(\C_{\beta_i})$ for each~$i$. Then $\class(E_\alpha)=\class(E_\beta)$, so there exists $g\in V_{\Gamma,E}$ that maps $E_\alpha$ to $E_\beta$, and in particular $g$ maps $E\setminus E_\alpha$ to $E\setminus E_\beta$.  Then the homeomorphism $f\colon E\to E$ that maps each $\C_{\alpha_i}$ to $\C_{\beta_i}$ by a canonical similarity and agrees with $g$ on $E\setminus E_\alpha$ is the desired element of~$V_{\Gamma,E}$.
\end{proof}

\subsection{Rational similarity groups}\label{ssec:RSGs}

In this subsection we introduce the class of rational similarity groups, which will be our main focus for the remainder of the paper.

\begin{definition}\label{def:rsg}
Let $\Sigma_{\Gamma}$ be a subshift of finite type with no isolated points or empty cones, let $E\subseteq \Sigma_\Gamma$ be a nonempty clopen set, and let $\R_{\Gamma,E}$ be the associated rational group.  A subgroup $G\leq \R_{\Gamma,E}$ will be called a \newword{rational similarity group (RSG)} if, for every pair of cones $\C_\alpha,\C_\beta \subsetneq E$ with $t(\alpha)=t(\beta)$, there exists $g\in G$ that maps $\C_\alpha$ to $\C_\beta$ by the canonical similarity.
\end{definition}

It follows from the results in \cite{BelkBleakMatucci} that every hyperbolic group $G$ that acts faithfully on its horofunction boundary $\partial_h G$ is an RSG (see \cref{sec:hyp}), and our embedding of $G$ into a finitely presented simple group will depend on the corresponding full closure $[[G \mid \partial_h G]]$.  The following proposition describes the close relationship between full RSGs and Thompson groups.

\begin{proposition}[Full RSGs and $V_{\Gamma,E}$]\label{prop:RSGsAndV}
Suppose\/ $\Sigma_\Gamma$ has an irreducible core. Then $V_{\Gamma,E}$ is an RSG, as is any subgroup of $\R_{\Gamma,E}$ that contains $V_{\Gamma,E}$.  Conversely, every full RSG in $\R_{\Gamma,E}$ contains $V_{\Gamma,E}$.
\end{proposition}
\begin{proof}
If $G$ contains $V_{\Gamma,E}$, then by \cref{cor:MappingConesWithV} all canonical similarities can be achieved using elements of $G$. Conversely, if $G$ is a full RSG then since every element of $V_{\Gamma,E}$ locally agrees with a canonical similarity at every point, and since $G$ is full, every element of $V_{\Gamma,E}$ lies in $G$.
\end{proof}

\begin{remark}
Without the assumption of an irreducible core, there do exist examples of $\Gamma$ where $V_{\Gamma,E}$ is not an RSG. For instance, take $\Gamma$ to be the graph with two nodes $v$ and $w$, one loop $a$ at $v$, three loops $b_1,b_2,b_3$ at $w$, and one edge $x$ from $v$ to $w$. Then no element of $V_\Gamma$ can realize the canonical similarity from $\C_a$ to $\C_{aa}$, since the complement of~$\C_a$, namely~$\C_x$, can only be partitioned into an odd number of cones, whereas the complement of $\C_{aa}$, namely $\C_{ax}\cup\C_x$, can only be partitioned into an even number of cones.
\end{remark}
In the irreducible core case, \cref{prop:RSGsAndV} makes it easy to exhibit a variety of other existing examples of RSGs, beyond the $V_{\Gamma,E}$. As a rudimentary example, note that given any subgroup $H\le \R_{\Gamma,E}$ one can form an RSG by simply taking $G=\langle H,V_{\Gamma,E}\rangle$. Let us look at some more interesting, explicit examples from the literature.

\begin{example}[R\"over--Nekrashevych groups]\label{ex:rn}
Let $\Gamma$ consist of a single node and $n\ge 2$ (loop) edges, so $\Sigma_\Gamma=\C_n$ is the usual $n$-ary Cantor space that is the boundary of the rooted $n$-regular tree $T_n$. Note that $\Aut(T_n)$ is isomorphic to the wreath product $\Aut(T_n)\wr S_n$. Call a subgroup $G\le \Aut(T_n)$ \newword{self-similar} if its image in $\Aut(T_n)\wr S_n$ under this isomorphism lies in $G \wr S_n$. For $g\in G$, if the image of $g$ under this isomorphism is $((g_1,\dots,g_n),\sigma)$, call the $g_i$ the \newword{level-1 states} of $g$. By the \newword{states} of $g$ we mean all elements obtainable by a finite sequence of taking level-1 states. Thus, $G$ is self-similar if and only if all the states of any $g\in G$ themselves lie in $G$. We call a self-similar group \newword{finite-state} if each of its elements has only finitely many states. See \cite{Nekrashevych-book} for a wealth of information about self-similar groups.

Given a self-similar group $G$, the \newword{R\"over--Nekrashevych group} $V_n(G)$ is the group of homeomorphisms $h$ of $\C_n$ for which there exists a partition of $\C_n$ into cones such the image of any cone in the partition under $h$ is some cone in $\C_n$, and the local action of $h$ at each cone in the partition is an element of $G$. These groups were essentially introduced by Scott in \cite{scott84} in very different language. Subsequently, R\"over studied the special case when $G$ is the Grigorchuk group in \cite{roever99}, and Nekrashevych considered the groups in full generality and with modern terminology in \cite{nekrashevych04}. Note that for any element $h$ of $V_n(G)$, all but finitely many local actions of $h$ lie in $G$, so if $G$ is finite-state then $V_n(G)$ is a subgroup of the rational group $\R_n=\R_{\Gamma}$. Also, clearly $V_n\le V_n(G)$, so we conclude that any R\"over--Nekrashevych group of a finite-state self-similar group is a full RSG.
\end{example}

\begin{example}[(Twisted) Brin--Thompson groups]\label{ex:BT}
In \cite{brin04}, Brin defined a class of higher-dimensional Thompson groups $nV$, which we refer to as the Brin--Thompson groups.  The first and second authors proved in \cite{Belk-Bleak-1} that these groups embed in $\R_{2^n}$, and indeed they are full RSGs.  In \cite{BelkZaremskyTwisted}, the first and fourth authors generalized these to the class of twisted Brin--Thompson groups $SV_G$, and when $S$ is finite these are again full RSGs.
\end{example}

\begin{example}[Circular Thompson groups $T_{n,r}$]\label{ex:T}
Every Higman--Thompson group $V_{n,r}\cong V_{\Gamma,E}$ (as in \cref{ex:V}) has an associated circular group $T_{n,r}\leq V_{n,r}$, which is the subgroup consisting of all elements that preserve a certain circular order on~$E$ (see~\cite{Brown}). These groups are RSGs, but they are not full as groups of homeomorphisms of the Cantor space~$E$. Indeed, the full closure of $T_{n,r}$ is precisely~$V_{n,r}$. Note that $T_{n,r}$ is full when viewed as a group of homeomorphisms of a circle.\end{example}

\begin{example}[Thompson's group $F$]\label{ex:F}
Thompson's group $V$ ($=V_{2,1}$) is the same as $V_{\Gamma}$, where $\Gamma$ is a graph with one node and two loops.  Thompson's group $F$ is the subgroup of $V$ that preserves the canonical linear order on $\Sigma_\Gamma$.  This group is not an RSG, since cones that contain the maximum or minimum points under the linear order cannot be mapped to other cones.

However, it is possible to realize $F$ as an RSG using a slightly different subshift.  Specifically, let $\Gamma'$ be the graph with three vertices $u$, $v$, and $w$ (corresponding to left cones, interior cones, and right cones, respectively), with loops at $u$ and $w$, two loops at $v$, and one directed edge from each of $u$ and $w$ to $v$. Then $F$ embeds into $V_{\Gamma'}$ as an RSG, though it is not full.
\end{example}

\begin{example}[Lodha--Moore group]\label{ex:LM}
The Lodha--Moore group originated in \cite{LodhaMoore,Lodha}, as the first example of a non-amenable type~$\F_{\!\infty}$ group with no non-abelian free subgroups. The original construction is analogous to Thompson's group $F$, but one can also consider a $V$-like version. Let $\Gamma$ have one node and two (loop) edges, so $\Sigma_\Gamma=\C_2$ is the usual binary Cantor space $\{0,1\}^\N$. Let $\lambda$ be the homeomorphism of $\Sigma_\Gamma$ defined by
\[
\lambda(00\cdot\omega) = 0\cdot \lambda(\omega),\quad \lambda(01\cdot\omega) = 10\cdot\lambda^{-1}(\omega),\quad \lambda(11\cdot\omega)=1\cdot\lambda(\omega)
\]
for all $\omega\in \Sigma_\Gamma$. One can check that $\lambda$ is a rational map. For any finite binary word $\alpha$, let $\lambda_\alpha$ be the homeomorphism of $\Sigma_\Gamma$ that sends $\alpha\cdot\omega$ to $\alpha\cdot\lambda(\omega)$ for all $\omega$ and fixes all points outside $\C_\alpha$. The ($V$-like) Lodha--Moore group $LM_V$ is then the group generated by $V_\Gamma$ together with all the $\lambda_\alpha$. This is a full RSG. Note that, just like the original Lodha--Moore group, $LM_V$ also has type $\F_{\!\infty}$, as recently shown by Farley \cite{Farley}. As with $F$ and $T$ (see Examples~\ref{ex:T} and~\ref{ex:F}), the Lodha--Moore groups for an interval and a circle can also be realized as RSGs, though they are not full. See \cite{LodhaCircle} for more on the circle version.
\end{example}

\begin{example}[Automorphisms of $V_{n,r}$]\label{ex:autV}
For a Higman--Thompson group $V_{n,r}\cong V_{\Gamma,E}$ as in \cref{ex:V}, the second author, Cameron, Maissel, Navas, and Olukoya proved that the automorphism group $\mathrm{Aut}(V_{n,r})$ embeds into $\R_{\Gamma,E}$~\cite{BCMNO}.  The group $\mathrm{Aut}(V_{n,r})$ contains $V_{\Gamma,E}$, and is therefore an RSG.  A feature of interest here is that the group $\mathrm{Aut}(\Sigma_\Gamma,\sigma)$ of automorphisms of the one-sided subshift $\Sigma_\Gamma$ (with shift operator $\sigma$) embeds as a subgroup $\mathcal{H}_n$ of the outer automorphism group $\mathrm{Out}(V_{n,r})$ (see \cite{Hedlund69} for information on $\mathrm{Aut}(\Sigma_\Gamma)$ and \cite{BCO} for the embedding mentioned here).  However, the group $\Aut(V_{n,r})$ is not full since the image of an element $f\in\Aut(V_{n,r})$ in $\mathrm{Out}(V_{n,r})$ can be determined from the restriction of $f$ to any cone.
\end{example}

\subsection{Nuclei and the contracting property}\label{ssec:nuclei}

In this short subsection we define the crucial notion of a subgroup of $\R_{\Gamma,E}$ being contracting. Let $\Sigma_\Gamma$ be a subshift of finite type without isolated points or empty cones, and let $E\subseteq \Sigma_\Gamma$ be a nonempty clopen set.

If $f\in \R_{\Gamma,E}$, then $f$ has only finitely many distinct local actions~$f|_\alpha$.  The \newword{nucleus} of~$f$, denoted $\Nuc_f$, is the set of all local actions that occur infinitely often. That is, $f|_{\alpha}\in \Nuc_f$ if $f|_{\alpha}=f|_{\beta}$ for infinitely many $\beta\in\Cones(E)$. 
 Note that $f|_\alpha\notin \Nuc_f$ holds for only finitely many $\alpha\in\Cones(E)$, so in particular there exists $N\in\N$ so that $f|_\alpha\in \Nuc_f$ for all $\alpha\in \Cones(E)$ for which the path $\alpha$ has length $N$ or greater.

\begin{definition}[Nucleus, contracting]\label{def:contracting}
Given a subgroup $G\le \R_{\Gamma,E}$, define the \newword{nucleus} $\Nuc_G$ of $G$ to be the union of all $\Nuc_g$ for $g\in G$. Thus, $\Nuc_G$ is the smallest set of maps such that for all $g\in G$ we have $g|_\alpha\in\Nuc_G$ for all but finitely many $\alpha\in \Cones(E)$. We say that $G\le \R_{\Gamma,E}$ is \newword{contracting} if $\Sigma_\Gamma$ has an irreducible core and $\Nuc_G$ is finite.
\end{definition}

This generalizes the notion of contracting for self-similar groups introduced by Nekrashevych~\cite{Nekrashevych-book}, and in the special case of R\"over--Nekrashevych groups (see \cref{ex:rn}) it is equivalent to the underlying self-similar group being contracting.  A related notion of contracting for asynchronous transducers appears in \cite{DonovenOlukoya} as well.  Their use case applies to maps arising as products of local maps coming from a fixed asynchronous transducer, as opposed to the whole collection of small-scale local maps from a particular group of homeomorphisms, but the concepts otherwise align.

Note that the elements of $\Nuc_G$ themselves are not necessarily in $G$. Indeed, the domain of an element of $\Nuc_G$ can be either all of $\Sigma_\Gamma$ or a cone $\C_v$ for some node $v$ in $\Gamma$, as opposed to the clopen set $E$. Moreover, the image of an element of $\Nuc_G$ is a clopen subset of $\Sigma_{\Gamma}$, but need not be a cone. Note also that the elements of $\Nuc_G$ are always injective open maps, since the elements of $G$ are homeomorphisms.

\begin{remark}
Our definition of $G\le \R_{\Gamma,E}$ being contracting includes both that $\Sigma_\Gamma$ has an irreducible core, and that $\Nuc_G$ is finite. It is worth mentioning that these are independent properties. For example, in the $V$-like Lodha--Moore group from \cref{ex:LM} there is an irreducible core but an infinite nucleus, and in the Houghton groups from \cref{ex:houghton} there is no irreducible core, but the nucleus is finite.
\end{remark}

\subsection{Classification of full RSGs}\label{ssec:construct_groups}

Let $\Sigma_\Gamma$ be a subshift of finite type with irreducible core~$\Gamma_0$.

\begin{definition}[Nucleus of injections]
Let $\Nuc$ be a set of nondegenerate maps on $\Sigma_\Gamma$.  We say that $\Nuc$ is a \newword{nucleus of injections} if it has the following properties:
\begin{description}

\listlabel{MapNuc} Each element of $\Nuc$ is a map $\C_v\to \C_w$ for some nodes $v,w$ of~$\Gamma_0$.\smallskip

\listlabel{IdNuc} For each node $v$ of $\Gamma_0$, the identity map on $\C_v$ is an element of $\Nuc$. (Equivalently, the nucleus of the identity map on $\Sigma_\Gamma$ is contained in $\Nuc$.)\smallskip

\listlabel{LocNuc}
For each $p\in\Nuc$, every local action $p|_\alpha$ belongs to~$\Nuc$.\smallskip

\listlabel{RecurNuc}
For each $p\in\Nuc$, there exists $q\in \Nuc$ such that $p\in \Nuc_q$.\smallskip

\listlabel{InvNuc}
Each $p\in\Nuc$ is an injective open map, and satisfies $\Nuc_{p^{-1}}\subseteq \Nuc$.\smallskip

\listlabel{ProdNuc}
If $p\colon \C_v\to \C_w$ and $q\colon \C_u\to \C_v$ are elements of $\Nuc$, then $\Nuc_{pq}\subseteq \Nuc$.
\end{description}
\end{definition}

Note that \MapNuc\ follows from \RecurNuc\ together with \cref{lem:finite_to_one}, but we include it for simplicity.  In the case where $\Nuc$ is finite, condition \RecurNuc\ is equivalent to saying that for each $p\in\Nuc$, there exists $q\in\Nuc$ and an edge $e$ in $\Gamma$ so that $q|_e=p$ (as in the nucleus of a contracting self-similar group), but for infinite $\Nuc$ these conditions are not equivalent.

\begin{proposition}\label{prop:NucleusHasProperties}
Let $\Sigma_\Gamma$ be a subshift of finite type with irreducible core~$\Gamma_0$. If $G\le \R_{\Gamma,E}$ is an RSG, then the nucleus $\Nuc$ for $G$ is a nucleus of injections.
\end{proposition}
\begin{proof}
Let $G\leq \R_{\Gamma,E}$ be an RSG.

For \MapNuc, let
\[
\Cones_0(E) = \{\alpha\in \Cones(E) \mid t(\alpha)\text{ is a vertex in $\Gamma_0$}\}
\]
and observe that $\Cones(E)\setminus \Cones_0(E)$ is a finite set. 
 If $g\in G$, then it follows from \cref{lem:finite_to_one} that both $\alpha$ and $\og(\alpha)$ lie in $\Cones_0(E)$ for all but finitely many~$\alpha$, so any element of $\Nuc$ must have this property.

For \IdNuc, since $\textrm{id}_E\in G$ and every node of $\Gamma_0$ is the terminus of an arbitrarily long directed path in $\Cones(E)$, we know that $\textrm{id}_{\C_v}\in\Nuc$ for all nodes $v$ of $\Gamma_0$, so \IdNuc\ holds. 

For \LocNuc, if $p\in \Nuc_g$ for some $g\in G$, then $p=g|_\beta$ for infinitely many different $\beta$.  If $p|_\alpha$ is a local action of $p$, then by \cref{lem:restrict_twice} we have $p|_\alpha=g|_{\beta\cdot\alpha}$ for all such $\beta$, so $p|_\alpha\in \Nuc_g$.

For \RecurNuc, let $p\in \Nuc_g$ for some $g\in G$. Then $g|_\alpha=p$ for infinitely many different~$\alpha$. Hence there must exist infinitely many different $\beta$ with the property that $g|_{\beta\cdot\gamma}=p$ for infinitely many different suffixes~$\gamma$. Since $g|_{\beta\cdot\gamma}=(g|_\beta)|_\gamma$ by \cref{lem:restrict_twice}, we conclude that  $p\in \Nuc_{g|_\beta}$ for infinitely many different~$\beta$.  Since $g$ is rational, one such $g|_\beta$ must lie in $\Nuc_g$, which gives $p\in \Nuc_{g|_\beta}$ and $g|_\beta\in \Nuc$.

For \InvNuc, let $p\in\Nuc_g$ for some $g\in G$, so $g|_\beta=p$ for infinitely many different $\beta$. Since $g$ is a homeomorphism, $p$ must be an injective open map.  Let $\C_\alpha$ be any cone contained in the image of~$p$.  Then for any $\beta$ with $g|_\beta=p$, we have $g^{-1}|_{\overline{g}(\beta)}(g|_\beta)=(g^{-1}g)|_\beta=1$ by \cref{lem:restrict_composition} and so $g^{-1}|_{\overline{g}(\beta)}=(g|_\beta)^{-1}$ (with domain the image of $p$), which implies that $g^{-1}|_{\og(\beta)\cdot \alpha} = (g^{-1}|_{\og(\beta)})|_\alpha = p^{-1}|_{\alpha}$ by \cref{lem:restrict_twice}.  Since there are infinitely many different $\og(\beta)$ by \cref{lem:finite_to_one}, we conclude that $p^{-1}|_\alpha$ is in $\Nuc_{g^{-1}}$ and hence in $\Nuc$.

Finally, for \ProdNuc, let $p\in \Nuc_g$ and $q\in \Nuc_h$ for some $g,h\in G$, where $p\colon \C_v\to \C_w$ and $q\colon \C_u\to \C_v$ for some nodes $u,v,w$ in $\Gamma_0$. 
Then there exists $\beta$ so that $g|_\beta=p$ and $\C_{\beta}$ is a proper subset of~$E$, and by \cref{lem:finite_to_one} there exists $\gamma$ so that $h|_\gamma=q$ and $\C_{\oh(\gamma)}$ is a proper subset of~$E$. Since $t(\beta)=t\bigl(\oh(\gamma)\bigr)=v$ and $G$ is an RSG, there exists $f\in G$ that maps $\C_{\oh(\gamma)}$ to $\C_\beta$ by the canonical similarity. Then $gfh\in G$ and $(gfh)|_{\gamma\cdot\alpha}=(pq)|_\alpha$ for all $\alpha$ by Lemmas~\ref{lem:restrict_composition} and~\ref{lem:restrict_twice}, so for all but finitely many~$\alpha$ we have that $(pq)|_\alpha$ is in $\Nuc_{gfh}$ and hence in $\Nuc$.
\end{proof}

\begin{remark}
It is not true that the nucleus for an arbitrary subgroup $G\leq \R_{\Gamma,E}$ is a nucleus of injections.  Though \MapNuc, \IdNuc, \LocNuc, \RecurNuc, and \InvNuc\ always hold, there are groups $G$ whose nucleus does not satisfy \ProdNuc.  For example, if $g$ and $h$ are the automorphisms of the infinite, rooted binary tree that satisfy the wreath recursion
\[
g = (0\;1)(h,h)\qquad\text{and}\qquad h=(g,g)
\]
then the group $\langle g\rangle\cong \Z_2$ has nucleus $\{g,h\}$, and this does not satisfy \ProdNuc\ since $(gh)|_\alpha=gh$ for all $\alpha$. 
\end{remark}

The following theorem serves as a classification of full RSGs.

\begin{theorem}\label{thrm:RSGCharacterization}
Let $\Sigma_\Gamma$ be a subshift of finite type with irreducible core~$\Gamma_0$. Let $\Nuc$ be a nucleus of injections over $\Sigma_\Gamma$, and let $E\subseteq \Sigma_\Gamma$ be a nonempty clopen set. Then
\[
G = \{f\in \R_{\Gamma,E} \mid \Nuc_f\subseteq \Nuc\}
\]
is a full RSG with nucleus $\Nuc$, and every full RSG has this form.
\end{theorem}

\begin{proof}
We prove first that $G$ contains $V_{\Gamma,E}$, and in particular $G$ is nonempty. Let $f\in V_{\Gamma,E}$.  Then there exists a partition $\C_{\alpha_1},\ldots,\C_{\alpha_n}$ of $E$ into cones so that $f$ acts as a canonical similarity on each cone.  Subdividing further, we may assume that each $t(\alpha_i)$ is in~$\Gamma_0$.  Then for each cone $\C_\beta$ contained in any of the $\C_{\alpha_i}$, the local action $f|_\beta$ is the identity map on $\C_v$ for some node $v$ of~$\Gamma_0$.  By property \IdNuc, it follows that $f\in G$, and therefore $G$ contains $V_{\Gamma,E}$.

We claim that $G$ is a subgroup of $\R_{\Gamma,E}$. We first prove that $G$ is closed under inversion.  Let $g\in G$.  We need to prove that $g^{-1}|_\alpha \in\Nuc$ for all but finitely many $\alpha$. First note that for $\alpha,\beta$ such that $\C_\alpha \subseteq g(\C_\beta)$, say with $\alpha=\og(\beta)\cdot\gamma$, by \cref{lem:restrict_twice}, we have $g^{-1}|_\alpha = (g^{-1}|_{\og(\beta)})|_{\gamma} = (g|_\beta)^{-1}|_\gamma$, where the last equality follows from the fact that $g^{-1}|_{\og(\beta)}$ and $(g|_\beta)^{-1}$ agree on~$\C_\gamma$. Now observe that for all but finitely many $\alpha$ there exists $\beta$ such that $\C_\alpha\subseteq g(\C_\beta)$ and $g|_\beta\in\Nuc$. By the above, for such an $\alpha$ and $\beta$, say $\alpha=\og(\beta)\cdot\gamma$, we have that $g^{-1}|_\alpha$ equals the local action at $\gamma$ of the inverse of an element of $\Nuc$. For all but finitely many $\gamma$ this lies in $\Nuc$ by property \InvNuc. Hence, considering decompositions $\alpha=\og(\beta)\cdot\gamma$ where $\beta$ has minimal length satisfying $g|_\beta\in\Nuc$, we see that for all but finitely many $\alpha$ there is such a decomposition satisfying $(g|_\beta)^{-1}|_\gamma \in\Nuc$. We conclude that $g^{-1}|_\alpha \in \Nuc$ for all but finitely many $\alpha$, so $g^{-1}\in G$.

Next we prove $G$ is closed under composition. Let $g,h\in G$, and we need to prove that $(g\circ h)|_\alpha\in\Nuc$ for all but finitely many $\alpha$. First note that for any $\beta$ we have $(g\circ h)|_\beta = (g|_{\oh(\beta)}\circ h|_\beta)|_{t(\beta)}$ by \cref{lem:restrict_composition}. Now for any cone $\C_\gamma\subseteq \C_{t(\beta)}$, \cref{lem:restrict_twice} says that the local action of this at $\gamma$ equals $(g|_{\oh(\beta)}\circ h|_\beta)|_\gamma$. Since $\oh$ is finite-to-one by \cref{lem:finite_to_one}, for all but finitely many $\beta$ this is the local action at $\gamma$ of a product of two elements of $\Nuc$. Property \ProdNuc\ implies this is itself an element of $\Nuc$ for all but finitely many $\gamma$. In particular, $(g\circ h)|_{\beta\cdot\gamma}$ lies in $\Nuc$ for all but finitely many $\beta$ and $\gamma$. We conclude that $g\circ h \in G$, so $G$ is a subgroup of $\R_{\Gamma,E}$.  Since $G$ contains $V_{\Gamma,E}$, it is an RSG by \cref{prop:RSGsAndV}.

To prove that $G$ is full, suppose $h\in\Homeo(E)$ locally agrees with $G$. This implies that there exists a partition of $E$ into cones $\C_{\alpha_1},\ldots,\C_{\alpha_n}$ and elements $g_1,\dots,g_n\in G$ such that $h(\alpha_i\cdot\omega)=g_i(\alpha_i\cdot\omega)$ for all $i$ and all $\omega\in \C_{t(\alpha_i)}$. Now for any cone $\C_\alpha$ contained in one of the $\C_{\alpha_i}$ (of which only finitely are not), we have $h|_\alpha = g_i |_\alpha$, so all but finitely many local actions of $h$ lie in $\Nuc$. Hence $h\in G$, so $G$ is full.

As for the nucleus $\Nuc_G$ of $G$, clearly $\Nuc_G\subseteq\Nuc$. For the opposite inclusion, if $p\in \Nuc$ then by \RecurNuc\ there exists $q\in\Nuc$ so that $p\in\Nuc_q$.  By \MapNuc, we know that $q\colon \C_v\to \C_w$ for some nodes $v,w$ in $\Gamma_0$.  Since $\Gamma_0$ is irreducible, we can find two disjoint cones $\C_\alpha,\C_\beta\subseteq E$ such that $t(\alpha) = v$ and $t(\beta)=w$.  Let $g\colon E\to E$ be the homeomorphism that maps $\C_\alpha$ into $\C_\beta$ by $L_\beta\circ q\circ L_{\alpha}^{-1}$ (where $L_\alpha$ and $L_\beta$ are canonical similarities), maps $g(\C_\alpha)$ back to $\C_\alpha$ by the inverse of this mapping, and is the identity elsewhere.  Then it is easy to check that $g\in G$ and $g|_\alpha=q$.  It follows that $\Nuc_q\subseteq \Nuc_g$, so $p\in \Nuc_g$, and therefore $p\in \Nuc_G$.

Finally, by \cref{prop:NucleusHasProperties} the nucleus for any RSG is a nucleus of injections, so suppose $H\leq \R_{\Gamma,E}$ is any full RSG with nucleus $\Nuc$.  We claim that $H=G$.  Clearly $H\leq G$, so we must show that $G\leq H$.  Since $H$ is full, it suffices to prove that every element of $G$ locally agrees with~$H$, so let $g\in G$ and $\omega \in E$.  We can choose a neighborhood of $\omega$  of the form $\C_\alpha$ such that $g|_\alpha\in\Nuc$, and by \cref{lem:finite_to_one} we can arrange for both $\C_{\alpha}$ and $\C_{\oh(\alpha)}$ to be properly contained in~$E$.  Let $p=g|_\alpha$.  Since $p\in\Nuc$, it is in the nucleus of $H$, so there exists $h\in H$ such that $p\in \Nuc_h$. That is, $h|_{\beta}=p$ for infinitely many~$\beta$. By \cref{lem:finite_to_one} we can choose one such $\beta$ so that neither $\C_\beta$ nor $\C_{\oh(\beta)}$ is all of~$E$. Since $t(\beta)=t(\alpha)$ and $t\bigl(\oh(\beta)\bigr)=t\bigl(\og(\alpha)\bigr)$, by \cref{cor:MappingConesWithV}~there exist elements $f,f'\in V_{\Gamma,E}$ so that $f$ maps $\C_\alpha$ to $\C_\beta$ by the canonical similarity, and $f'$ maps $\C_{\oh(\beta)}$ to $\C_{\og(\alpha)}$ by the canonical similarity.  Then $g$ agrees with $f'hf$ on $\C_\alpha$, and $f'hf\in H$ since $H$ is an RSG.  We conclude that $g$ locally agrees with~$H$, and thus $g\in H$, which proves that $H=G$.
\end{proof}

The following corollary gives a nice alternative characterization for the elements of a full RSG.

\begin{corollary}\label{cor:rsg_partition}
Let $\Sigma_\Gamma$ be a subshift of finite type with irreducible core~$\Gamma_0$. Let $G\leq \R_{\Gamma,E}$ be a full RSG with nucleus~$\Nuc$. Then a homeomorphism $h\colon E\to E$ lies in $G$ if and only if there exists a finite partition of $E$ into cones\/ $\C_{\alpha_i}$ such that each $g
|_{\alpha_i}\in\Nuc$.
\end{corollary}

\begin{proof}
If such a partition exists, then clearly $g\in G$ by \cref{thrm:RSGCharacterization}. Now suppose $g\in G$. Since $g|_\alpha\in\Nuc$ for all but finitely many $\alpha$, there exists $M$ such that $g|_\alpha\in\Nuc$ for all paths of length $M$. Now if $\{\alpha_1,\dots,\alpha_m\}$ is the set of all paths of length $M$, then the cones $\C_{\alpha_i}$ form the desired partition.
\end{proof}

We can use \cref{thrm:RSGCharacterization} to build some new explicit examples of RSGs, for which the nuclei contain non-invertible elements.

\begin{example}
Let $\Gamma$ have a single node and three (loop) edges, called $0$, $1$, and $2$. Note that $\Sigma_\Gamma$ is the usual $3$-ary Cantor space $\{0,1,2\}^\N$. Let $f\colon \Sigma_\Gamma\to \Sigma_\Gamma$ be the rational injection defined by
\[
f(0\cdot\omega) = 0\cdot f(\omega),\qquad f(1\cdot\omega) = 02\cdot\omega,\qquad f(2\cdot\omega)=1\cdot\omega
\]
for all $\omega\in \Sigma_\Gamma$. The following facts are easy to verify:
\begin{enumerate}
    \item $f|_0=f$, and $f|_1=f|_2=1$.\smallskip
    \item $\mathrm{im}(f) = \C_0 \cup \C_1$, with $f^{-1}|_0 = f$ and $f^{-1}|_1= 1$. \smallskip
    \item $f^2$ is the mapping $\omega\mapsto 0\cdot\omega$, so in particular $f^2|_\varnothing=1$.
\end{enumerate}
Let $\Nuc=\{1,f\}$. Then $\Nuc$ is a nucleus of injections, so the group defined as in \cref{thrm:RSGCharacterization} is a full, contracting RSG.
\end{example}

\begin{example}
Let $\Gamma$ have a single node and two (loop) edges, called $0$ and $1$, so $\Sigma_\Gamma$ is the usual $2$-ary Cantor space $\{0,1\}^\N$. Let $f\colon \Sigma_\Gamma\to \Sigma_\Gamma$ be the rational injection defined by
\[
f(0\cdot\omega) = 0\cdot f(\omega),\qquad f(10\cdot\omega) = 011\cdot\omega,\qquad f(11\cdot\omega)=10\cdot\omega
\]
for all $\omega\in \Sigma_\Gamma$.  The following facts are easy to verify:
\begin{enumerate}
    \item $f|_0=f$, and $f|_{10}=f|_{11}=1$.\smallskip
    \item $\mathrm{im}(f) = \C_0 \cup \C_{10}$, with $f^{-1}|_0 = f$ and $f^{-1}|_{10} = 1$. \smallskip
    \item $f^2$ is the mapping $\omega\mapsto 0\cdot\omega$, so in particular $f^2|_\varnothing=1$.
\end{enumerate}
Let $\Nuc=\{1,f\}$. Then $\Nuc$ is a nucleus of injections, so the group defined as in \cref{thrm:RSGCharacterization} is a full, contracting RSG.
\end{example}

\section{Finite presentability of full, contracting RSGs}\label{sec:fp}

This section comprises a proof of the following:

\begin{theorem}\label{thrm:fin_pres}
Every full, contracting RSG is finitely presented.
\end{theorem}

In the special case where the RSG is the R\"over--Nekrashevych group of a contracting self-similar group (see \cref{ex:rn}), Nekrashevych proved finite presentability in \cite{NekraFP}. Our general situation is much more complicated since $\Gamma$ may have more than one node, we do not have a self-similar group to work with, and the elements of $\Nuc$ might not even be invertible. Thus, a few aspects of our proof are inspired by the proof in \cite{NekraFP}, but in general many new ideas and approaches are needed.

Throughout this section, we fix a full, contracting RSG $G\le \R_{\Gamma,E}$, with nucleus $\Nuc$.  We introduce both a finite and an infinite generating set for $G$ in \cref{ssec:nuc_gens}, as well as a special class of words in the infinite generating set that we call ``normalish forms''.  In \cref{ssec:infinitepres} we construct an infinite presentation for $G$, and in \cref{ssec:finitepres} we prove that only finitely many relations are needed.

\subsection{Nuclear generators}\label{ssec:nuc_gens}

Let $\N\Nuc$ be the free commutative monoid generated by the elements of~$\Nuc$.  We think of elements of $\N\Nuc$ as formal sums $P=p_1+\cdots +p_k$ of elements of $\Nuc$. If $p\colon\C_v\to \C_w$ is an element of $\Nuc$, let
\[
\partial_1(p) = \class(p(\C_v)) - \class(\C_v) \text{.}
\]
This induces a monoid homomorphism $\partial_1\colon \N\Nuc\to \Classes(V_\Gamma)$. Elements of $\ker(\partial_1)$ are called \newword{cycles} in $\N\Nuc$. (This homomorphism $\partial_1$ is closely related to the first boundary homomorphism for \'etale groupoid homology defined by Matui \cite{Matui12}.)

\begin{lemma}\label{lem:monoid_fp}
The monoid\/ $\ker(\partial_1)$ of cycles in\/ $\N\Nuc$ is finitely generated.
\end{lemma}

\begin{proof}
Note that for $P,P'\in \N\Nuc$, if $P$ and $P + P'$ are both cycles then $P'$ is a cycle. This means that $\ker(\partial_1)$ is a subtractive submonoid of $\N\Nuc$ in the sense of \cite{Eilenberg}, and hence is finitely generated by \cite[Proposition~7.1]{Eilenberg}. \end{proof}

Fix some finite generating set for $\ker(\partial_1)$.  We will use these generators to construct a finite generating set for~$G$.

First we need some definitions.  We will use the term \newword{code} to mean any set $\{\alpha_1,\ldots,\alpha_k\}$ in $\Cones(E)$ whose corresponding cones $\C_{\alpha_i}$ are pairwise disjoint.  A code is \newword{complete} if $\bigcup_{i=1}^k \C_{\alpha_i}=E$, and \newword{incomplete} otherwise. If $P=p_1+\cdots +p_k$ is a generator for $\ker(\partial_1)$, where $p_i\colon \C_{v_i}\to \C_{w_i}$, a \newword{domain code} for $P$ is a code $\{\nu_1,\ldots,\nu_k\}$ such that $t(\nu_i)=v_i$ for each~$i$, and \newword{range code} is defined analogously.

Now, for each generator $P=p_1+\cdots+p_k$, we choose a \newword{model nuclear generator} $h_P\in G$ as follows.  Choose domain and range codes $\{\nu_1,\ldots,\nu_k\}$ and $\{\xi_1,\ldots,\xi_k\}$ for $P$, and let $D=\bigcup_{i=1}^k \C_{\nu_i}$ and $R=\bigcup_{i=1}^k \C_{\xi_i}$.  Thanks to flexibility, we can choose the domain and range codes to both be incomplete, so $D$ and $R$ are both proper subsets of $E$.  Let $g\colon D\to R$ be the map that sends each $\C_{\nu_i}$ to $\C_{\xi_i}$ with local action~$p_i$.  Since $P\in \ker(\partial_1)$, the image $g(D)$ has the same class as~$D$, so by \cref{cor:MappingConesWithV} we can choose some $s\in V_{\Gamma,E}$ that maps $g(D)$ onto~$D$.  Now we define $h_P$ to be the homeomorphism of $E$ that agrees with $s\circ g$ on $D$ and is the identity elsewhere.  Note that $h_P\in G$ since $G$ is full.  We should emphasize that $h_P$ depends on the choice of domain and range codes, though we suppress this from the notation.

The main goal of this subsection is to prove the following proposition.

\begin{proposition}\label{prop:GFinitelyGenerated}
The group $G$ is finitely generated, with generating set consisting of the model nuclear generators $h_P$ together with a finite generating set for~$V_{\Gamma,E}$.
\end{proposition}

We begin by defining some conjugates of the model nuclear generators that will be useful.  If $f\in V_{\Gamma,E}$, we say that $f$ \newword{maps} a code $\{\alpha_1,\ldots,\alpha_k\}$ to a code $\{\beta_1,\ldots,\beta_k\}$ if $f$ maps each cone $\C_{\alpha_i}$ to $\C_{\beta_i}$ by the canonical similarity.

\begin{definition}[Nuclear generators]
A \newword{nuclear generator} is a conjugate
\[
h = ch_Pc^{-1}
\]
where $h_P$ is a model nuclear generator with domain code $\{\nu_1,\ldots,\nu_k\}$ and $c$ is an element of $V_{\Gamma,E}$ that maps this to\ some code $\{\alpha_1,\ldots,\alpha_k\}$.
\end{definition}

Note that $h$ is supported on the union $\bigcup_{i=1}^k \C_{\alpha_i}$.  This is the \newword{domain of support} for~$h$.  The code $\{\alpha_1,\ldots,\alpha_k\}$ is the \newword{nuclear code} 
for~$h$, and the corresponding \newword{nuclear states} are the summands $p_i$ of~$P$.  We think of the nuclear code and corresponding nuclear states as part of the definition of a nuclear generator.

A \newword{rectifier} for a nuclear generator $h$ is an element $r\in V_{\Gamma,E}$ such that $(rh)|_{\alpha_i}=p_i$ for each~$i$. Here $h=ch_P c^{-1}$ and $P=p_1+\cdots+p_k$ as above. For example, if $s$ is the element used to construct $h_P$, then $s^{-1}$ is a rectifier for~$h_P$, and more generally $s^{-1}c^{-1}$ is a rectifier for any nuclear generator $ch_Pc^{-1}$.

\begin{definition}[Normalish form]\label{def:normalish}
A \newword{normalish form} for an element $g\in G$ is a word for $g$ of the form
\[
fh_1\cdots h_n
\]
where $f\in V_{\Gamma,E}$ and $h_1,\ldots,h_n$ are nuclear generators whose domains of support are pairwise disjoint.
\end{definition}

As we will see, every element of $G$ has a normalish form, but this form is not unique.  The \newword{nuclear code} for a normalish form is the union of the nuclear codes for the $h_i$.

\begin{lemma}[Recognizing normalish forms]\label{lem:RecognizingNormalish}
Let $g\in G$, and let $h_1\cdots h_n$ be a normalish form with nuclear code $\{\alpha_1,\ldots,\alpha_k\}$ and corresponding nuclear states $p_1,\ldots,p_k$.  Then the following are equivalent:
\begin{enumerate}
    \item There exists $f\in V_{\Gamma,E}$ such that $g=fh_1\cdots h_n$.\smallskip
    \item There exist $r_1,\ldots,r_k\in V_{\Gamma,E}$ such that $(r_ig)|_{\alpha_i}=p_i$ for each~$i$, and $g$ agrees with some element of $V_{\Gamma,E}$ on the complement of\/~$\bigcup_{i=1}^k\C_{\alpha_i}$.
\end{enumerate}
\end{lemma}
\begin{proof}
If (i) holds then $g$ agrees with $f$ on the complement of $\bigcup_{i=1}^k \C_{\alpha_i}$.  Furthermore, for each $i$ there exists $j$ so that $\alpha_i$ is in the nuclear code for $h_j$. If $r\in V_{\Gamma,E}$ is a rectifier for $h_j$, then $rf^{-1}g$ agrees with $rh_j$ on $\C_{\alpha_i}$, so $(rf^{-1}g)|_{\alpha_i}=(rh_j)|_{\alpha_i}=p_i$, which proves~(ii).

Now suppose (ii) holds, and let $h=h_1\cdots h_n$.  We must prove that $gh^{-1}\in V_{\Gamma,E}$.  Let $E'=\bigcup_{i=1}^k \C_{\alpha_i}$.  Then $h(E')=E'$, so $E$ is the disjoint union
\[
(E\setminus E')\cup \bigcup_{i=1}^k
h(\C_{\alpha_i}).
\]
Thus it suffices to prove that $gh^{-1}$ agrees with some element of $V_{\Gamma,E}$ on each of these sets.  Since $h$ is the identity on $E\setminus E'$ and $g$ agrees with some element of $V_{\Gamma,E}$ on $E\setminus E'$, we already know that this holds for~$E\setminus E'$.

Now consider one of the sets $h(\C_{\alpha_i})$. We know that $\alpha_i$ is part of the nuclear code for some $h_j$, and if $r$ is a rectifier for $h_j$ then $(r_ig)|_{\alpha_i}=p_i=(rh_j)|_{\alpha_i}$.  It follows that $(r_ig)(rh_j)^{-1}=r_igh_j^{-1}r^{-1}$ agrees with the canonical similarity $\C_{\overline{rh_j}(\alpha_i)}\to \C_{\overline{r_ig}(\alpha_i)}$ on $rh_j(\C_{\alpha_i})$. Since $r,r_i\in V_{\Gamma,E}$, it follows that $gh_j^{-1}$ agrees with some element of $V_{\Gamma,E}$ on $h_j(\C_{\alpha_i})$.  But $h$ agrees with $h_j$ on $\C_{\alpha_i}$, so it follows that $gh^{-1}$ agrees with some element of $V_{\Gamma,E}$ on $h(\C_{\alpha_i})$.  We conclude that $gh^{-1}\in V_{\Gamma,E}$, so (i) holds.
\end{proof}

\begin{proposition}[Existence of normalish forms]\label{prop:NormalishFormsExist}
Let $g\in G$, let $A$ be a code, and suppose that:
\begin{enumerate}
    \item Each local action $g|_{\alpha}\;(\alpha\in A)$ lies in $\Nuc$.\smallskip
    \item If $A$ is incomplete, then $g$ agrees with some element of $V_{\Gamma,E}$ on the complement of\/~$\bigcup_{\alpha\in A} \C_{\alpha}$. \smallskip
    \item If $A$ is complete, then\/ $\sum_{\alpha\in A} g|_\alpha$ is not a generator for\/ $\ker(\partial_1)$.\smallskip
\end{enumerate}
Then there exists a normalish form for $g$ whose nuclear code is $A$.
\end{proposition}
\begin{proof}
Let $E'=\bigcup_{\alpha\in A} \C_\alpha$. Since $g$ agrees with an element of $V_{\Gamma,E}$ on the complement of $E'$, we know that
\begin{multline*}
\qquad \class(g(E'))=\class(E)-\class(g(E\setminus E')) \\ = \class(E)-\class(E\setminus E') = \class(E')\qquad
\end{multline*}
in $\Classes(V_{\Gamma,E})$, so the sum $P = \sum_{\alpha\in A} g|_\alpha$ lies in $\ker(\partial_1)$.  Thus $P=P_1+\cdots+ P_n$ for some generators $P_i$ of $\ker(\partial_1)$, so we can partition $A$ into sets $A_1,\ldots,A_n$ such that $\sum_{\alpha\in A_i} g|_{\alpha}=P_i$ for each~$i$. 

Let $C_i=\bigcup_{\alpha\in A_i} \C_{\alpha}$ for each~$i$.  By condition (iii), either $n>1$, or $n=1$ and $C_1\ne E$, and therefore in either case each $C_i$ is a proper subset of~$E$.  By \cref{cor:MappingConesWithV}, we can construct nuclear generators $h_1,\ldots,h_n$ such that each $h_i$ has nuclear code~$A_i$ and nuclear states $\{g|_{\alpha}\}_{\alpha\in A_i}$. Then by \cref{lem:RecognizingNormalish} there exists $f\in V_{\Gamma,E}$ such that $fh_1\cdots h_n$ is the desired normalish form for $g$.
\end{proof}

\begin{corollary}\label{cor:EveryElementNormalishForm}
Every element of $G$ has a normalish form.
\end{corollary}
\begin{proof}
Let $g\in G$.  By \cref{cor:rsg_partition}, there exists a finite partition of $E$ into cones $\{C_{\alpha}\}_{\alpha\in A}$ such that each $g|_\alpha\in \Nuc$.  Refining this partition further, we can make $|A|$ large enough so that the element $\sum_{\alpha\in A} g|_{\alpha}$ of $\ker(\partial_1)$ is not one of the generators for $\ker(\partial_1)$. Then by \cref{prop:NormalishFormsExist}, there exists a normalish form for $g$ with nuclear code~$A$.
\end{proof}

\begin{proof}[Proof of \cref{prop:GFinitelyGenerated}]
By \cref{cor:EveryElementNormalishForm}, every element of $G$ has a normalish form.  Since every nuclear generator is a conjugate of a model nuclear generator by an element of $V_{\Gamma,E}$, we see that $G$ is generated by the model nuclear generators together with the elements of~$V_{\Gamma,E}$. But $V_{\Gamma,E}$ is finitely generated by \cref{cor:V_fp}, so $G$ is finitely generated.
\end{proof}

\subsection{An infinite presentation}\label{ssec:infinitepres}

In this subsection we describe an infinite presentation for $G$ with respect to the nuclear generators and the elements of $V_{\Gamma,E}$. We will use this presentation in \cref{ssec:finitepres} to derive a finite presentation for~$G$.

First we need some terminology. We say that a code $B$ is a \newword{refinement} of a code $A$ if $\bigcup_{\alpha\in A}\C_\alpha=\bigcup_{\beta\in B}\C_\beta$ and for every $\beta\in B$ there exists $\alpha\in A$ with $\C_\beta\subseteq \C_\alpha$.  We say that $B$ is an \newword{elementary refinement} of $A$ if there exists $\alpha\in A$ such that
\[
B = \bigl(A\setminus\{\alpha\}\big) \cup \{\beta_1,\ldots,\beta_k\}
\]
where $\C_{\beta_1},\ldots,\C_{\beta_k}$ are the maximal proper subcones of $\C_\alpha$.  Note that any refinement can be realized by a sequence of elementary refinements.

Next, if $h$ and $h'$ are nuclear generators with nuclear codes $A$ and $A'$, respectively, we say that an element $c\in V_{\Gamma,E}$ is a \newword{rigid conjugator} from $h$ to $h'$ if $c$ maps $A$ to $A'$ in a way that preserves the corresponding nuclear states.  Note then that $h'=chc^{-1}$.

Finally, we say that a normalish form $fh_1\cdots h_n$ is \newword{exact} if $f=1$.

\begin{definition}[An infinite set of relations for $G$]\label{def:RelInf}
Let $\mathrm{Rel}_{\mathrm{inf}}$ be the following set of relations in $G$, with respect to the generating set consisting of all elements of $V_{\Gamma,E}$ and all nuclear generators.
\begin{description}

\listlabel{VRel} All relations in $V_{\Gamma,E}$.\smallskip

\listlabel{VNuc} All true relations of the form $h=f$, where $h$ is a nuclear generator and $f\in V_{\Gamma,E}$. \smallskip

\listlabel{DisjComm} All relations of the form $[h_1,h_2]=1$, where $h_1$ and $h_2$ are nuclear generators whose domains of support are disjoint.\smallskip

\listlabel{Conj} All relations $h'=chc^{-1}$, where $h$ and $h'$ are  nuclear generators and $c\in V_{\Gamma,E}$ is a rigid conjugator from $h$ to $h'$.
\smallskip

\listlabel{Refine} For every nuclear generator $h$ and every elementary refinement $B$ of the nuclear code for $h$, a relation of the form $h=fh_1\cdots h_n$, where the right side is a normalish form for $h$ with nuclear code $B$.\smallskip

\listlabel{Inv} For every nuclear generator $h$, a relation $h^{-1}=fh_1\cdots h_n$, where the right side is some normalish form for~$h^{-1}$.\smallskip

\listlabel{Prod} For every exact normalish form $h_1\cdots h_m$, say with nuclear code~$A$, and every nuclear generator $h$ whose nuclear code $B$ is a subset of $A$, a relation $h_1\cdots h_mh^{-1}=fh_1'\cdots h_n'$, where the right side is a normalish form whose nuclear code contains~$A\setminus B$.
\end{description}
\end{definition}

Note that \VRel, \VNuc, and \Conj\ hold in $G$ by definition, and \DisjComm\ clearly holds as well. The relations \Refine, \Inv, and \Prod\ involve some arbitrary choices for the normalish forms on the right-hand sides of the relations, and these relations will hold in $G$ as long as normalish forms satisfying the given conditions exist.  The following proposition verifies this.

\begin{proposition}
The normalish forms required for the right-hand sides of the relations\/ {\normalfont\Refine}, {\normalfont\Inv}, and\/ {\normalfont\Prod} exist.
\end{proposition}

\begin{proof}
For \Inv, any normalish form for $h^{-1}$ suffices, and such a normalish form always exists by \cref{cor:EveryElementNormalishForm}.

For \Refine, suppose $h$ is a nuclear generator with nuclear code~$A$, and let $B$ be an elementary refinement of $A$. We must show that there exists a normalish form for $h$ with nuclear code $B$. Let $r$ be a rectifier for $h$.  Then $(rh)|_\alpha\in \Nuc$ for all $\alpha\in A$, and it follows that $(rh)|_\beta\in\Nuc$ for all $\beta\in B$.  Since $rh$ agrees with $r$ on the complement of $\bigcup_{\beta\in B} \C_\beta$, by \cref{prop:NormalishFormsExist} there exists a normalish form $fh_1\cdots h_n$ for $rh$ whose nuclear code is~$B$.  Then $f'h_1\cdots h_n$ is a normalish form for $h$ with nuclear code~$B$, where $f'=r^{-1}f$.

For \Prod, let $h_1\cdots h_m$ be an exact normalish form for some $g\in G$ with nuclear code~$A$, and let $h$ be a nuclear generator whose nuclear code $B$ is a subset of~$A$. We must prove that there exists a normalish form for $gh^{-1}$ whose nuclear code contains $A\setminus B$.  If $B=A$ then $A\setminus B=\emptyset$, so the claim is just that $h_1\cdots h_mh^{-1}$ has a normalish form, and this is immediate from \cref{cor:EveryElementNormalishForm}. Now suppose $B$ is a proper subset of $A$. For each~$i$, let $A_i$ be the nuclear code for $h_i$, and let $r_i$ be a rectifier for $h_i$. 

Next observe that there exists a unique complete code $\{\nu_1,\ldots,\nu_M\}$ of minimum size such that each $t(\nu_j)$ is a node in~$\Gamma_0$.  In particular, any cone $\C_\nu\subseteq E$ with $t(\nu)$ a node in $\Gamma_0$ is contained in some $\C_{\nu_j}$. Choose an incomplete code $C=\{\xi_{ij}\mid 1\le i\leq m\text{, } 1\le j\leq M\}$ such that each $t(\xi_{ij})=t(\nu_j)$, and for each $1\leq i\leq m$ let $k_i\colon E\to E$ be the (non-surjective) map that sends each $\C_{\nu_j}$ to $\C_{\xi_{ij}}$ by the canonical similarity.  Note that, for any $g'$ in $G$ and any $\alpha$, if the local action $g'|_{\alpha}$ lies in~$\Nuc$ then $t(\overline{g'}(\alpha))$ is a node in $\Gamma_0$. Hence, in this case, $t(\overline{g'}(\alpha))$ has some $\nu_j$ as a prefix, and so each $k_i$ acts as a canonical similarity on $\C_{\overline{g'}(\alpha)}$. In particular $(k_i\circ g')|_{\alpha}=g'|_\alpha$ for each~$i$. 

Now let $r$ be any element of $V_{\Gamma,E}$ that agrees with $k_i\circ r_i$ on each $\bigcup_{\alpha\in A_i\setminus B} h_i(\C_\alpha)$. Such an element exists since $A\setminus B$ and $C$ are both incomplete codes. Note that $r$ satisfies $(rg)|_\alpha = (r h_i)|_\alpha=(r_ih_i)|_\alpha$ for each $i$ and each $\alpha\in A_i\setminus B$.  In particular, $(rg)|_\alpha\in \Nuc$ for all $\alpha\in A\setminus B$, and hence $(rgh^{-1})|_{\alpha}\in \Nuc$ for all $\alpha\in A\setminus B$.  Let $B'$ be a refinement of $B$ so that $(rgh^{-1})|_{\beta}\in \Nuc$ for all $\beta\in B'$, and let $A'=(A\setminus B)\cup B'$.  Then $(rgh^{-1})|_\alpha\in \Nuc$ for all $\alpha\in A'$. By refining $B'$ further we can make sure that $\sum_{\alpha\in A'} (rgh^{-1})|_\alpha$ is not a generator for $\ker(\partial_1)$, so by \cref{prop:NormalishFormsExist} there exists a normalish form $fh_1'\cdots h_n'$ for $rgh^{-1}$ with nuclear code~$A'$.  Then $(r^{-1}f)h_1'\cdots h_n'$ is a normalish form for $gh^{-1}$ whose nuclear code $A'$ contains $A\setminus B$.
\end{proof}

\newcommand{\eqinf}{\equiv_{\mathrm{inf}}}%
Our goal for the remainder of this subsection is to prove that the relations in $\mathrm{Rel}_{\mathrm{inf}}$ give a presentation for~$G$.  If $w$ and $w'$ are words in the nuclear generators and the elements of $V_{\Gamma,E}$, we will write $w \eqinf w'$ if the relation $w=w'$ follows from the relations in $\mathrm{Rel}_{\mathrm{inf}}$.

\begin{lemma}[Refining a normalish form]\label{lem:RefineNormalish}
Let $w$ be a normalish form with nuclear code $A$, and let $B$ be any refinement of $A$.  Then there exists a normalish form $w'$ with nuclear code $B$ such that $w\eqinf w'$.
\end{lemma}

\begin{proof}
It suffices to consider the case where $B$ is an elementary refinement of~$A$, say
\[
B = \bigl(A\setminus\{\alpha_0\}\bigr)\cup \{\beta_1,\ldots,\beta_k\}.
\]
Suppose $w$ is $fh_1\cdots h_n$.  By \DisjComm, we may permute $h_1,\ldots,h_n$ freely, so suppose $\alpha_0$ is part of the nuclear code $A_1$ for $h_1$.  Let $B_1=\bigl(A_1\setminus\{\alpha_0\}\bigr)\cup \{\beta_1,\ldots,\beta_k\}$.  Then $B_1$ is an elementary refinement of~$A_1$, so by \Refine\ we have $h_1\eqinf f'h_1'\cdots h_m'$ for some normalish form $f'h_1'\cdots h_m'$ with nuclear code $B_1$.  Then
\[
w \eqinf f'' h_1'\cdots h_m'h_2\cdots h_n,
\]
where $f''=ff'$, and the right side is the desired normalish form $w'$ with nuclear code~$B$.
\end{proof}

We will say that a normalish form is \newword{complete} if its nuclear code is complete, and \newword{incomplete} otherwise.

\begin{lemma}[Completing a normalish form]\label{lem:CompletingNormalish}
If $w$ is any normalish form, there exists a complete normalish form $w'$ such that $w\eqinf w'$.
\end{lemma}

\begin{proof}
Observe that if $v$ is any node of $\Gamma_0$ and $C_\beta$ is a cone with $t(\beta)=v$, then the identity element of $G$ can be viewed as a nuclear generator $e_\beta$ with nuclear code $\{\beta\}$.  It follows from \VNuc\ that $e_\beta\eqinf 1$.

Now suppose $w=fh_1\cdots h_n$ is an incomplete normalish form. Let $A$ be its nuclear code, and let $E'=\bigcup_{\alpha\in A} \C_{\alpha}$.  We can partition $E\setminus E'$ into finitely many cones $\C_{\beta_1},\ldots,\C_{\beta_k}$ such that each $t(\beta_i)$ is a node of~$\Gamma_0$.  Then $w\eqinf fh_1\cdots h_n e_{\beta_1}\cdots e_{\beta_k}$, where the right side is a complete normalish form $w'$ with nuclear code $A\cup\{\beta_1,\ldots,\beta_k\}$.
\end{proof}

\begin{lemma}[Checking for the identity]\label{lem:TrivialNormalish}
If $w$ is a normalish form for the identity, then $w\eqinf 1$.
\end{lemma}

\begin{proof}
Suppose $w=fh_1\cdots h_n$ is a normalish form for the identity, and let $C_i$ be the domain of support for each~$h_i$.  Then each $h_i$ agrees with $f^{-1}$ on $C_i$ and is the identity elsewhere, so each $h_i$ is equal to some element  $f_i\in V_{\Gamma,E}$.  By \VNuc, it follows that $h_i\eqinf f_i$ for each $i$, so $w\eqinf ff_1\cdots f_n$. Since $ff_1\cdots f_n=1$, it follows from \VRel\ that $w\eqinf 1$.
\end{proof}

\begin{lemma}[Multiplying by an element of $V_{\Gamma,E}$]\label{lem:MultiplyByf}
For any normalish form $w$ and any $f\in V_{\Gamma,E}$, there exists a normalish form $w'$ such that $wf\eqinf w'$.
\end{lemma}

\begin{proof}
Let $A$ be the nuclear code for $w$. We can choose a refinement $B$ of $A$ such that $f^{-1}$ acts as a canonical similarity on $\C_\beta$ for each $\beta\in B$.  By \cref{lem:RefineNormalish}, there exists a normalish form $w_B=f'h_1\cdots h_n$ with nuclear code $B$ such that $w\eqinf w_B$.  By \Conj, there exists for each $h_i$ a nuclear generator $h_i'$ such that $f^{-1}h_if\eqinf h_i'$.  Then $w_Bf = f'h_1\cdots h_nf\eqinf f'fh_1'\cdots h_n'$, and combining $f'f$ using a relation from \VRel\ gives the desired normalish form $w'$.
\end{proof}

\begin{lemma}[Multiplying by a nuclear generator]\label{lem:MultiplyNormalish}
If $w$ is a normalish form and $h$ is a nuclear generator, then there exists a normalish form $w'$ such that $wh\eqinf w'$.
\end{lemma}

\begin{proof}
Suppose $w=f h_1\cdots h_n$.  By \cref{lem:CompletingNormalish}, we may assume that $w$ is complete. By \Inv\ and \cref{lem:CompletingNormalish} we can write $h^{-1}$ as a complete normalish form, say $f' h_1'\cdots h_m'$. Since our two normalish forms are complete, their nuclear codes have a common refinement, so by \cref{lem:RefineNormalish} we may assume that $w$ and $f'h_1'\cdots h_m'$ have the same nuclear code~$A$. After applying \DisjComm, our goal is to put
\[
wh^{-1} \eqinf f h_1\cdots h_n (h_1')^{-1}\cdots (h_m')^{-1}(f')^{-1}
\]
into normalish form. The nuclear code $B_1$ of $h_1'$ is a subset of~$A$. Thus, using \Prod\ and \VRel, we can put the subword $f h_1\cdots h_n (h_1')^{-1}$ into normalish form while keeping $A\setminus B_1$ a subset of the overall nuclear code. The nuclear code $B_2$ of $h_2'$ is a subset of $A\setminus B_1$, so we can once again apply \Prod\ and \VRel, and continue in this way (since the domains of support of the $h_i'$ are pairwise disjoint) until the entire word equals some normalish form times $(f')^{-1}$. Now we conclude using \cref{lem:MultiplyByf}.
\end{proof}

\begin{proposition}\label{prop:RelInfSuffices}
The relations of\/ $\mathrm{Rel}_{\mathrm{inf}}$ define an infinite presentation for~$G$.
\end{proposition}
\begin{proof}
Let $w$ be any word for the identity.  Applying \Inv\ to any inverses of nuclear generators occurring in $w$, we obtain a word $w'$ without any inverses of nuclear generators such that $w\eqinf w'$.  Applying \cref{lem:MultiplyByf} and \cref{lem:MultiplyNormalish} inductively, we obtain a normalish form $w''$ such that $w'\eqinf w''$.  Now since $w''$ is a normalish form for the identity, \cref{lem:TrivialNormalish} implies that $w''\eqinf 1$, and hence $w\eqinf 1$.
\end{proof}

\subsection{A finite presentation}\label{ssec:finitepres}

In this subsection we use the infinite presentation from \cref{ssec:infinitepres} to derive a finite presentation for~$G$, using the finite generating set consisting of the model nuclear generators $h_P$ and a finite generating set for $V_{\Gamma,E}$ (see \cref{prop:GFinitelyGenerated}).

First, if $h$ is a nuclear generator, define the \newword{rigid centralizer} of $h$ to be the group of all rigid conjugators from $h$ to itself.

\begin{proposition}\label{prop:rigid_cent_fg}
The rigid centralizer of any nuclear generator is finitely generated.
\end{proposition}

\begin{proof}
Let $h$ be a nuclear generator with nuclear code~$A$ and domain of support~$E'$.  By definition, any rigid conjugator from $h$ to $h$ must permute the cones $\{\C_\alpha\}_{\alpha\in A}$ by canonical similarities.  In particular, since $A$ is finite, the group $\Fix(E')$ of elements of $V_{\Gamma,E}$ that fix $E'$ pointwise has finite index in the rigid centralizer.  But $\Fix(E')$ is isomorphic to $V_{\Gamma,E\setminus E'}$, and is therefore finitely generated by \cref{cor:V_fp}.
\end{proof}

Next, we say that two exact normalish forms $h_1\cdots h_n$ and $h_1'\cdots h_n'$ are \newword{rigidly conjugate} if there exists a single element $c\in V_{\Gamma,E}$ that is a rigid conjugator from $h_i$ to $h_i'$ for each~$i$.

Note that every nuclear generator is rigidly conjugate to a model nuclear generator, so there are only finitely many rigid conjugacy classes of nuclear generators.  The following proposition generalizes this.

\begin{proposition}[Criterion for rigid conjugacy]\label{prop:ConjugateNormalish}
Two exact normalish forms $h_1\cdots h_n$ and $h_1'\cdots h_n'$ are rigidly conjugate if and only if they are either both complete or both incomplete, and $h_i'$ is rigidly conjugate to $h_i$ for each~$i$.
\end{proposition}

\begin{proof}The forward direction is trivial.  For the converse, suppose the given normalish forms satisfy the given conditions.  For each $i$, let $c_i$ be a rigid conjugator from $h_i$ to $h_i'$. Since the nuclear codes for our normalish forms are either both complete or both incomplete, there exists $c\in V_{\Gamma,E}$ such that $c$ agrees with each $c_i$ on the domain of support of~$h_i$.  Then $c$ is the desired rigid conjugator from $h_1\cdots h_n$ to $h_1'\cdots h_n'$.
\end{proof}

\begin{corollary}\label{cor:FinitelyManyNormalishForms}
For any $N\in\N$, there are only finitely many rigid conjugacy classes of exact normalish forms $h_1\cdots h_n$ with $n\leq N$.
\end{corollary}

\begin{proof}
Any nuclear generator is rigidly conjugate to a model nuclear generator, so there are only finitely many rigid conjugacy classes of nuclear generators. 
 The result follows easily from this and \cref{prop:ConjugateNormalish}.
\end{proof}

Now, fix an $N\in\N$ so that each generator for $\ker(\partial_1)$ has at most $N$ summands, and choose representatives for the rigid conjugacy classes of exact normalish forms $h_1\cdots h_n$ with $n\leq N$.  We refer to these finitely many representatives as \newword{model normalish forms}.

We are now ready to state the relations in our finite presentation for~$G$. For convenience, we actually state another infinite presentation, with generators consisting of the elements of $V_{\Gamma,E}$ together with all nuclear generators.  However, the presentation we give will easily reduce to a finite presentation through Tietze transformations. (See \cite[Section~1.5]{magnus66} for an account of using Tietze transformations with infinite presentations.)

\begin{definition}[Defining relations for $G$]\label{def:RelFin}
Let $\Rel_{\mathrm{fin}}$ be the following set of relations.

\begin{description}
\listlabel{VRel} All the relations in $V_{\Gamma,E}$.\smallskip

\listlabel{Model} For each nuclear generator $h$ that is not a model nuclear generator, one relation of the form $h=ch_Pc^{-1}$, where $h_P$ is a model nuclear generator and $c$ is a rigid conjugator from $h_P$ to~$h$.\smallskip

\listlabel{VNucFin} Every true relation of the form $h_P=f$, where $h_P$ is a model nuclear generator and $f\in V_{\Gamma,E}$.\smallskip

\listlabel{CentConjFin} For each model nuclear generator~$h_P$, relations $c_ih_Pc_i^{-1}=h_P$, where $\{c_1,\ldots,c_n\}$ is some finite generating set for the rigid centralizer of~$h_P$ (see \cref{prop:rigid_cent_fg}).\smallskip

\listlabel{DisjCommFin} For each pair $h_P,h_{P'}$ of model nuclear generators, one relation $[h_P,h]=1$, where $h$ is some nuclear generator that is rigidly conjugate to $h_{P'}$, and such that $h_Ph$ is an incomplete normalish form.\smallskip

\listlabel{RefineFin} For each model nuclear generator $h_P$ and each elementary refinement $B$ of the nuclear code for $h_P$, one corresponding relation $h_P=fh_1\cdots h_n$ of type \Refine.\smallskip

\listlabel{InvFin} For each model nuclear generator $h_P$, one corresponding relation $h_P^{-1}=fh_1\cdots h_n$ of type \Inv.\smallskip

\listlabel{ProdFin} For each model normalish form $h_1\cdots h_m$, say with nuclear code $A$, and each nuclear generator $h$ whose nuclear code $B$ is a subset of $A$, one corresponding relation $h_1\cdots h_mh^{-1}=fh_1'\cdots h_n'$ of type \Prod.
\end{description}
\end{definition}

Note that all of these sets of relations are finite except for \VRel\ and \Model, which we plan to simplify with Tietze transformations.  In particular, we should point out that for \ProdFin, even though $h$ is not necessarily a model nuclear generator, nonetheless there can be only finitely many such $h$ whose nuclear code is a subset of~$A$.  Note also that all of the relations in $\Rel_{\mathrm{fin}}$ are contained in $\Rel_{\mathrm{inf}}$, with \Model\ and \CentConjFin\ both being subsets of \Conj.

Our goal is to prove that there exists a choice of infinite presentation $\Rel_{\mathrm{inf}}$ as described in \cref{ssec:infinitepres} so that all of the relations in $\Rel_{\mathrm{inf}}$ follow from those in $\Rel_{\mathrm{fin}}$.  In particular, this will involve certain choices of the relations in \Refine, \Inv, and \Prod\  in $\Rel_{\mathrm{inf}}$ so that they follow from those in $\Rel_{\mathrm{fin}}$.  This will prove that $\Rel_{\mathrm{fin}}$ is a set of defining relations for~$G$.  We will then use Tietze transformations to deduce that $G$ is finitely presented.
\newcommand{\eqfin}{\equiv_{\mathrm{fin}}}%
For the remainder of this section, if $w$ and $w'$ are words, write $w\eqfin w'$ if the relations of $\Rel_{\mathrm{fin}}$ imply that~$w=w'$.

\begin{lemma}\label{lem:HaveVNuc}
The relations in\/ {\normalfont\VNuc} follow from those in\/ $\Rel_{\mathrm{fin}}$.
\end{lemma}
\begin{proof}
Let $h$ be a nuclear generator, and suppose $h=f$ for some $f\in V_{\Gamma,E}$.  By \Model, we know that $h\eqfin ch_Pc^{-1}$ for some model nuclear generator $h_P$ and some $c\in V_{\Gamma,E}$.  Then $h_P$ also lies in $V_{\Gamma,E}$, so by \VNucFin, $h_P\eqfin f'$ for some $f'\in V_{\Gamma,E}$.  Then $h\eqfin cf'c^{-1}$, and by \VRel\ we have $cf'c^{-1}\eqfin f$.
\end{proof}

\begin{lemma}\label{lem:HaveConj}
The relations in\/ {\normalfont\Conj} follow from those in\/ $\Rel_{\mathrm{fin}}$.
\end{lemma}
\begin{proof}
Let $h_2=ch_1c^{-1}$ be a relation in \Conj, where $h_1,h_2$ are nuclear generators and $c$ is a rigid conjugator from $h_1$ to $h_2$.  Then $h_1$ and $h_2$ have the same nuclear states, so they are rigidly conjugate to the same model nuclear generator $h_P$.  By \Model, we know that $h_1\eqfin d_1h_Pd_1^{-1}$ and $h_2\eqfin d_2h_Pd_2^{-1}$ for some rigid conjugators $d_1$ and $d_2$. (In the special case where $h_i=h_P$, there is no such relation in \Model, but in this case we can take $d_i=1$.) Then
\[
h_P = d_2^{-1}h_2d_2 = d_2^{-1}ch_1c^{-1}d_2 = d_2^{-1}cd_1h_Pd_1^{-1}c^{-1}d_2 \text{,}
\]
so $d_2^{-1}cd_1$ is in the rigid centralizer of $h_P$.  By \CentConjFin, we know that $[h_P,c_i]\eqfin 1$ for each generator $c_i$ of this rigid centralizer, and by \VRel\ it follows that $[h_P,d_2^{-1}cd_1]\eqfin 1$, so $d_2h_Pd_2^{-1}\eqfin cd_1h_Pd_1^{-1}c^{-1}$.  Then
\[
h_2 \eqfin d_2h_Pd_2^{-1} \eqfin cd_1h_Pd_1^{-1}c^{-1}
\eqfin  
ch_1c^{-1}
\]
as desired.
\end{proof}

\begin{lemma}\label{lem:HaveRefine}
We can choose the relations in\/ {\normalfont\Refine} to follow from those in\/~$\Rel_{\mathrm{fin}}$.
\end{lemma}
\begin{proof}
Let $h$ be a nuclear generator with nuclear code~$A$, and let $B$ be an elementary refinement of $A$.  From \Model, we have $h\eqfin ch_Pc^{-1}$, where $h_P$ is a model nuclear generator and $c$ is a rigid conjugator from $h_P$ to $h$. Let $A'$ be the nuclear code for $h_P$, so $c$ maps $A$ to $A'$.  Let $B'=\{\overline{c}(\beta) \mid \beta \in B\}$, so $B'$ is an elementary refinement of $A'$, and $c$ maps $B$ to $B'$.  From \RefineFin, we have a relation $h_P \eqfin f'w'$, where $w'$ is an exact normalish form with nuclear code $B'$. Then $c$ is a rigid conjugator from $w'$ to some exact normalish form $w$ with nuclear code~$B$. It follows from \Conj\ and \cref{lem:HaveConj} that $w\eqfin cw'c^{-1}$, and hence
\[
h\eqfin ch_Pc^{-1} \eqfin cf'w'c^{-1} \eqfin cf'c^{-1}w.
\]
Combining $cf'c^{-1}$ using \VRel, the right side becomes a normalish form with nuclear code $B$, as desired.
\end{proof}

\begin{lemma}\label{lem:HaveDisjComm}
The relations in\/ {\normalfont\DisjComm} follow from those in\/~$\Rel_{\mathrm{fin}}$.
\end{lemma}
\begin{proof}
Let $[h_1,h_2]=1$ be a relation in \DisjComm, where $h_1$ and $h_2$ are nuclear generators whose domains of support $E_1,E_2$ are disjoint.

Suppose first that $E_1\cup E_2\ne E$.  For $i=1,2$, let $h_{P_i}$ be a model nuclear generator that is rigidly conjugate to $h_i$.  By \DisjCommFin, we know that $[h_{P_1},h]\eqfin 1$, where $h$ is some nuclear generator that is rigidly conjugate to $h_{P_2}$, and such that $h_{P_1}h$ is an incomplete normalish form.  Since $h_1h_2$ is an incomplete normalish form, by \cref{prop:ConjugateNormalish} there exists $c\in V_{\Gamma,E}$ that rigidly conjugates $h_{P_1}h$ to $h_1h_2$.  Then $h_1\eqfin ch_{P_1}c^{-1}$ and $h_2\eqfin chc^{-1}$ by \Conj\ and \cref{lem:HaveConj}, so it follows that $[h_1,h_2]\eqfin 1$.

All that remains is the case where $E_1\cup E_2=E$.  In this case, let $A$ be the nuclear code for $h_2$, and let $B$ be an elementary refinement of~$A$. By \Refine\ and \cref{lem:HaveRefine}, we have $h_2\eqfin f'h_1'\cdots h_n'$, where the right side is some normalish form with nuclear code $B$. Note that $f'$ must be supported on $E_2$, so $[h_1,f']\eqfin 1$ by \Conj\ and \cref{lem:HaveConj}.  If $n\geq 2$, then $[h_1,h_i']\eqfin 1$ for each $i$ by the argument above, and therefore $[h_1,h_2]\eqfin 1$.  Now suppose $n=1$, so it suffices to prove that $[h_1,h_1']\eqfin 1$. Note that the nuclear code for $h_1'$ is strictly larger than the nuclear code for~$h_1$. Since there is a maximum size for the nuclear code of a nuclear generator, we can continue this process until \Refine\ gives us a normalish form with two or more nuclear generators, at which point we have reduced to the $n\ge 2$ case above, so we are done.
\end{proof}

\begin{lemma}[Refining again]\label{lem:RefineNormalishAgain}
Let $w$ be a normalish form with nuclear code $A$, and let $B$ be any refinement of $A$.  Then there exists a normalish form $w'$ with nuclear code $B$ such that $w\eqfin w'$.
\end{lemma}
\begin{proof}
By Lemmas~\ref{lem:HaveRefine} and~\ref{lem:HaveDisjComm}, the relations \Refine\ and \mbox{\DisjComm} follow from those in\/~$\Rel_{\mathrm{fin}}$.  Thus, applying the proof of \cref{lem:RefineNormalish} verbatim with $\eqfin$ in place of $\eqinf$ yields the desired~$w'$.
\end{proof}

\begin{lemma}\label{lem:HaveInv}
We can choose the relations in\/ {\normalfont\Inv} to follow from those in\/~$\Rel_{\mathrm{fin}}$.
\end{lemma}
\begin{proof}
Let $h$ be a nuclear generator.  From \Model, we have $h\eqfin ch_Pc^{-1}$ for some model nuclear generator $h_P$, where $c$ is a rigid conjugator from $h_P$ to $h$. From \InvFin, we have $h_P^{-1}\eqfin w$, where $w$ is some normalish form with nuclear code $A$.  Let $B$ be a refinement of $A$ such that $c^{-1}$ acts as a canonical similarity on $\C_\beta$ for each $\beta\in B$. By \cref{lem:RefineNormalishAgain}, there exists a normalish form $w'$ with nuclear code $B$ such that $w\eqfin w'$.  Then $c^{-1}$ rigidly conjugates $w'$ to some normalish form $w''$.  By \Conj\ and \cref{lem:HaveConj} we know that $c^{-1}w'c\eqfin w''$, so $h^{-1}\eqfin ch_P^{-1}c^{-1} \eqfin cwc^{-1} \eqfin cw'c^{-1} \eqfin w''$.
\end{proof}

\begin{lemma}\label{lem:HaveProd}
We can choose the relations in\/ {\normalfont\Prod} to follow from those in\/ $\Rel_{\mathrm{fin}}$.
\end{lemma}
\begin{proof}
Let $w$ be an exact normalish form  $h_1\cdots h_m$ with nuclear code~$A$, and let $h$ be a nuclear generator whose nuclear code $B$ is a subset of~$A$.  Suppose first that $B$ intersects the nuclear code of each~$h_i$.  Since $|B|\leq N$ by the definition of $N$ (see the text after \cref{cor:FinitelyManyNormalishForms}), it follows that $m\leq N$, so $w$ is rigidly conjugate to one of our model normalish forms $w_0$, say with nuclear code~$A_0$.  Let $c$ be a rigid conjugator from $w$ to $w_0$, and note that $c$ also rigidly conjugates $h$ to some nuclear generator $h_0$ whose nuclear code $B_0$ is contained in~$A_0$.  By \Conj\ and \cref{lem:HaveConj}, we know that $w_0h_0^{-1} \eqfin cwh^{-1}c^{-1}$. By \ProdFin, we have $w_0h_0^{-1} \eqfin f_0w_0'$, where $f_0\in V_{\Gamma,E}$ and $w_0'$ is some exact normalish form whose nuclear code contains $A_0\setminus B_0$.

Now observe that $c^{-1}$ acts as a canonical similarity on $\C_{\alpha}$ for each $\alpha\in A_0\setminus B_0$.  Let $A_1$ be a refinement of $A_0$ that contains $A_0\setminus B_0$ and such that $c^{-1}$ acts as a canonical similarity on $C_\alpha$ for each $\alpha\in A_1$.  By \cref{lem:RefineNormalishAgain}, there exists $f_1\in V_{\Gamma,E}$ and an exact normalish form $w_1$ with nuclear code $A_1$ so that $f_0w_0'\eqfin f_1w_1$.  Then $c^{-1}$ rigidly conjugates $w_1$ to some exact normalish form $w_2$ whose nuclear code contains $A\setminus B$, and it follows from \Conj\ and \cref{lem:HaveConj} that $w_2\eqfin c^{-1}w_1c$.  Then
\[
wh^{-1} \eqfin c^{-1}w_0h_0^{-1} c \eqfin c^{-1}f_0w_0'c \eqfin c^{-1}f_1w_1c \eqfin c^{-1}f_1cw_2
\]
and combining the $c^{-1}f_1c$ on the right using \VRel\ yields the desired normalish form.

All that remains is the case where $B$ does not intersect the nuclear code of each $h_i$ in the original normalish form $h_1\cdots h_m$.  By \DisjComm\ and \cref{lem:HaveDisjComm}, we can permute the $h_i$ freely, so we may assume that $B$ intersects the nuclear codes of $h_1,\ldots,h_j$ for some $j$ and does not intersect the nuclear codes of $h_{j+1},\ldots,h_m$. By \DisjComm\ and \cref{lem:HaveDisjComm}, we know that
\[
h_1\cdots h_m h^{-1} \eqfin h_1\cdots h_j h^{-1} h_{j+1}\cdots h_m.
\]
Let $A_1'$ and $A_2'$ be the nuclear codes for $h_1\cdots h_j$ and $h_{j+1}\cdots h_m$, respectively, and note that $B\subseteq A_1'$. By the argument above, there exists a normalish form $w''$ whose nuclear code $A''$ contains $A_1'\setminus B$ such that $h_1\cdots h_jh^{-1}\eqfin w''$. Then $wh^{-1} \eqfin w''h_{j+1}\cdots h_m$, and the word on the right is a normalish form whose nuclear code contains $(A_1'\setminus B)\cup A_2' = A\setminus B$, as desired.
\end{proof}

\begin{proof}[Proof of \cref{thrm:fin_pres}]
By Lemmas~\ref{lem:HaveVNuc}, \ref{lem:HaveConj}, \ref{lem:HaveRefine}, \ref{lem:HaveDisjComm}, \ref{lem:HaveInv}, and \ref{lem:HaveProd}, all of the relations $\Rel_{\mathrm{inf}}$ listed in \cref{def:RelInf} follow from the relations $\Rel_{\mathrm{fin}}$ listed in \cref{def:RelFin}.  By \cref{prop:RelInfSuffices}, the relations in $\Rel_{\mathrm{inf}}$ give a presentation for $G$, and therefore the relations in $\Rel_{\mathrm{fin}}$ do as well.

The relations in $\Rel_{\mathrm{fin}}$ use an infinite generating set consisting of all the elements of $V_{\Gamma,E}$ together with all of the nuclear generators.  Furthermore, the relations \VRel\ and \Model\ are infinite families.  However, for each relation $h=ch_Pc^{-1}$ in \Model, we can use a Tietze transformation to remove this relation as well as the generator~$h$, leaving us with only the finitely many model nuclear generators.  Furthermore, since $V_{\Gamma,E}$ is finitely presented by \cref{cor:V_fp}, we can use (non-elementary) Tietze transformations to remove all but finitely many of the generators from $V_{\Gamma,E}$ and all but finitely many of the relations from \VRel, which leaves us with a finite presentation for~$G$.
\end{proof}

We would like to reiterate \cref{quest:F_infty}, which asks whether all full, contracting RSGs have type $\F_{\!\infty}$, which is stronger than being finitely presented. Note that this is true for the special case of $V_{\Gamma,E}$, by \cref{cor:V_fp}. In terms of a topological proof, for $V_{\Gamma,E}$ one can mimic the ``standard'' approach to proving type $\F_{\!\infty}$ for certain Thompson-like groups. As soon as $\Nuc$ contains non-identity elements however, several steps of the general proof outline break down and do not have clear alternatives. Even for contracting R\"over--Nekrashevych groups, it remains an open question whether they always have type $\F_{\!\infty}$.

\section{Embedding hyperbolic groups into RSGs}\label{sec:hyp}

The main result of \cite{BelkBleakMatucci} is that every hyperbolic group admits a faithful rational representation \cite[Theorem~1]{BelkBleakMatucci}, and the main result of this section is the following improvement:

\begin{theorem}\label{thrm:hyp_to_contracting}
Every hyperbolic group embeds into a full, contracting~RSG.
\end{theorem}

The key improvement here is ``local to global'': roughly speaking, in \cite{BelkBleakMatucci} it was proved that each element individually has only finitely many local actions, and here we prove that there is a common finite set of local actions containing all but finitely many of each element's local actions. In what follows we will often be discussing a group $G$ with some fixed finite generating set and associated word metric $d$.  In such cases we will often identify $G$ with its Cayley graph, and so for example refer to geodesics in $G$.  In this case we write $C(x)$ for the \newword{cone} on $x\in G$, namely $C(x)\coloneqq\{y\in G\mid d(1,y)=d(1,x)+d(x,y)\}$. For $g\in G$ we also write $|g|$ for the word length of $g$, i.e.\ $|g|\coloneqq d(1,g)$.  If $G$ happens to be hyperbolic, then we will also assume a constant of hyperbolicity $\delta>0$.  

In Subsections~\ref{ssec:horofunction} and~\ref{ssec:TypesOfAtoms}, we briefly recall the definition of the horofunction boundary, as well as the tree of atoms and associated machinery from~\cite{BelkBleakMatucci}. In \cref{ssec:free_Z_factor} we prove that the action of $G$ on its horofunction boundary is particularly well-behaved when $G$ has $\Z$ as a proper free factor. Finally, in Subsections~\ref{ssec:contractinglemma} and~\ref{ssec:contracting} we prove that the image of a hyperbolic group in the rational group is contracting as long as the action on the horofunction boundary is well-behaved.

\subsection{The horofunction boundary and the tree of atoms}\label{ssec:horofunction}

Gromov defined a compact boundary $\partial_h X$ for any metric space~$X$, known as the horofunction boundary (see \cite[Section~7.5]{Gromov1987} or \cite[Chapter~II.8]{BrHa}). If $G$ is a group and $d$ is the word metric on $G$ with respect to some finite generating set, then $\partial_h G$ is a compact, totally disconnected, metrizable space. We will define the horofunction boundary in this more restricted context.

Let $F(G,\Z)$ be the abelian group of all integer-valued functions on~$G$, and let $\oF(G,\Z)$ be the quotient of $F(G,\Z)$ by the subgroup of all constant functions. Viewing $F(G,\Z)=\Z^G$ as a topological space with the product topology, we also get a (quotient) topology on $\oF(G,\Z)$.  For each $x\in G$, let $d_x\colon G\to\Z$ be the function $d_x(y)\coloneqq d(x,y)$, and let $\od_x$ denote the image of this function in $\oF(G,\Z)$.  Then the mapping $x\mapsto \od_x$ defines a topological embedding $i\colon G \to \oF(G,\Z)$.

\begin{definition}[Horofunction boundary]
The \newword{horofunction boundary} $\partial_h G$ of $G$ is the set of all limit points of $i(G)$ in $\oF(G,\Z)$.
\end{definition}

A function $h\colon G\to \Z$ whose image in $\oF(G,\Z)$ lies in $\partial_h G$ is known as a \newword{horofunction}.  This terminology comes from hyperbolic geometry, where each point on the boundary of hyperbolic $n$-space has associated horofunctions whose level sets are horospheres.  The horofunctions associated to Gromov's horofunction boundary are similar to, but distinct from, the horofunctions for $\delta$-hyperbolic spaces introduced by Coornaert and Papadopoulos~\cite{Coornaert-Papadopoulos-1}.

It is a fact that $\partial_h G$ is always compact and totally disconnected \cite[Proposition~1.28]{BelkBleakMatucci}, and $G$ acts on $\partial_h G$ by homeomorphisms.  Unlike some other boundaries, the horofunction boundary is not a quasi-isometry invariant, and indeed the homeomorphism type of $\partial_h G$ can depend on the finite generating set chosen for~$G$.  The horofunction boundary is convenient for us primarily because it is totally disconnected, and as a result can sometimes be identified with a clopen subset of a subshift of finite type.

As described in \cite{BelkBleakMatucci}, the horofunction boundary of a group can be realized as the space of ends of a certain infinite, rooted tree called the tree of atoms, which we will now define.  Let $B_n$ denote the $n$-ball in $G$, that is, the ball of radius $n$ centered at the identity, and define an equivalence relation $\sim$ on $G$ by $x\sim y$ if $\od_x$ and $\od_y$ agree on $B_n$.  That is, $x\sim y$ if $d_x-d_y$ is constant on~$B_n$.  It turns out that there are only finitely many equivalence classes \cite[Proposition~3.3]{BelkBleakMatucci}, and these are the \newword{atoms for $\boldsymbol{B_n}$}.  Note that any atom for $B_{n+1}$ must be contained in an atom for~$B_n$.

\begin{definition}[Tree of atoms]
The \newword{tree of atoms} $\A(G)$ of $G$ is the tree with a vertex for each infinite atom of each $B_n$ ($n\ge 0$), and with an edge from an atom of $B_n$ to an atom of $B_{n+1}$ whenever the latter is contained in the former.
\end{definition}

For example, the root of $\A(G)$ is the unique atom of $B_0$, namely all of $G$ (assuming $G$ is infinite).  We denote by $\A_n(G)$ the vertices of $\A(G)$ representing the infinite atoms of $B_n$.  If $A$ is an atom for $B_n$, we denote by $\od_A$ the function on~$B_n$ (up to an additive constant) that is equal to the restriction of $\od_x$ to $B_n$ for all $x\in A$.

Each infinite atom $A\in\A_n(G)$ has a \newword{shadow}
\[
\partial A = \{\oh\in \partial_h G \mid \oh\text{ agrees with }\od_A\text{ on }B_n\}\text{,}
\]
which is a clopen subset of $\partial_h G$.  These form a basis for the topology on $\partial_h G$, and indeed $\partial_h G$ is homeomorphic to the space of ends of $\A(G)$, or equivalently the space of infinite descending paths in $\A(G)$ \cite[Theorem~3.6]{BelkBleakMatucci}. Note that this statement holds for arbitrary groups $G$, and in fact for arbitrary locally finite graphs under the path metric.

\begin{figure}
    \centering
    $\underset{\textstyle\rule{0pt}{14pt}\text{(a)}}{\fbox{\includegraphics{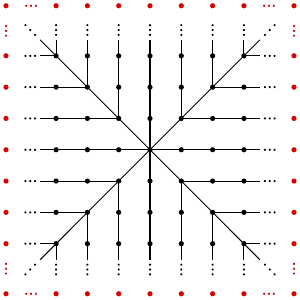}}}$\qquad
    $\underset{\textstyle\rule{0pt}{14pt}\text{(b)}}{\fbox{\includegraphics{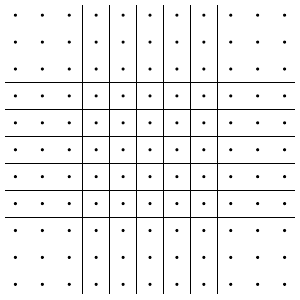}}}$
    \caption{(a) The tree of atoms for $\Z^2$ and the horofunction boundary. (b) The atoms for $B_3\subseteq \Z^2$.}
    \label{fig:Z2TreeOfAtoms}
\end{figure}
\begin{example}\label{ex:Z2}
\cref{fig:Z2TreeOfAtoms}(a) shows the tree of atoms for $\Z^2$ as well as the horofunction boundary $\partial_h\Z^2$ with respect to the generating set $\{(1,0),(0,1)\}$.  The set $\A_0(\Z^2)$ consists of the single root atom $\Z^2$, which we have placed in the center of the figure. For $n\geq 1$, the atoms of $B_n$ are the sets $X\times Y$, where each of $X$ and $Y$ is one of
$\{k \mid k\leq -n\}$, $\{k \mid k\geq n\}$, or a singleton set $\{k\}$ for $-n<k<n$.  For example, the atoms for $B_3$ are shown in \cref{fig:Z2TreeOfAtoms}(b).  Note that $8n$ of the atoms for $B_n$ are infinite, and thus appear in the tree of atoms.  Of these $8n$, four of them have three children each, while the remaining atoms have one child each.  Note that the elements of $\A(\Z^2)$ happen to be in one-to-one correspondence with the elements of~$\Z^2$, but this is just an artifact of this example.

As shown in \cref{fig:Z2TreeOfAtoms}(a), the horofunction boundary in this case is homeomorphic to the complement of $\Z^2$ in $\widehat{\Z}^2$, where $\widehat{\Z}=\Z\cup\{\pm\infty\}$ is the two-point compactification of~$\Z$.  For example, each point $(+\infty,n)\in \partial_h\Z^2$ corresponds to the horofunction $(x,y) \mapsto -x+|y-n|$, and the point $(-\infty,+\infty)\in\partial_h\Z^2$ corresponds to the horofunction $(x,y)\mapsto x-y$.
\end{example}

\begin{example}
If $F_n$ is a free group with a basis as generating set, then the atoms for $F_n$ are precisely the same as the cones, the tree of atoms is isomorphic to the Cayley graph of $F_n$, and the horofunction boundary is homeomorphic to the Gromov boundary~$\partial F_n$, i.e.\ the space of ends of the Cayley graph. In general, Webster and Winchester have proven that if $G$ is a hyperbolic group then the Gromov boundary $\partial G$ is a quotient of $\partial_h G$, with the quotient map $\partial_h G\twoheadrightarrow \partial G$ being finite-to-one~\cite{WeWi}.  Note that $\partial_h G$ is always totally disconnected, whereas $\partial G$ is often a connected space such as a sphere.
\end{example}

\begin{remark}
Even though the group $G$ acts on $\partial_h G$ by homeomorphisms, there is no natural action of $G$ on the tree of atoms or on the atoms themselves; indeed, the image of an atom under translation by an element of $G$ might not even be an atom.
\end{remark}

\begin{remark}
Note that each function $d_x$ is $1$-Lipschitz, meaning that $d_x(y)-d_x(z) \in \{-1,0,1\}$ for every adjacent pair of vertices $y$ and $z$ in $B_n$.  As a result, we can visualize $\od_x$ as a \newword{vector field} on~$B_n$, i.e.\ an assignment of directions to some subset of the edges of~$B_n$.  In particular, if $e$ is an edge connecting $y$ and $z$, we direct $e$ from $y$ to~$z$ if $d_x(y)-d_x(z)=1$, and from $z$ to $y$ if $d_x(y)-d_x(z)=-1$, leaving $e$ undirected if $d_x(y)=d_x(z)$.  Intuitively, all of the edges of $B_n$ are directed to point towards $x$, whenever possible.  Note that two $1$-Lipschitz functions induce the same vector field if and only if they differ by a constant, so such a vector field precisely determines~$\od_x$.  Clearly there are finitely many such vector fields on $B_n$, which is why $B_n$ has only finitely many different atoms. We will not make use of vector fields in this paper, but see \cite{BelkBleakRIMS} for more on this viewpoint, including pictures.
\end{remark}

\subsection{Types of atoms}\label{ssec:TypesOfAtoms}
Let $G$ be a group with a fixed finite generating set and associated word metric $d$, and let $\A(G)$ be the resulting tree of atoms for $G$.  The following definition is taken from \cite[Definition~3.7]{BelkBleakMatucci}.

\begin{definition}[Morphisms of subtrees, same type]\label{def:morphisms}
Given two infinite atoms $A_1\in \A_m(G)$ and $A_2\in \A_n(G)$ we say that an element $g\in G$ is a \newword{morphism} from $A_1$ to $A_2$ if the following hold:
\begin{enumerate}
    \item $gA_1=A_2$.\smallskip
    \item $g(A_1\cap B_{m+k}) = A_2\cap B_{n+k}$ for all $k\geq 0$.\smallskip
    \item For each $k>0$ and each $A_1'\in \A_{m+k}(A_1)$ there exists $A_2'\in \A_{n+k}(A_2)$ such that $gA_1'=A_2'$.
\end{enumerate}
If such a morphism exists, we say that $A_1$ and $A_2$ have \newword{the same type}.
\end{definition}

Of the three conditions above, condition (i) is the most fundamental, and we are not aware of any groups where condition (ii) does not follow from condition (i).  Condition (iii) says that $g$ induces an isomorphism from the subtree of descendant atoms of $A_1$ to the subtree of descendant atoms of $A_2$.  There are examples where a group $G$ has two atoms $A\in \A_n(G)$ and $A'\in \A_{n+1}(G)$ that are equal as subsets of $G$, but these two atoms nonetheless have different types, for instance in the case when $A'$ is the only child of $A$, but $A'$ has multiple children.

\begin{remark}\label{rem:SelfSimilarTree}
Note that there can be only finitely many morphisms between a given pair of atoms, and the set of morphisms is closed under compositions, inverses, and restrictions to child atoms. In \cite{BelkBleakMatucci}, these facts were used to give the tree of atoms of a hyperbolic group the structure of a ``self-similar tree'' but we will not need that terminology here.
\end{remark}

Having the same type is an equivalence relation on atoms, and the corresponding equivalence classes are the \newword{types} of atoms for~$G$.  The following proposition is fundamental to our work on hyperbolic groups.

\begin{proposition}\cite[Corollary~3.28]{BelkBleakMatucci}\label{prop:fin_many_types}
If $G$ is a hyperbolic group, then $\A(G)$ has only finitely many different types of atoms.
\end{proposition}

This phenomenon is not unique to hyperbolic groups.  For example, $\Z^2$ has exactly nine different types of infinite atoms with respect to the generating set $\{(1,0),(0,1)\}$, including the type of the root atom (see \cref{ex:Z2}). As with a group that has finitely many cone types, a group with finitely many different types of atoms must have a rational growth series. Indeed, it is conceivable that a group has finitely many types of atoms if and only if it has finitely many cone types, though neither direction is obvious.

\begin{definition}[Type graph]
If $G$ is a group with finitely many types of (infinite) atoms, the corresponding \newword{type graph} is the finite directed graph $\Gamma$ with one node for each type, and with $n$ directed edges from $v$ to $w$ if each atom of type $v$ has $n$ children of type $w$. The \newword{root node} of $\Gamma$ is the node corresponding to the type of the root atom $G\in \A_0(G)$.
\end{definition}

\begin{figure}
\centering
\quad\includegraphics{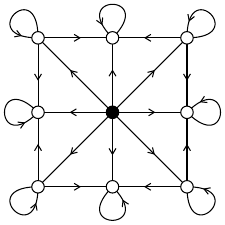}
\qquad\quad
\includegraphics{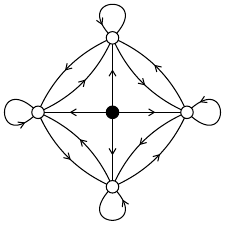}
\caption{Type graphs for $\Z^2$ and $F_2$ with respect to the usual generating sets, where the black dot is the root node.}
    \label{fig:TypeGraphs}
\end{figure}
For example, \cref{fig:TypeGraphs} shows the type graphs for $\Z^2$ and the free group $F_2$ with respect to the usual generating sets.

If $G$ has finitely many types of atoms and $\Gamma$ is the associated type graph, then it is possible to identify $\partial_h G$ with the cone $\C_v\subseteq \Sigma_\Gamma$, where $v$ is the root node of~$\Gamma$.  This identification takes the form of a homeomorphism $\varphi\colon \C_v\to \partial_h G$, which has the following properties:
\begin{enumerate}
    \item For each path $\alpha\in\Cones(v)$, the homeomorphism $\varphi$ maps the cone~$\C_\alpha$ to the shadow of some atom $A_\alpha\in \A_n(G)$ with type corresponding to the node $t(\alpha)$, where $n$ is equal to the length of~$\alpha$.  This determines an isomorphism of trees $\Cones(v)\to \A(G)$.\smallskip
    \item If $\alpha,\beta\in \Cones(v)$ and $t(\alpha)=t(\beta)$, then the canonical similarity $\C_\alpha\to \C_\beta$ induces a morphism from $A_\alpha$ to $A_\beta$.  That is, there exists a morphism $g\in G$ from $A_\alpha$ to $A_\beta$ so that
    \[
    \varphi(\beta\cdot\omega) = g\,\varphi(\alpha\cdot\omega)
    \]
    for all $\omega\in \C_{t(\alpha)}$.
\end{enumerate}
We will refer to such a $\varphi$ as a \newword{system of addresses} for~$\partial_h G$.  If $\varphi$ maps a sequence $\omega\in \C_v$ to a point $\oh\in\partial_h G$, we will say that $\omega$ is the \newword{address} of~$\oh$.  Similarly, if $\alpha\in \Cones(\C_v)$, we will call $\alpha$ the \newword{address} of the atom~$A_\alpha$. Note that $\C_v$ has no empty cones, since every infinite atom in $\mathcal{A}_n$ contains at least one infinite atom in $\mathcal{A}_{n+1}$.
 
\begin{remark}
There is not a single canonical choice for a system of addresses $\varphi$.  Instead, the construction of $\varphi$ involves the choice of a ``rigid structure'' on the tree $\A(G)$, as described in \cite[Section~2.3]{BelkBleakMatucci}.  (See also \cite[Proposition~2.21]{BelkBleakMatucci}, for which the construction of the isomorphism $\Phi$ does not actually require the tree to be ``branching''.)  However, it follows from the proof of \cite[Proposition~2.18]{BelkBleakMatucci} that there are only finitely many different choices for~$\varphi$.
\end{remark}

\begin{proposition}\label{prop:hyp_similarities}
Suppose $G$ has finitely many types of atoms, and let $\varphi\colon \C_v\to \partial_h G$ be a system of addresses as above, so $\varphi$ together with the action of $G$ on $\partial_h G$ induce an action of $G$ on $\C_v$. If this action is rational then the image of $G$ in $\R_{\Gamma,\C_v}$ is an RSG.
\end{proposition}

\begin{proof}
We just need to show that for any $\alpha,\beta\in\Cones(v)$ with $t(\alpha)=t(\beta)$, there exists an element of $G$ mapping $\C_\alpha$ to $\C_\beta$ by the canonical similarity. Indeed, this follows immediately from the fact that canonical similarities on cones in $\C_v$ correspond to morphisms on shadows of atoms in $\partial_h G$.
\end{proof}

The following theorem is a restatement of the main result from~\cite{BelkBleakMatucci}.

\begin{theorem}[\cite{BelkBleakMatucci}]\label{thrm:BBM}
Let $G$ be a hyperbolic group, and let $\varphi\colon \C_v\to \partial_h G$ be a system of addresses for $\partial_h G$.  If $\partial_h G$ has no isolated points, then the induced action of $G$ on\/ $\C_v$ is by rational homeomorphisms.
\end{theorem}

Between \cref{prop:hyp_similarities} and \cref{thrm:BBM}, so far the action of $G$ on $\partial_h G$ gives us a map from $G$ to an RSG, which constitutes progress toward \cref{thrm:hyp_to_contracting}, that every hyperbolic group embeds in a full, contracting RSG. Obviously embedding into a contracting RSG would be enough to embed into a full, contracting RSG, so achieving ``full'' is not an issue. However, there are the following obstacles to further progress:
\begin{enumerate}
    \item If $G$ has non-trivial finite normal subgroups, then the action of $G$ on $\partial_h G$ might not be faithful (see \cref{rmk:horofunction_problems}).\smallskip
    \item We do not know whether $\partial_h G$ has no isolated points, or whether the type graph $\Gamma$ for $G$ has an irreducible core.
\end{enumerate}
We conjecture that (ii) is never a problem if the hyperbolic group $G$ is non-elementary, i.e.\ not virtually cyclic.  In the present paper we will sidestep these difficulties by first embedding $G$ into the free product $G*\Z$.  We prove in the next subsection that for $G*\Z$, the obstacles (i) and (ii) are alleviated, and therefore $G*\Z$ is isomorphic to an RSG. Following this, it will just remain to prove the contracting property.

\newswitch{here}\setTrue{here}
\subsection{Atoms in \texorpdfstring{\except{toc}{$G*{\protect\fakebold{\Z}}$}\except{here}{$G*\Z$}}{G * Z}}\label{ssec:free_Z_factor}

In this subsection we consider the horofunction boundary for a free product $G*\Z$, where $G$ is any non-trivial group.  We always assume that the generating set for $G*\Z$ consists of the generators for $G$ together with a generator $t$ for the $\Z$ factor.  The main result is the following.

\begin{theorem}\label{thrm:FreeProductBoundary}
If $G$ is a non-trivial group, then $\partial_h(G*\Z)$ has no isolated points and $G*\Z$ acts faithfully on $\partial_h(G*\Z)$.  Furthermore, if $G*\Z$ has finitely many types of atoms, then the associated type graph has an irreducible core.
\end{theorem}

\begin{remark}\label{rmk:horofunction_problems}
It is not true in general that the horofunction boundary $\partial_h G$ of an infinite group $G$ has no isolated points, or that $G$ acts faithfully on $\partial_h G$.  For example, the horofunction boundary $\partial_h \Z$ is a two-point space, and $\Z$ acts trivially on~$\partial_h\Z$.  More generally, for all $n\ge 1$ the horofunction boundary of $\Z^n$ has isolated points, though the action is faithful for $n\geq 2$.

There are also non-elementary hyperbolic groups $G$ for which the action of $G$ on $\partial_h G$ is not faithful.  For example, if $G$ has a non-trivial finite normal subgroup $N$ and $S$ is any generating set for $G$ that is closed under multiplication by elements of~$N$, then the corresponding horofunction boundary $\partial_h G$ is naturally homeomorphic to $\partial_h(G/N)$, and in particular $N$ acts trivially on $\partial_h G$.  Note that if no such $N$ exists then the action of $G$ on $\partial_h G$ is automatically faithful, since in this case $G$ acts faithfully on its Gromov boundary $\partial G$, which is a quotient of $\partial_h G$ (see \cite{WeWi}).

We do not know whether the horofunction boundary of a non-elementary hyperbolic group can have isolated points, nor do we know whether the subshift of finite type associated to a non-elementary hyperbolic group always has an irreducible core.
\end{remark}

The proof of \cref{thrm:FreeProductBoundary} occupies the remainder of this subsection.  Let $G$ be a non-trivial group, and for each $n$ let $B_n$ denote the $n$-ball in $G*\Z$. If $w$ is an element of $G*\Z$ that ends with $t$ (in the sense that any minimum-length word for $w$ ends with $t$), then the cone for $w$ is the set
\[
C(w) = \{wh \mid h\in G*\Z\text{ and $h$ does not begin with $t^{-1}$}\}.
\]
By symmetry, a similar description holds for $C(w)$ if $w$ ends with $t^{-1}$.

\begin{lemma}\label{lem:ConesAreAtoms}
If $w$ is any element of $G*\Z$ that ends in~$t$, then $C(w)$ is an atom in $\A_{|w|}(G*\Z)$.  Moreover, any two such atoms have the same type.
\end{lemma}

\begin{proof}
Let $n=|w|$, and let $A$ be the (\textit{a priori} finite or infinite) atom for $B_n$ that contains~$w$.  We claim that $A=C(w)$.  First, observe from the geometry of $G*\Z$ that any geodesic from a point in $C(w)$ to a point in $B_n$ must pass through $w$.  It follows that $\od_x=\od_w=\od_A$ for all $x\in C(w)$, so $C(w)\subseteq A$.  Conversely, if $x\in A$, then since $w,wt^{-1}\in B_n$ and $\od_x$ agrees with $\od_w$ on~$B_n$, the vertex $x$ must be farther from $wt^{-1}$ than from $w$, and therefore $x\in C(w)$.  We conclude that $A=C(w)$, so $C(w)\in \A_n(G*\Z)$.

Now suppose $C(w)$ and $C(w')$ are two such atoms, with $n=|w|$ and $n'=|w'|$. We claim that $w'w^{-1}$ is a morphism from $C(w)$ to $C(w')$.  Clearly $w'w^{-1}$ maps $C(w)$ to $C(w')$, and indeed maps $C(w)\cap B_{n+k}$ to $C(w')\cap B_{n'+k}$ for all $k$.  We claim that $w'w^{-1}$ maps each infinite atom contained in $C(w)$ to an atom of the appropriate level contained in $w'$.  To see this, observe that if $A\in \A_{n+k}(G*\Z)$ is any atom contained in $C(w)$, then
\[
\od_A(p) = \od_A(w) + d(w,p)
\]
for all $p\in B_{n+k}\setminus C(w)$.  In particular, $A$ is completely determined by the restriction of $\od_A$ to $B_{n+k}\cap C(w)$.  Similarly, any atom of $\A_{n'+k}(G*\Z)$ contained in $C(w')$ is determined by the restriction of its distance function to $C(w')\cap B_{n'+k}$.  Thus $w'w^{-1}$ maps $A$ to the atom $A'\in \A_{n'+k}(G*\Z)$ that is contained in $C(w')$ and satisfies
\[
\od_{A'}(p) = \od_A\bigl(w(w')^{-1}p\bigr)
\]
for all $p\in C(w')\cap B_{n'+k}$. We conclude that $w'w^{-1}$ is a morphism from $C(w)$ to $C(w')$, so these two atoms have the same type.
\end{proof}

\begin{lemma}\label{lem:atoms_in_atoms}
Let $G$ be a non-trivial group, and let $A\in \A_{n}(G*\Z)$ for some $n\geq 1$, where $G*\Z=G*\langle t\rangle$.  Then $C(t)$ has a descendant of the same type as~$A$, and $A$ has a descendant of the same type as $C(t)$.
\end{lemma}

\begin{proof}
Note first that if $w$ is any element of $G*\Z$ that ends in $t^{-1}$, then  $C(w)$ is an atom by the same argument as in \cref{lem:ConesAreAtoms}.  Moreover, any two such atoms have the same type.

For the first statement, observe that $C(t^{-1})\in \A_1(G*\Z)$, so every element of $\A_n(G*\Z)$ for $n\geq 1$ is either contained in $C(t^{-1})$ or disjoint from it.  Since $C(t)$ has a descendant of the same type as $C(t^{-1})$, namely $C(tgt^{-1})$ for any non-trivial~$g\in G$, any descendant of $C(t^{-1})$ has the same type as some descendant of $C(t)$.  Suppose then that $A\in\A_n(G*\Z)$ is an atom that is disjoint from $C(t^{-1})$, so $tA\subseteq C(t)$.  We claim that $tA$ is an atom with the same type as $A$.

To see this, note that $t$ maps the complement of $C(t^{-1})$ isometrically to $C(t)$.  Moreover, $t$ maps $B_n\setminus C(t^{-1})$ to $B_{n+1}\cap C(t)$ for each~$n$.  If $a\in A$, it follows that
\[
\od_{ta}(p) = \od_{a}(t^{-1}p)
\]
for all $p\in C(t)\cap B_{n+1}$.  Since every geodesic from $ta$ to $(G*\Z)\setminus C(t)$ passes through $t$, we also know that
\[
\od_{ta}(p) = \od_{ta}(t) + d(p,t)
\]
for all $p\in B_{n+1}\setminus C(t)$.  It follows that $tA$ is contained in a single atom $A'\in\A_{n+1}(G*\Z)$, with
\[
\od_{A'}(p)=\begin{cases}\od_A(t^{-1}p) & \text{if }p\in C(t)\cap B_{n+1}, \\[3pt] \od_A(1) + d(p,t) & \text{if }p\in B_{n+1}\setminus C(t).\end{cases}
\]
Furthermore, if $x$ is any point in $A'$, then $x$ must lie in $C(t)$ since $\od_{A'}(t)<\od_{A'}(1)$.  In this case, it is easy to see that $\od_{t^{-1}x}$ agrees with $\od_A$ on~$B_n$, and therefore $x\in tA$.  This proves that $tA$ is an atom.  Moreover, the same argument applies to any descendant of~$A$, so $t$ maps the descendants of $A$ to descendants of~$tA$.  Since $(tA)\cap B_{n+k+1} = t(A\cap B_{n+k})$ for all $k\geq 0$, we conclude that $t$ is a morphism from $A$ to $tA$, so $tA$ has the same type as~$A$.

For the other direction, let $A\in\A_n(G*\Z)$ for some $n\geq 1$, and let $a$ be a point of $A$.  If $a$ does not end with $t^{-1}$,  then any geodesic from any point in $C(at)$ to $B_n$ must pass through~$a$, and it follows easily that $C(at)\subseteq A$.  If $a$ ends in $t^{-1}$, we can fix a non-trivial element $f\in G$.  Then any geodesic from any point in $C(aft)$ to $B_n$ must pass through $a$, so again $C(aft)\subseteq A$.  In either case, this gives a descendant of $A$ with the same type as~$C(t)$.
\end{proof}

\begin{proof}[Proof of \cref{thrm:FreeProductBoundary}]
We claim first that $\partial_h(G*\Z)$ has no isolated points.  By \cref{lem:atoms_in_atoms}, every infinite atom in $G*\Z$ contains an atom with the same type as $C(t)$.  But $C(t)$ contains at least one pair of disjoint infinite atoms, e.g.\ $C(t^2)$ and $C(tgt)$ for some non-trivial element $g\in G$.  It follows that every infinite atom in $G*\Z$ contains at least one disjoint pair of infinite atoms.  In $\partial_h(G*\Z)$, this means that every basic open set contains at least one pair of disjoint basic open sets, which proves that $\partial_h (G*\Z)$ has no isolated points.

To prove $G*\Z$ acts faithfully on $\partial_h(G*\Z)$, let $w$ be any non-trivial element of $G*\Z$.  Let $v$ be any element of $G*\Z$ that ends with $t$, and does not have a minimum-length word that starts with the same generator (or inverse generator) as a minimum length word for $w$ or $w^{-1}$.  Then $v$ and $wv$ both end with $t$ and $C(v)$ and $C(wv)$ are disjoint.  Then $w$ maps the infinite atom $C(v)$ to the infinite atom $C(wv)$, which is disjoint from $C(v)$, so $w$ acts non-trivially on the horofunction boundary.

Finally, suppose $G*\Z$ has finitely many types of (infinite) atoms.  By \cref{lem:atoms_in_atoms}, every node of the type graph $\Gamma$ has a directed path to the type of $C(t)$, and there is a directed path from the type of $C(t)$ to every other node in $\Gamma$ except for the basepoint (i.e.\ the type of $G*\Z$ itself).  It follows that the induced subgraph $\Gamma_0$ on all nodes but the basepoint is strongly connected. Furthermore, $\Gamma_0$ cannot be a directed cycle since $\partial_h(G*\Z)$ is infinite, so we conclude that $\Sigma_\Gamma$ has an irreducible core.
\end{proof}

\subsection{The contracting lemma}\label{ssec:contractinglemma}

In this subsection we prove a geometric lemma about atoms in hyperbolic groups that is the main content of the proof that hyperbolic groups are contracting. Throughout this subsection, let $G$ be an infinite $\delta$-hyperbolic group, and let $B_n$ denote the $n$-ball in~$G$.

For any $x\in G\setminus B_{n-1}$, let $N(x,B_n)$ be the set of \newword{nearest neighbors} of $x$ in $B_n$.  This can be described in the following equivalent ways:
\begin{enumerate}
    \item It is the set of all $p\in B_n$ so that $d(x,p)\leq d(x,q)$ for all $q\in B_n$.\smallskip
    \item It is the set of all points at which geodesics from $1$ to $x$ intersect the $n$-sphere $S_n\coloneqq B_n\setminus B_{n-1}$.
    \smallskip
    \item It is the set of points in $B_n$ at which $\od_x$ obtains its minimum value.
\end{enumerate}
It follows from (iii) that $N(x,B_n)$ depends only on which atom for $B_n$ contains $x$. Thus, if $A\in\A_n(G)$, we can define
\[
N(A) \coloneqq N(x,B_n)
\]
where $x$ is any point in~$A$.

Our goal in this subsection is to prove the following lemma.

\begin{lemma}[The contracting lemma]\label{lem:contractinglemma}
Let $g\in G$, and let $A\in \A_n(G)$ with $n >  2|g|+39\delta+13$.  Then there exists $A'\in\A(G)$ such that $gA\subseteq A'$ and $N(A')$ lies within the $(18\delta+6)$-neighborhood of $g\,N(A)$.
\end{lemma}

The important thing here is that $18\delta+6$ is a constant.  We will show in \cref{ssec:contracting} that the local action $g|_A$ depends only on the $G$-orbit of the pair $\bigl(g\,N(A),N(A')\bigr)$
together with a finite amount of additional data, from which it will follow that $G$ has finite nucleus.

The proof of \cref{lem:contractinglemma} involves some technical results from \cite{BelkBleakMatucci}.  If $x\in G\setminus B_{n-1}$, the \newword{visible set} $V(x,B_n)$ is the set of all points $p\in B_n$ such that every geodesic $[p,x]$ intersects $B_n$ only at~$p$.  Again, it is possible to prove that $V(x,B_n)$ depends only on which atom for $B_n$ contains~$x$ \cite[Proposition~2.15]{BelkBleakMatucci}.  In particular, if $A\in\A_n(G)$ we can define
\[
V(A) \coloneqq V(x,B_n)
\]
where $x$ is any point of $A$.  The following proposition lists some further properties of $N(A)$ and $V(A)$.

\begin{proposition}\label{prop:PropertiesSets}
Let $x\in G\setminus B_{n-1}$. Then:
\begin{enumerate}
\item \label{part:DiameterN} $N(x,B_n)$ has diameter at most $2\delta$.\smallskip

\item \label{part:PNearN} $N(x,B_n)\subseteq V(x,B_n)\subseteq S_n$, and $V(x,B_n)$ is contained in a $(4\delta+2)$-neighborhood of each point in $N(x,B_n$).
\smallskip

\item \label{part:PropertyV} If $b\in B_n$, then there exists a geodesic $[b,x]$ that contains a point of~$V(x,B_n)$. \smallskip

\item \label{part:JustCheckU} If $A\in\A_n(G)$ and $U$ is a subset of $B_n$ that contains  $V(x,B_n)\cup V(A)$,   then $x\in A$ if and only if $\od_x$ agrees with $\od_A$ on~$U$.
\end{enumerate}
\end{proposition}
\begin{proof}
Statement (i) follows from the standard fact than any two geodesics with the same endpoints synchronously $2\delta$-fellow travel. 
 Statement (ii) follows from Propositions~3.19 and 3.21 of \cite{BelkBleakMatucci}, while statements (iii) and (iv) are Propositions 3.16 and 3.18 in \cite{BelkBleakMatucci}, respectively.
\end{proof}

We now prove two lemmas that we will need for the proof of \cref{lem:contractinglemma}.  For $\of\in\oF(G,\Z)$ and $S\subseteq G$, write
\[
\|\of\|_S \coloneqq \frac{1}{2}\sup\bigl\{|f(s)-f(s')| \;\bigr|\; s,s'\in S\bigr\} 
\]
where $f$ is any representative for $\of$ in $F(G,\Z)$.

\begin{lemma}\label{lem:MappingAtomsIntoAtoms}
Let $g\in G$, let $\epsilon\in\N$, and let $A\in \A_n(G)$, where
\begin{equation}\label{eq:n}
n\geq 2|g|+\delta+2\epsilon.
\end{equation}
Let $U$ be the $(6\delta+2+2\epsilon)$-neighborhood of $g\,N(A)$. If\/ $\bigl\|\od_{1}-\od_{g}\bigr\|_U \leq \epsilon$, then there exists $A'\in \A(G)$ such that $gA\subseteq A'$ and $N(A')$ is contained in $U$.
\end{lemma}

\begin{proof}
Note that by \cref{prop:PropertiesSets}(\ref{part:PNearN}), the set $g\,V(A)$ is contained in the $(4\delta+2)$-neighborhood of $g\,N(A)$, which means that $g\,V(A)\subseteq U$.  Hence, $g\,N(A)\subseteq g\,V(A)\subseteq U\cap gS_n$.

Now, since $\bigl\|\od_{1}-\od_{g}\|_U\leq \epsilon$, there exists $C\in\Z$ so that
\begin{equation}\label{eq:dist}
0 \leq d(u,1)-d(u,g) + C\leq 2\epsilon
\end{equation}
for all $u\in U$. In particular,
\begin{equation}\label{eq:SubsetUnion}
U\cap g S_n\subseteq \bigcup_{k\leq i\leq k+2\epsilon} S_i\qquad\text{and}\qquad U\cap S_k\subseteq g B_n
\end{equation}
where $k\coloneqq n-C$. Note that our hypothesis on $n$ ensures that $k$ is a positive integer.  Indeed, since $|d(u,1)-d(u,g)|\leq |g|$, it follows from~(\ref{eq:dist}) that $C\leq |g|+2\epsilon$,
so by (\ref{eq:n}), 
\begin{equation}\label{eq:LowerBoundk}
k = n-C \ge n-|g|-2\epsilon\ge |g|+\delta
\end{equation}
and hence $k>0$.

Now consider a point $x\in gA$.  We claim that $x\notin B_{k-1}$ and $V(x,B_k)\subseteq U$.  For the first statement, we know from (\ref{eq:n}) that $|g|\leq n$, so  $g^{-1}\in B_n$. Since $g^{-1}x\in A$, it follows from \cref{prop:PropertiesSets}(\ref{part:PropertyV}) that some geodesic $[g^{-1},g^{-1}x]$ passes through $V(A)$.  Hence the geodesic $[1,x]\coloneqq g[g^{-1},g^{-1}x]$ passes through a point $y$ of $g\,V(A)$. Then $y\in U\cap gS_n$, so by (\ref{eq:SubsetUnion}) we have that $k\leq |y|\leq k+2\epsilon$.  It follows that $|x|\geq |y|\geq k$, so $x\notin B_{k-1}$.

To prove that $V(x,B_k)\subseteq U$, let $[g,x],[1,x],[1,g]$ be a geodesic triangle. Since $g^{-1}x\in A$, the geodesic $[1,g^{-1}x]\coloneqq g^{-1}[g,x]$ passes through $N(A)$, so $[g,x]$ passes through $g\,N(A)$ at some point~$q$.  Then $q\in  U\cap gS_n$, so by (\ref{eq:SubsetUnion}) we have
\begin{equation}\label{eq:Lengthq}
 k\leq |q|\leq k+2\epsilon.
\end{equation}
Combining this with~(\ref{eq:LowerBoundk}), we deduce that $|q|\geq |g|+\delta$.  In particular, $q$ does not lie within $\delta$ of any interior point of $[1,g]$. Also, $d(q,g)=n>\delta$, so by the thin triangle condition, it follows that $q$ lies within $\delta$ of some point $x'\in[1,x]$.  Let $q'$ be the point at which $[1,x]$ passes through $S_k$, so $q'\in N(x,B_k)$. Since $d(q,x')\leq \delta$ and $\bigl||q|-k\bigr|\leq 2\epsilon$ by (\ref{eq:Lengthq}), it follows from the triangle inequality that $\bigl||x'|-k\bigr|\leq \delta+2\epsilon$.  In particular, $d(x',q')\leq \delta+2\epsilon$, so $d(q,q')\leq 2\delta+2\epsilon$, again by the triangle inequality.  But \cref{prop:PropertiesSets}(\ref{part:PNearN}) tells us that $V(x,B_k)$ lies within a $(4\delta+2)$-neighborhood of $q'$, so $V(x,B_k)$ lies within a $(6\delta+2+2\epsilon)$-neighborhood of~$q$. Since $q\in g\,N(A)$, we conclude that $V(x,B_k)\subseteq U$.
 
We are now ready to prove the desired statements.  Since $A$ is infinite, there is an infinite atom $A'\in\A_k(G)$ that intersects~$gA$. Note that $V(A')\subseteq U$, since $A'$ contains a point of $gA$ and $V(x,B_k)\subseteq U$ for every $x\in gA$.  In particular, $N(A')\subseteq U$.  We claim that $gA\subseteq A'$. Let $x\in gA$, and fix a point $p\in A'\cap gA$. Since $x,p\in gA$, we know that $\od_x$ agrees with $\od_p$ on $gB_n$.
Since $p\in A'$, we know that $\od_p$ agrees with $\od_{A'}$ on $B_k$.  But $U\cap S_k$ is contained in $gB_n$ by (\ref{eq:SubsetUnion}) and is also contained in $B_k$, so $\od_x$ agrees with $\od_{A'}$ on $U\cap S_k$. 
 Since $V(A')\subseteq U\cap S_k$ and $V(x,B_k)\subseteq U\cap S_k$, it follows from \cref{prop:PropertiesSets}(\ref{part:JustCheckU}) that $x\in A'$, and therefore $gA\subseteq A'$.
\end{proof}

\begin{lemma}\label{lem:DistanceFunctionNotChanging}
Let $p\in B_n$, and let $S$ be a subset of $G\setminus B_n$.  If
\[
\mathrm{diam}(S) < 2n-2|p|-4\delta
\]
then $\bigl\|\od_p-\od_{1}\bigr\|_S\leq 6\delta+2$.
\end{lemma}

\begin{proof}
Note that $d_1$ is just word length.  It suffices to prove that
\[
\bigl|\bigl( d(x,p) - |x|\bigr) - \bigl(d(y,p)-|y|\bigr)\bigr| \leq 2(6\delta+2)
\]
for all $x,y\in S$.  Let $x,y\in S$, and fix geodesics $[x,y]$, $[x,p]$, and $[y,p]$. Let $m_{xy}$ be the point on the geodesic $[x,y]$ which is farthest from $x$ but still within $\delta$ of $[x,p]$.  By the thin triangle condition, $m_{xy}$ must also lie within $\delta+1$ of $[y,p]$,
so there exist points $m_{xp}\in [x,p]$ and $m_{yp}\in [y,p]$ so that $d(m_{xy},m_{xp})\leq \delta$ and $d(m_{xy},m_{yp})\leq \delta+1$. Then
\begin{multline*}
\bigl|\bigl(d(x,p)-d(y,p)\bigr) - \bigl(d(x,m_{xy}) - d(y,m_{xy})\bigr)\bigr| \\
= \bigl|d(x,m_{xp}) + d(m_{xp},p) - d(y,m_{yp}) - d(m_{yp},p) - d(x,m_{xy}) + d(y,m_{xy})\bigr| \\
\leq |d(x,m_{xp})-d(x,m_{xy})| + |d(y,m_{yp})-d(y,m_{xy})| +|d(m_{xp},p)-d(m_{yp},p)| \\
\leq d(m_{xp},m_{xy}) + d(m_{yp},m_{xy}) + d(m_{xp},m_{yp}) \\ \leq \delta + (\delta+1) + (2\delta+1) = 4\delta+2.
\end{multline*}
Now fix geodesics $[x,1]$, $[y,1]$, and $[1,p]$.  We know that either $d(x,m_{xy}) \le \frac{1}{2} d(x,y)$ or $d(y,m_{xy}) \le \frac{1}{2}d(x,y)$. Let us assume the first case, the other being analogous. We see that
\[
d(x,m_{xy})+|m_{xy}| \ge |x| \ge n+1
\]
from which it follows that
\[
|m_{xy}| \geq n+1-\frac{d(x,y)}{2} \geq n+1 - \frac{\mathrm{diam}(S)}{2}
\]
so we have
\[
d(m_{xy},m_{yp})+|m_{yp}| \ge |m_{xy}| \ge n+1 - \frac{\mathrm{diam}(S)}{2}
\]
and since $d(m_{xy},m_{yp}) \le \delta+1$ we get

\[
|m_{yp}| \geq n+1-\frac{\mathrm{diam}(S)}{2} - (\delta+1) > |p|+\delta
\]
and therefore $m_{yp}$ does not lie within $\delta$ of $[1,p]$.  Applying the thin triangle condition to $[1,p]\cup [y,p]\cup [y,1]$, there exists a point $m_{y}\in[y,1]$ such that $d(m_{yp},m_{y})\leq \delta$. A similar argument for the triangle $[1,p]\cup [x,p]\cup [x,1]$ yields a point $m_{x}\in [x,1]$ such that $d(m_{xp},m_{x})\leq \delta$.  Then $d(m_{xy},m_{x})\leq 2\delta$ and $d(m_{xy},m_{y})\leq 2\delta+1$, so by the same reasoning as above
\begin{multline*}
\bigl|\bigl(|x|-|y|\bigr) - \bigl(d(x,m_{xy}) - d(y,m_{xy})\bigr)\bigr| \\
\leq d(m_{x},m_{xy}) + d(m_{y},m_{xy}) + d(m_{x},m_{y}) \\
\leq 2\delta + (2\delta+1) + (4\delta+1) = 8\delta+2.
\end{multline*}
By the triangle inequality, it follows that
\begin{multline*}
\bigl|\bigl( d(x,p) - |x|\bigr) - \bigl(d(y,p)-|y|\bigr)\bigr| \\ =
\bigl|\bigl(d(x,p)-d(y,p)\bigr) - \bigl(|x| - |y|\bigr)\bigr| \\ \leq (4\delta+2)+ (8\delta + 2) = 12\delta + 4
\end{multline*}
as desired.
\end{proof}

We are now ready to prove the contracting lemma.

\begin{proof}[Proof of \cref{lem:contractinglemma}]
Let $g\in G$, let $A\in \A_n(G)$ with $n >  2|g|+39\delta+13$, and let $U$ be the $(18\delta+6)$-neighborhood of $g\,N(A)$.  We must prove that there exists $A'\in \A(G)$ such that $gA\subseteq A'$ and $N(A')\subseteq U$.

By \cref{prop:PropertiesSets}(\ref{part:DiameterN}), the diameter of~$g\,N(A)$ is at most $2\delta$, so 
\[
\mathrm{diam}(U)\leq  2\delta + 2(18\delta+6) = 38\delta+12.
 \]
Since $N(A)\subseteq S_n$, we know that $g\,N(A)$ is disjoint from $B_{n-|g|-1}$, so $U$ is disjoint from $B_{n'}$, where $n'=n-|g|-1-(18\delta+6)$. We now have
\begin{align*}
2n'-2|g|-4\delta &\geq 2\bigl(n-|g|-1-(18\delta+6)\bigr) -2|g| -4\delta \\
&= 2n-4|g| - 40\delta - 14
\\
&> 2\bigl(2|g|+39\delta+13\bigr) -4|g| - 40\delta - 14 \\
& =38\delta+12 \geq \mathrm{diam}(U).
\end{align*}
By \cref{lem:DistanceFunctionNotChanging}, it follows that $\bigl\|\od_{1}-\od_{g}\bigr\|_U\leq 6\delta+2$.  Set $\epsilon:=6\delta+2$.  Then $\bigl\|\od_{1}-\od_{g}\bigr\|_U\leq \epsilon$ and
\[
2|g|+\delta + 2\epsilon = 2|g|+13\delta +4 < n.
\]
Since $18\delta+6=6\delta+2+2\epsilon$, it follows from \cref{lem:MappingAtomsIntoAtoms} that there exists $A'\in \A(G)$ such that $gA\subseteq A'$ and $N(A')\subseteq U$.
\end{proof}

\subsection{Proof that \texorpdfstring{$G$}{\textit{G}} has finite nucleus}\label{ssec:contracting}

The goal of this subsection is to prove the following theorem.

\begin{theorem}\label{thrm:contracting}
Let $G$ be a hyperbolic group for which $\partial_h G$ has no isolated points, and let\/ $\Gamma$ be the type graph for $G$ with root node~$r$.  Then the image of $G$ in $\R_{\Gamma,\C_r}$ has finite nucleus.
\end{theorem}

We will make use of the contracting lemma (\cref{lem:contractinglemma}) proven in the last subsection, and we will also need some additional technical results from \cite{BelkBleakMatucci}.

Recall the notation $B_n$ for the $n$-ball in $G$, $S_n$ for the $n$-sphere, and $N(A)$ for the set of nearest neighbors of $A$ in $B_n$ ($A\in \A_n(G)$). Now let
\[
\widehat{N}(A) \coloneqq S_n \cap B_{4\delta+2}(N(A))
\]
where $B_{4\delta+2}(N(A))$ denotes the $(4\delta+2)$-neighborhood of $N(A)$. The following proposition is a version of the main technical result in~\cite{BelkBleakMatucci}.

\begin{proposition}\label{prop:MakeMorphisms}
Let $A,A'\in\A(G)$, let $g\in G$, and suppose that
\begin{enumerate}
    \item $g\,\widehat{N}(A)=\widehat{N}(A')$,\smallskip
    \item $g\,\od_A$ agrees with $\od_{A'}$ on $\widehat{N}(A')$, and\smallskip
    \item $g\,C(x) = C(gx)$ for all $x\in \widehat{N}(A)$.
\end{enumerate}
Then $g$ is a morphism from $A$ to $A'$.
\end{proposition}
\begin{proof}
By \cite[Proposition~3.21]{BelkBleakMatucci}, the set $\widehat{N}(A)$ contains the proximal set for $A$ as defined in \cite[Definition~3.20]{BelkBleakMatucci}.  Thus the given proposition follows from \cite[Proposition~3.27]{BelkBleakMatucci}.
\end{proof}

Now fix a system of addresses $\varphi\colon \C_r\to \partial_h G$, and let $A_\alpha$ denote the atom associated to each $\alpha\in\Cones(\C_r)$.  For $g\in G$, let $g^\varphi$ denote the homeomorphism of $\C_r$ induced by $g$ via $\varphi$.  We write $g|_\alpha\coloneqq (g^\varphi)|_\alpha$ and $\og(\alpha)\coloneqq \overline{g^\varphi}(\alpha)$ for each $\alpha\in \Cones(\C_r)$. Set
\[
G^\varphi = \bigl\{g^\varphi \;\bigr|\; g\in G\bigr\}
\]
and let $\Nuc_G$ denote the nucleus of $G^\varphi$.  Note that $G^\varphi$ may be a proper quotient of~$G$ if the action of $G$ on $\partial_h G$ is not faithful. 

If $A_\alpha$ and $A_\beta$ are atoms with $t(\alpha)= t(\beta)$, then our system of addresses determines a canonical morphism $g$ from $A_\alpha$ to $A_\beta$, which has the property that $g|_\alpha$ is the identity on $\C_{t(\alpha)}$.  If $g$ is an arbitrary morphism from $A_\alpha$ to $A_\beta$, then $g|_{\alpha}$ need not be the identity, but it is still a self-homeomorphism of $\C_{t(\alpha)}$.

\begin{proposition}\label{prop:MorphismGroups}
For $v$ a node in\/ $\Gamma$, set
\[
\mathrm{Mor}(v,G) = \bigl\{g|_\alpha \;\bigr|\; g\text{ is a morphism from $A_\alpha$ to $A_\beta$}\bigr\}\text{,}
\]
for some choice of $A_\alpha,A_\beta\in\A(G)$ of type $v$. Then\/ $\mathrm{Mor}(v,G)$ does not depend on the choice of $A_\alpha$ and $A_\beta$, and is a finite subgroup of $\R_{\Gamma,\C_v}$.
\end{proposition}
\begin{proof}
First note that morphisms are closed under composition and inverses.  That is, if $g$ is a morphism from $A_{\alpha_1}$ to~$A_{\alpha_2}$ and $h$ is a morphism from $A_{\alpha_2}$ to $A_{\alpha_3}$, then $g^{-1}$ is a morphism from $A_{\alpha_2}$ to $A_{\alpha_1}$ and $hg$ is a morphism from $A_{\alpha_1}$ to $A_{\alpha_3}$.  Also note that $(hg)|_{\alpha_1} = h|_{\alpha_2}\circ g|_{\alpha_1}$ since these are homeomorphisms of $\C_v$, i.e., the formula in \cref{lem:restrict_composition} does not require the extra restriction to a local action.  It follows easily that $\mathrm{Mor}(v,G)$ is a subgroup of $\R_{\Gamma,\C_v}$ and it does not depend on the choice of $\alpha$ and $\beta$.

To prove that $\mathrm{Mor}(v,G)$ is finite, let $A_\alpha$ be an atom of type~$v$.   Since $A_\alpha\ne\emptyset$, there exists a $k\in \N$ so that $A_\alpha\cap S_k\ne \emptyset$.  Then any morphism from $A_\alpha$ to $A_\alpha$ must permute the points of $A_\alpha\cap S_k$. There are only finitely many elements of $G$ that permute these points, so there are only finitely many morphisms from $A_\alpha$ to $A_\alpha$, and hence $\mathrm{Mor}(v,G)$ is finite.
\end{proof}

Two local actions $g|_\alpha,h|_\beta\colon \C_v\to \C_w$ are \newword{morphism-equivalent} if there exist $k\in \mathrm{Mor}(v,G)$ and $\ell\in \mathrm{Mor}(w,G)$ such that 
\[
h|_\beta = \ell\circ (g|_\alpha)\circ k.
\]
Morphism-equivalence is the same as the notion of ``equivalent restrictions'' on the self-similar tree of atoms given in \cite{BelkBleakMatucci}.  Clearly each morphism-equivalence class is finite, so if we want to prove that the nucleus $\Nuc_G$ of $G^\varphi$ is finite it suffices to prove that $\Nuc_G$ has only finitely many morphism-equivalence classes.

The proof of rationality in \cite[Section~3.6]{BelkBleakMatucci} hinges on a notion of a ``mapping triple'', which is an ordered triple $(g,A,A')$ satisfying certain properties.  Our proof here will use a different notion of a mapping triple, which we now define.

\begin{definition}[Mapping triple]
A \newword{mapping triple} is an ordered triple $(g,A_\alpha,A_\beta)$ such that
\begin{enumerate}
    \item $g\in G$,\smallskip
    \item $A_\alpha\in \A_n(G)$ for some $n>2|g|+39\delta+13$, and\smallskip
    \item $A_\beta\in \A(G)$ is an atom for which $gA_\alpha\subseteq A_\beta$ and $N(A_\beta)$ is contained in the $(18\delta+6)$-neighborhood of $g\,N(A_\alpha)$.
\end{enumerate}
\end{definition}
By the contracting lemma (\cref{lem:contractinglemma}), for any $g\in G$ and any $A_\alpha\in \A_n(G)$ with $n>2|g|+39\delta+13$, there exists $A_\beta\in\A(G)$ such that $(g,A_\alpha,A_\beta)$ is a mapping triple. In particular, for all $g\in G$, for all but finitely many $\alpha$, there exists a mapping triple of the form $(g,A_\alpha,A_\beta)$.

We say that two mapping triples $(g,A_\alpha,A_\beta)$ and $(h,A_\zeta,A_\eta)$ are \newword{equivalent} if there exists a morphism $\ell$ from $A_\beta$ to $A_\eta$ such that $h^{-1}\ell g$ is a morphism from $A_\alpha$ to $A_\zeta$.

\begin{lemma}\label{lem:mapping_triples_equiv}
If the mapping triples $(g,A_\alpha,A_\beta)$ and $(h,A_\zeta,A_\eta)$ are equivalent then $g|_\alpha$ and $h|_\zeta$ are morphism-equivalent.
\end{lemma}

\begin{proof}
Choose a morphism $\ell$ from $A_\beta$ to $A_\eta$ such that $k\coloneqq h^{-1}\ell g$ is a morphism from $A_\alpha$ to $A_\zeta$.  Since $hkA_\alpha = hA_\zeta$, we have that $\overline{hk}(\alpha) = \oh(\zeta) = \oh(\overline{k}(\alpha))$, so $(hk)|_\alpha = h|_\zeta \circ k|_\alpha$ by \cref{lem:restrict_composition}.

Let $\beta'=\og(\alpha)$ and $\eta'=\oh(\zeta)$, so $A_{\beta'}$ is the smallest atom in $A_\beta$ that contains $gA_\alpha$, and $A_{\eta'}$ is the smallest atom in $A_\eta$ that contains $hA_\zeta$.  Since $\ell$ is a morphism, it maps the tree of descendants of $A_\beta$ isomorphically to the tree of descendants of $A_\zeta$, and since $\ell g A_\alpha = hk A_\alpha = h A_\zeta$, it follows that $\ell A_{\beta'}=A_{\eta'}$.  Indeed, $\ell$ must be a morphism from $\beta'$ to $\eta'$.  Now we compute 
\[
\overline{\ell g}(\alpha) = \overline{hk}(\alpha) = 
\oh(\zeta) =
\eta' =
\overline{\ell}(\beta') =
\overline{\ell}(\og(\alpha)),
\]
so $(\ell g)|_\alpha = \ell|_{\beta'}\circ g|_\alpha$ by \cref{lem:restrict_composition}.  Since $hk=\ell g$, we conclude that  $h|_\zeta \circ k|_\alpha = \ell|_{\beta'}\circ g|_\alpha$. But $k|_\alpha \in \mathrm{Mor}(t(\alpha),G)$ and $\ell|_{\beta'}\in \mathrm{Mor}(t(\beta'),G)$, so $g|_\alpha$ and $h|_\zeta$ are morphism-equivalent.
\end{proof}

Now we are poised to prove \cref{thrm:contracting} that the image $G^\varphi$ of $G$ in $\R_{\Gamma,\C_r}$ has finite nucleus.

\begin{proof}[Proof of \cref{thrm:contracting}]
The proof closely follows the proof of Theorem~3.10 given in Subsection~3.6 of \cite{BelkBleakMatucci}. That proof established that each $g\in G$ individually has only finitely many local actions, and here we need to prove that all the $g\in G$ collectively have a finite set of local actions containing all but finitely many of each element's local actions.

Define the \newword{signature} of a mapping triple $(g,A_\alpha,A_\beta)$, to be the following information:
\begin{enumerate}
\item The sets $g\,\widehat{N}(A_\alpha)$ and $\widehat{N}(A_\beta)$,\smallskip
\item The functions $g\,\od_{A_\alpha}$ on~$g\,\widehat{N}(A_\alpha)$ and $\od_{A_\beta}$ on~$\widehat{N}(A_\beta)$, and\smallskip
\item The set $g\,C(g^{-1}p)$ for each $p\in g\,\widehat{N}(A_\alpha)$, and the cone $C(q)$ for each $q\in \widehat{N}(A_\beta)$.
\end{enumerate}

We say that two mapping triples $(g,A_\alpha,A_\beta)$ and $(h,A_\zeta,A_\eta)$ have \newword{equivalent signatures} if there exists $\ell\in G$ so that
\begin{enumerate}
\item
$\ell$ maps $g\,\widehat{N}(A_\alpha)$ to $h\,\widehat{N}(A_\zeta)$ and $\widehat{N}(A_\beta)$ to $\widehat{N}(A_\eta)$,\smallskip
\item $\ell g\,\od_{A_\alpha}$ agrees with $h\,\od_{A_\zeta}$ on~$h\,\widehat{N}(A_\zeta)$, and $\ell\,\od_{A_\beta}$ agrees with $\od_{A_\eta}$ on~$\widehat{N}(A_\eta)$, and\smallskip
\item $\ell g\,C(g^{-1}p) = h\,C(h^{-1}\ell p)$ for all $p\in g\,\widehat{N}(A_\alpha)$, and $\ell\,C(q) = C(\ell q)$ for all $q\in \widehat{N}(A_\beta)$.
\end{enumerate}
It follows immediately from \cref{prop:MakeMorphisms} that if two mapping triples $(g,A_\alpha,A_\beta)$ and $(h,A_\zeta,A_\eta)$ have equivalent signatures via an element $\ell\in G$, then $\ell$ is a morphism from $A_\beta$ to $A_\eta$ and $k\coloneqq h^{-1}\ell g$ is a morphism from $A_\alpha$ to~$A_\zeta$, so the two mapping triples are equivalent.  Then by \cref{lem:mapping_triples_equiv}, it follows that the local actions $g|_\alpha$ and $h|_\zeta$ are morphism-equivalent. To summarize, equivalent signatures implies equivalent mapping triples implies equivalent local actions.

As observed previously, the contracting lemma (\cref{lem:contractinglemma}) tells us that, for any $g\in G$ and any $A_\alpha\in \A_n(G)$ with $n>2|g|+39\delta+13$, there exists $A_\beta\in\A(G)$ such that $(g,A_\alpha,A_\beta)$ is a mapping triple. In particular, for every $p\in \Nuc_G$ there exists a mapping triple $(g,A_\alpha,A_\beta)$ such that $g|_\alpha = p$.  Therefore, to prove that $\Nuc_G$ is finite, it suffices to prove that there are only finitely many equivalence classes of signatures.

Consider a mapping triple $(g,A_\alpha,A_\beta)$. By \cref{prop:PropertiesSets}(\ref{part:DiameterN}), the sets $N(A_\alpha)$ and $N(A_\beta)$ have diameter at most $2\delta$, and the same holds for $g\,N(A_\alpha)$.  By the definition of a mapping triple, $N(A_\beta)$ is contained in the $(18\delta+6)$-neighborhood of $g\,N(A_\alpha)$, so it follows that $N(A_\beta)\cup g\,N(A_\alpha)$ has diameter at most $(18\delta+6)+2(2\delta) = 22\delta+6$.  Then $\widehat{N}(A_\beta)\cup g\,\widehat{N}(A_\alpha)$ has diameter at most $(22\delta+6)+2(4\delta+2)=30\delta +10$.  In particular, up to the action of $G$ there are only finitely many $G$\mbox{-}orbits of pairs $\bigl(g\,\widehat{N}(A_\alpha),\widehat{N}(A_\beta)\bigr)$. Since a hyperbolic group has only finitely many cone types~\cite{Cannon1}, every such $G$\mbox{-}orbit corresponds to finitely many equivalence classes of signatures, so we conclude that there are only finitely many equivalence classes of signatures.
\end{proof}

Finally, we can prove the main result of this section, that every hyperbolic group embeds into a full, contracting RSG.

\begin{proof}[Proof of \cref{thrm:hyp_to_contracting}]
Let $G$ be a non-trivial hyperbolic group. Since $G$ embeds into the hyperbolic group $G*\Z$, we can assume without loss of generality that our group $G$ has a proper $\Z$ free factor. By \cref{prop:fin_many_types}, the tree of atoms of any hyperbolic group has finitely many types, so by \cref{thrm:FreeProductBoundary} we conclude that $\partial_h(G)$ has no isolated points, the action of $G$ on $\partial_h(G)$ is faithful, and the associated type graph has an irreducible core. By \cref{thrm:BBM}, the action is also rational, i.e.\ there exists a subshift of finite type $\Sigma_\Gamma$ and a node $r$ of $\Gamma$ such that, identifying $\partial_h G$ with $\C_r$, we have that $G$ embeds as a subgroup of $\R_{\Gamma,\partial_h G}$. By \cref{prop:hyp_similarities} this embedded copy of $G$ is an RSG. It is also contracting, by \cref{thrm:contracting}. Now $G$ embeds into $[[\,G\mid \partial_h G\,]]$, and this is a full, contracting RSG by \cref{prop:NucleusHasProperties} and \cref{thrm:RSGCharacterization}.
\end{proof}

\section{Boone--Higman embeddings}\label{sec:ending}

Now we are poised to prove our main result, \cref{thrm:bh_hyp} from the introduction, that hyperbolic groups satisfy the Boone--Higman conjecture. The last remaining key step is the following.

\begin{proposition}\label{prop:contr_to_simple}
Every full, contracting RSG embeds into a finitely presented simple group.
\end{proposition}

Outside the application to hyperbolic groups, \cref{prop:contr_to_simple} also yields the following, which is immediate from the fact that R\"over--Nekrashevych groups are full RSGs (see \cref{ex:rn}).

\begin{corollary}\label{cor:RN_BH}
Every contracting R\"over--Nekrashevych group embeds into a finitely presented simple group. Hence, all contracting self-similar groups and all contracting R\"over--Nekrashevych groups satisfy the Boone--Higman conjecture.\qed
\end{corollary}

Before embarking on the proof of \cref{prop:contr_to_simple}, let us explain where our finitely presented simple groups come from. (First we should mention that full, contracting RSGs are not always themselves simple, even up to finite index; for example, one can check that $[[\,F_2\mid \partial_h F_2\,]]$ has abelianization $\Z^2\oplus (\Z/2\Z)^2$.)  The source of our examples is so called twisted Brin--Thompson groups, introduced by the first and fourth authors in \cite{BelkZaremskyTwisted}. Given a group $G$ acting faithfully on a set $S$, consider the Cantor space $\C^S$, where $\C$ is the usual Cantor set $\{0,1\}^\N$, and $\C^S$ is given the product topology. Thompson's group $V=V_2$ acts on~$\C$, so there is an action of the (restricted, permutational) wreath product $V \wr_S G$ on $\C^S$, with each copy of $V$ acting on the appropriate coordinate, and $G$ permuting the coordinates. Now the \newword{twisted Brin--Thompson group} $SV_G$ is the full group induced by $V\wr_S G$, i.e.
\[
SV_G \coloneqq [[\,V\wr_S G \mid \C^S\,]].
\]
The group $SV_G$ is always simple, but it may or may not have good finiteness properties. We should mention that twisted Brin--Thompson groups admit no isometric action with loxodromic elements on any hyperbolic metric space, by a result of Balasubramanya--Fournier-Facio--Genevois \cite{BFFG}, so it is interesting that in some sense our embedding of hyperbolic groups into finitely presented simple groups serves to completely eliminate this feature of hyperbolicity.

The following, from \cite{zaremskyTaste}, is an improved version of a result from~\cite{BelkZaremskyTwisted}, which essentially says that groups admitting certain actions satisfy the Boone--Higman conjecture. Following Cameron \cite{Cameron}, a group of permutations of a set $S$ is \newword{oligomorphic} if for every $n\ge 1$ the induced action on $S^n$ has finitely many orbits.

\begin{theorem}\label{thrm:action_to_simple}\cite[Theorem~4.2]{zaremskyTaste}
Let $G$ be a finitely presented, oligomorphic group of permutations of a set $S$ such that the stabilizer in $G$ of any finite subset of $S$ is finitely generated. Then the twisted Brin--Thompson group $SV_G$ is a finitely presented simple group into which $G$ embeds.
\end{theorem}

Our goal is to verify the conditions of this theorem for full, contracting RSGs. We begin by observing that full groups often have oligomorphic actions. For the following proposition, an action of a group $G$ on a set $S$ is \newword{highly transitive} if $G$ acts transitively on $n$-tuples of distinct elements of $S$ for all $n\geq 1$.

\begin{proposition}\label{prop:HighlyTransitiveAction}
Let $X$ be a Hausdorff space that has a basis of clopen sets, and let $G$ be any full group of homeomorphisms of~$X$.  Then $G$ acts highly transitively (and hence oligomorphically) on each of its orbits.
\end{proposition}
\begin{proof}
Let $S$ be a $G$-orbit in $X$, and let $S'=\{s_1,\ldots,s_n\}$ be a finite subset of~$S$.   Choose $g_1,\ldots,g_n\in G$ so that $g_i(s_1)=s_i$ for each~$i$, where $g_1$ is the identity.  Since $X$ is Hausdorff and has a basis of clopen sets, we can choose a clopen neighborhoood $E$ of $s_1$ so that the clopen sets $g_1(E),\ldots,g_n(E)$ are pairwise disjoint.  For each $i\ne j$, let $g_{ij}$ be the homeomorphism that maps $g_i(E)$ to $g_j(E)$ by $g_jg_i^{-1}$, maps $g_j(E)$ to $g_i(E)$ by $g_ig_j^{-1}$, and is the identity elsewhere.  Since $G$ is full, each $g_{ij}$ lies in $G$, and these generate a subgroup of $G$ that can permute the elements of $S'$ in any desired fashion.  It follows easily that the action is highly transitive.
\end{proof}

Since every compact, totally disconnected metrizable space has a basis of clopen sets, it follows from \cref{prop:HighlyTransitiveAction} that every full RSG acts highly transitively on each of its orbits.

Our next goal is to explore point stabilizers. We say that a point $\omega$ in a subshift $\Sigma_{\Gamma}$ is \newword{rational} if it is eventually repeating, i.e.\ if there exist finite words $\sigma$ and $\tau$ such that $\omega=\sigma\cdot\tau\cdot\tau\cdot\cdots$, written $\omega=\sigma\cdot \tau^\infty$ (so $\sigma$ is some prefix and $\tau$ repeats forever). It is easy to see that the set of rational points in any clopen set $E\subseteq \Sigma_\Gamma$ is stabilized by the action of the rational group~$\R_{\Gamma,E}$.

If $G\leq \Homeo(X)$ is a group of homeomorphisms of a space $X$, the stabilizer $\mathrm{Stab}_G(x)$ of a point $x\in X$ has a normal subgroup $\mathrm{Fix}_G^0(x)$ consisting of all elements that are the identity in some neighborhood of~$x$.  The quotient
\[
[G]_x \coloneqq \mathrm{Stab}_G(x)\bigr/\mathrm{Fix}_G^0(x)
\]
is called the \newword{group of germs} of $G$ at $x$.  If $g\in \mathrm{Stab}_G(x)$, its image $[g]_x$ in $[G]_x$ is the \newword{germ} of $g$ at~$x$.

\begin{proposition}\label{prop:CyclicStabilizers}
Let $G\leq \R_{\Gamma,E}$ be an RSG with a finite nucleus, and let $\omega\in E$ be a rational point.  Then the group of germs\/ $[G]_\omega$ is virtually infinite cyclic.
\end{proposition}
\begin{proof}
Recall that $\omega=\sigma\cdot\tau^\infty$, and let us assume $\tau$ is chosen to not be a power of a shorter word. Clearly $t(\sigma)=t(\sigma\cdot \tau)$.  Since $G$ is an RSG,  there exists $f\in G$ that takes $\C_\sigma$ to $\C_{\sigma\cdot\tau}$ by the canonical similarity. Note that $f$ fixes $\omega$. Now we claim that the cyclic group $\langle [f]_\omega \rangle$ has finite index in~$[G]_\omega$. 

Consider the function $\Nuc\to\Nuc$ that sends $p$ to $p|_\tau$. Since $\Nuc$ is finite, every element of $\Nuc$ is either periodic or pre-periodic under this function. Let $\Nuc_\tau$ denote the set of periodic points, and let $M$ be the least common multiple of their periods. Now take an arbitrary $h\in \Stab_G(\omega)$. For all sufficiently large $i$, the local action of $h$ at $\sigma\cdot\tau^i$ lies in $\Nuc$. By \cref{lem:restrict_twice}, in fact for all sufficiently large $i$, the local action of $h$ at $\sigma\cdot\tau^i$ lies in~$\Nuc_\tau$. This implies that the sequence $i\mapsto h|_{\sigma\cdot\tau^{iM}}$ is eventually constant. We therefore get a well defined function $\lambda\colon \Stab_G(\omega)\to \Nuc_\tau$ sending $h$ to the eventually constant value of this sequence. Note that if $h$ and $h'$ agree in an open neighborhood of $\omega$ then they have the same $\lambda$ value, so $\lambda$ induces a well defined function, which we will also call $\lambda$, from $[G]_\omega$ to $\Nuc_\tau$.

We now claim that every fiber of $\lambda$ lies in a right coset of $\langle [f]_\omega\rangle$. Since $\Nuc_\tau$ is finite, this will prove that $\langle [f]_\omega\rangle$ has finite index in $[G]_\omega$ as desired. Let $p\in\Nuc_\tau$, and let $h,h'\in \Stab_G(\omega)$ with $\lambda(h)=\lambda(h')=p$. Then we can fix a sufficiently large $i$ so that $h|_{\sigma\cdot\tau^{iM}}=h'|_{\sigma\cdot\tau^{iM}}=p$. Let $\zeta\coloneqq\oh(\sigma\cdot\tau^{iM})$ and $\zeta'\coloneqq\overline{h'}(\sigma\cdot\tau^{iM})$. Thus we have $\omega=h(\omega)=h(\sigma\cdot\tau^{iM}\cdot\tau^\infty)=\zeta\cdot p(\tau^\infty)$ and $\omega=h'(\omega)=h'(\sigma\cdot\tau^{iM}\cdot\tau^\infty)=\zeta'\cdot p(\tau^\infty)$, so we conclude that $\omega=\zeta\cdot p(\tau^\infty)=\zeta'\cdot p(\tau^\infty)$. Since $\omega=\sigma\cdot\tau^\infty$ and $\zeta$ and $\zeta'$ are both prefixes of $\omega$, up to increasing $i$ as needed to ensure that $\sigma$ is a prefix of both $\zeta$ and $\zeta'$, this means that $\zeta=\sigma\cdot\tau^j\cdot \alpha$ and $\zeta'=\sigma\cdot\tau^{j'}\cdot \alpha'$ for some $j,j'\ge 0$ and some $\alpha,\alpha'$ proper prefixes of $\tau$. Writing $\tau=\alpha\cdot\beta=\alpha'\cdot\beta'$, so $\beta$ and $\beta'$ are nonempty suffixes of $\tau$, we see that $p(\tau^\infty)=\beta\cdot\tau^\infty=\beta'\cdot\tau^\infty$, and since $\tau$ is not a proper power this implies $\beta=\beta'$; this then implies that $\alpha=\alpha'$. Without loss of generality, say $j\le j'$. Then $h'$ agrees with $f^{j'-j}h$ on the cone $\C_{\sigma\cdot\tau^{iM}}$, so $[h]_\omega$ and $[h']_\omega$ differ by left multiplication by an element of $\langle [f]_\omega \rangle$, and we are done.
\end{proof}

\begin{remark}
If $G\leq \R_{\Gamma,E}$ is an RSG with finite nucleus, then a similar argument shows that the group of germs $[G]_\omega$ at any irrational point $\omega\in E$ is finite.
\end{remark}

\begin{remark}
If $G$ is a hyperbolic group, it is well-known that the stabilizer of any point in the Gromov boundary $\partial G$ is virtually cyclic, and it follows that the same holds for stabilizers in $\partial_h G$. By \cref{prop:CyclicStabilizers}, this was a necessary condition for $G$ to have finite nucleus.  Indeed, given a group $G$ acting on a Cantor space~$X$, if there exist points in $X$ whose stabilizers are not virtually cyclic, then by \cref{prop:CyclicStabilizers} there is no way to assign addresses to points in $X$ so that the corresponding action is rational with finite nucleus.
\end{remark}

The argument in the following proposition follows the arguments given by the first author, James Hyde, and the third author in~\cite{BelkHydeMatucciStabs}, as well as some unpublished simplifications partially due to James Hyde.

\begin{proposition}\label{prop:fin_gen_stabs}
Let $G\leq \R_{\Gamma,E}$ be a full, contracting RSG, and let $S'$ be a finite set of rational points in~$E$.  Then the stabilizer\/ $\mathrm{Stab}_G(S')$ is finitely generated.
\end{proposition}
\begin{proof}
Let $\Fix_G(S')$ be the subgroup of $\Stab_G(S')$ consisting of elements that fix $S'$ pointwise.  Since $\Fix_G(S')$ has finite index in $\Stab_G(S')$, it suffices to prove that $\Fix_G(S')$ is finitely generated.  This group fits into an exact sequence
\[
\Fix_G^0(S') \hookrightarrow \Fix_G(S') \to \prod_{\omega\in S'} [G]_\omega
\]
where $\Fix_G^0(S')$ is the group of elements that are the identity in some neighborhood of $S'$.  Each $[G]_\omega$ is virtually cyclic by \cref{prop:CyclicStabilizers}, so the product $\prod_{\omega\in S'} [G]_\omega$ is virtually free abelian, and hence the image of $\Fix_G(S')$ must be finitely generated.  Therefore, it suffices to prove that $\Fix_G^0(S')$ is contained in a finitely generated subgroup of $\Fix_G(S')$.
 
Let $S'=\{\omega_1,\ldots,\omega_n\}$, and write each $\omega_i$ as $\sigma_i\cdot \tau_i^\infty$, where the prefixes $\sigma_i$ are long enough so that the cones $\C_{\sigma_1},\ldots,\C_{\sigma_n}$ are pairwise disjoint subsets of $E$ whose union $U=\bigcup_{i=1}^n \C_{\sigma_i}$ is not all of~$E$.  The cones $\C_{\sigma_i\cdot\tau_i}$ are also pairwise disjoint and have proper union, call it $U'$, and clearly $t(\sigma_i\cdot\tau_i)=t(\sigma_i)$ for all $i$. Since $\Sigma_\Gamma$ has an irreducible core, by \cref{cor:MappingConesWithV} there exists $f\in V_{\Gamma,E}$ such that $f$ acts on each $\C_{\sigma_i}$ by the canonical similarity $\C_{\sigma_i}\mapsto \C_{\sigma_i\cdot\tau_i}$, for all $1\le i\le n$. In particular $f$ fixes $S'$. Also note that $U\supset f(U) \supset f^2(U)\supset\cdots$, and any open neighborhood of $S'$ contains $f^k(U)$ for all sufficiently large $k$. This implies that $\Fix_G^0(S')$ equals the union of the conjugates $f^i\, \Fix_G(U) f^{-i}$ for $i\ge 0$. Now it suffices to show that $\Fix_G(U)$ is finitely generated, since then $\Fix_G^0(S')$ will be contained in the finitely generated group generated by $\Fix_G(U)$ and $f$.  But $\Fix_G(U)$ is isomorphic to the full RSG in $\R_{\Gamma,E\setminus U}$ with the same nucleus as~$G$, as described in \cref{thrm:RSGCharacterization}. In particular $\Fix_G(U)$ is finitely generated by \cref{thrm:fin_pres}, so it follows that $\Stab_G(S')$ is finitely generated.
\end{proof}

Now we can prove the main result of this section.

\begin{proof}[Proof of \cref{prop:contr_to_simple}]
Let $G\le \R_{\Gamma,E}$ be a full, contracting RSG. Since $\Sigma_\Gamma$ has an irreducible core, we know that $G$ is finitely presented by \cref{thrm:fin_pres}. Let $S$ be the $G$-orbit of a rational point $\omega$ in $E$. Note that $S$ is dense in $E$, so $G$ acts faithfully on $S$. This action is oligomorphic by \cref{prop:HighlyTransitiveAction}, and the stabilizer in $G$ of any finite subset of $S$ is finitely generated by \cref{prop:fin_gen_stabs}. By \cref{thrm:action_to_simple}, the twisted Brin--Thompson group $SV_G$ is finitely presented, and it is simple, as all twisted Brin--Thompson groups are. Since $G$ embeds in $SV_G$, we are done.
\end{proof}

Finally we can prove \cref{thrm:bh_hyp}, that hyperbolic groups embed in finitely presented simple groups, and thus satisfy the Boone--Higman conjecture.

\begin{proof}[Proof of \cref{thrm:bh_hyp}]
Let $G$ be a hyperbolic group. By \cref{thrm:hyp_to_contracting} there exists a full, contracting RSG into which $G$ embeds as a subgroup. This embeds in a finitely presented simple group by \cref{prop:contr_to_simple}.
\end{proof}

\bibliography{bibHypEmbedding}
\newcommand{\etalchar}[1]{$^{#1}$}

\end{document}